\documentclass[11pt,amstex]{article}  
\usepackage{color}              
\usepackage{graphicx}
\usepackage[]{amsmath}
\usepackage{amssymb}
\usepackage{hyperref}
\usepackage{latexsym, enumerate, amsfonts, amsthm, bbm}  

\newtheorem{lemma}{Lemma}[section]
\newtheorem{proposition}[lemma]{Proposition}
\newtheorem{theorem}[lemma]{Theorem}

\newtheorem{remark}[lemma]{Remark}

\newtheorem{corollary}[lemma]{Corollary}
\newtheorem{definition}[lemma]{Definition}
\newtheorem{notation}[lemma]{Notation}

\newcommand{\EE}{\mathbb E}
\newcommand{\NN}{\mathbb N}
\newcommand{\PP}{\mathbb P}

\newcommand{\RR}{\mathbb R}
\newcommand{\ZZ}{\mathbb Z}

\newcommand{\SC}{\mathcal C}
\newcommand{\SD}{\mathcal D}
\newcommand{\SF}{\mathcal F}

\newcommand{\SL}{\mathcal L}

\newcommand{\CP}{\mathcal P}

\newcommand{\SU}{\mathcal U}

\newcommand{\1}{\mathbbm{1}}

\newcommand{\normx}[2]{\parallel\! #1 \!\parallel_{ #2 }}

\newcommand{\fatcdot}{\boldsymbol{\cdot}}
\newcommand{\bigmid}{\bigm|}
\newcommand{\Bigmid}{\Bigm|}
\newcommand{\biggmid}{\biggm|}
\newcommand{\Biggmid}{\Biggm|}

\newcommand{\mittl}{\mathcal H}
\newcommand{\deltap}{\delta'}
\newcommand{\deltapp}{\delta''}

\newcommand{\eqn}[1]{\begin{equation} #1 \end{equation}}
\newcommand{\eqan}[1]{\begin{align} #1 \end{align}}
\newcommand{\lbeq}[1]{\label{#1}}
\newcommand{\nn}{\nonumber}

\numberwithin{equation}{section}

\begin{document}

\author{
Sandra Kliem
\thanks{
Institut f\"ur Mathematik, Goethe Universit\"at, Frankfurt am Main, Germany.
E-mail: {\tt kliem@math.uni-frankfurt.de}}}

\title{Right marker speeds of solutions to the KPP equation with noise}

\maketitle

\begin{abstract}
\noindent
We consider the one-dimensional KPP-equation driven by space-time white noise. We show that for all parameters above the critical value for survival, there exist stochastic wavelike solutions which travel with a deterministic positive linear speed. We further give a sufficient condition on the initial condition of a solution to attain this speed. Our approach is in the spirit of corresponding results for the nearest-neighbor contact process respectively oriented percolation. Here, the main difficulty arises from the moderate size of the parameter and the long range interaction. Stopping times and averaging techniques are used to overcome this difficulty.
\end{abstract}

\vspace{0.3in}


\noindent
{\bf MSC2000 subject classification.} {Primary
60H15; 
Secondary
35R60. 
} 






%
\section{Introduction}
%
%
The Kolmogorov-Petrovskii-Piskunov-(KPP)-equation (also known as the Kolmogorov- or Fisher-equation) with noise is given as
\eqn{ 
\label{equ:SPDE}
  \partial_t u = \partial_{xx} u + \theta u - u^2 + u^{\frac{1}{2}} dW, \qquad t>0, x \in \RR, \qquad u(0,x) = u_0(x) \geq 0, 
}
where $W=W(t,x)$ is space-time white noise and $\theta > 0$ a parameter. The deterministic part of this one-dimensional stochastic partial differential equation (SPDE) is, after appropriate scaling, a case of the well-studied KPP-equation. Note that by Mueller and Tribe~\cite[Lemma~2.1.2]{MT1994}, constant front-factors in the PDEs/SPDEs to be referred to, can and will be changed without comment to fit into our framework. Including the noise term, one can think of $u(t,x)=u_t(x)=u_t^{(u_0)}(x)$ as the (random) density of a population in time and space. Leaving out the term $\theta u-  u^2$, the above SPDE is the density of a super-Brownian motion (cf. Perkins~\cite[Theorem~III.4.2]{bP2002}), the latter being the high density limit of branching particle systems. The additional term of $\theta u$ models linear mass creation at rate $\theta > 0$, $-u^2$ models death due to competition respectively overcrowding. In \cite{MT1995}, Mueller and Tribe obtain solutions to \eqref{equ:SPDE} as (weak) limits of \textit{approximate} densities of occupied sites in rescaled one-dimensional long range contact processes.

Let $\SC^+$ denote the space of non-negative continuous functions on $\RR$. The existence and uniqueness in law of solutions to \eqref{equ:SPDE} in the \textit{space of non-negative continuous functions with slower than exponential growth} $\SC_{tem}^+$, 
\eqn{
\lbeq{equ:Ctem}
  \SC_{tem}^+ = \left\{ f \in \SC^+: \normx{f}{\lambda} < \infty \mbox{ for all } \lambda>0 \right\} \mbox{ with } \normx{f}{\lambda} = \sup_{x \in \RR} |f(x)| e^{-\lambda |x|},
}
is established in Tribe~\cite[Theorem~2.2]{T1996}. Here, a solution to \eqref{equ:SPDE} is to be understood in the sense of a \textit{weak solution} (see Notation~\ref{NOTA:tribe} below). Denote with $\PP_{u_0}$ the law of such a solution starting in $u_0 \in \SC_{tem}^+$. By \cite[Theorem~2.2]{T1996}, the map $f \mapsto \PP_f$ on $\SC_{tem}^+$ is continuous and the family of laws $\PP_f, f \in \SC_{tem}^+$ forms a strong Markov family. For $\nu \in \CP(\SC_{tem}^+)$, the space of probability measures on $\SC_{tem}^+$, denote $\PP_\nu(A) = \int_{\SC_{tem}^+} \PP_f(A) \nu(df)$. Use $\EE_{u_0}$ respectively $\EE_{\nu}$ to denote respective expectations.

Let $\tau = \inf\{ t \geq 0: u(t,\fatcdot) \equiv 0\}$ be the \textit{extinction-time} of the process. By \cite[Theorem~1]{MT1994}, there exists a critical value $\theta_c>0$ such that for any initial condition $u_0 \in \SC_c^+ \backslash \{0\}$ with \textit{compact} support and $\theta < \theta_c$, the extinction-time of $u$ solving \eqref{equ:SPDE} is finite almost surely. For $\theta>\theta_c$, \textit{survival}, that is $\tau=\infty$, happens with positive probability. 

The investigation of the dynamics of solutions to \eqref{equ:SPDE} is a major challenge, where the main difficulty comes from the competition term $-u^2$. Without competition, the underlying additive property facilitates the use of Laplace functionals. Including competition, only subadditivity in the sense of \cite[Lemma~2.1.7]{MT1994} respectively Kliem \cite[Remark~2.1(i)]{K2017} holds, that is, for $u_0, v_0 \in \SC_{tem}^+$ and $w_0\equiv u_0+v_0$ there exists a coupling of solutions $(u_t)_{t \geq 0}, (v_t)_{t \geq 0}, (w_t)_{t \geq 0}$ to \eqref{equ:SPDE} with respective initial conditions $u_0, v_0, w_0$ such that $w_t(x) \leq u_t(x)+v_t(x)$ for all $t \geq 0, x \in \RR$ almost surely. 

Write $\langle f,g \rangle = \int f(x) g(x) dx$. For the process in \eqref{equ:SPDE} one has a self-duality relationship in the form
\eqn{
\label{equ:SPDE-self-duality}
  \EE_{u_0}\!\left[ e^{-2 <u(t),v_0>} \right] = \EE_{u_0}\otimes\EE_{v_0}\!\left[ e^{-2 \langle u(s),v(t-s) \rangle} \right] = \EE_{v_0}\!\left[ e^{-2 <u_0,v(t)>} \right] 
}
for all $0 \leq s \leq t$ and $u_0, v_0 \in \SC_{tem}^+$, where $u(t), v(t)$ are independent solutions to \eqref{equ:SPDE} with initial condition $u_0$ respectively $v_0$ (cf. \cite[(2.1)]{K2017}). Use $\CP(E)$ to denote the space of probability measures on $E$. In \cite[Remark~2.5]{K2017} this self-duality is used to prove existence of a unique \textit{upper invariant distribution} $\mu \in \CP(\SC_{tem}^+)$ satisfying 
\eqn{
\label{equ:KHT}
  \lim_{t \rightarrow \infty} \inf\!\left\{ \EE_\psi\!\left[ e^{-2 < u(T+t) , \phi > } \right]; \psi \in \SC_{tem}^+ \right\} = \int e^{-2<f,\phi>} \mu(df) = \PP_\phi(\tau<\infty)
}
for all $T>0, \phi \in \SC_c^+$. In \cite[Theorem~1]{HT2004}, Horridge and Tribe give sufficient conditions (``uniformly distributed in space'') for initial conditions to be in the domain of attraction of $\mu$. They characterize $\mu$ by the right hand side of (\ref{equ:KHT}) and show that it is the unique translation invariant stationary distribution satisfying $\mu(\{f: f \not\equiv 0\}) = 1$. The result and method of proof are in the spirit of Harris' convergence theorem for additive particle systems (cf. Durrett \cite[Theorem~3.3]{bD1995}).

Recall the construction of solutions to \eqref{equ:SPDE} from \cite{MT1995} by means of limits of densities of rescaled \textit{long range} contact processes. When investigating solutions to the SPDE \eqref{equ:SPDE}, it is only natural to anticipate and/or investigate behavior similar in spirit to the approximating systems. Indeed, \cite{HT2004} successfully applied the method of proof of Harris' convergence theorem for additive particle systems to prove a corresponding result in the context of SPDEs \eqref{equ:SPDE}. Due to the long range interaction and the lack of a dual process, results for long range contact processes are limited. More is known for the \textit{nearest-neighbor} contact process $(\xi_t)_{t \geq 0}$ on $\ZZ$ (cf. Griffeath~\cite{G1981}), where the neighborhood of a site $x \in \ZZ$ is restricted to $\{x-1,x+1\}$. For the nearest-neighbor contact process a full description of the limiting law of a solution is available. The limiting law is the weighted average of the Dirac-measure on the ``all-unoccupied" configuration and the upper invariant measure of the process, $\nu$, where the weight on the former coincides with the extinction probability (see \cite[Theorem~5]{G1981}). 

In what follows, let $S$ be the space of all subsets of $\ZZ$. By identifying the state of the process $\xi_t$ at time $t$ with the set of occupied sites, we can consider $(\xi_t)_{t \geq 0}$ as an $S$-valued process. Let $\lambda$ be the birth-parameter, the death-parameter is set to one. Set $\lambda_c = \sup\{ \lambda \geq 0: \PP( \tau^{\{0\}}=\infty )=0 \}$, where $\tau^{\{0\}} = \inf\{ t \geq 0: \xi_t^{\{0\}} = \emptyset \}$ is the extinction time of the population starting with zero being the only occupied site at time $0$. 

The proof of complete convergence for the \textit{nearest-neighbor} case relies in essence on the progression of the so-called edge processes $l_t^A \equiv \min\{ x: x \in \xi_t^A \}, r_t^A \equiv \max\{ x: x \in \xi_t^A \}$, $A \in S$ fixed. Due to the nearest-neighbor interaction one can easily show that 
\eqn{
\label{equ:contact-prop1}
  \xi_t^{\{0\}} = \xi_t^A \cap  [l_t^{\{0\}} , r_t^{\{0\}}] = \xi_t^{(-\infty,0] \cap \ZZ} \cap \xi_t^{[0,\infty) \cap \ZZ} \mbox{  for all } 0 \in A \subset \ZZ \mbox{ on } \{\tau^{\{0\}}>t\}
}
(cf. \cite[Theorem~3]{G1981}). Moreover, 
\eqn{
\label{equ:contact-prop2}
  l_t^{\{0\}} = l_t^{[0,\infty) \cap \ZZ} \mbox{ and } r_t^{\{0\}} = r_t^{(-\infty,0] \cap \ZZ} \quad \mbox{ on } \{\tau^{\{0\}}>t\}.
}
In \cite[Theorem~1.4 and Section~4]{D1980} respectively \cite[Section~3, (8)--(9)]{D1984}, Durrett shows for the \textit{nearest-neigbor} contact process respectively for oriented percolation in two dimensions that
\eqn{
\lbeq{equ:asconv-markers}
  -\lim_{t \rightarrow \infty} \frac{l_t^{\{0\}}}{t} = \lim_{t \rightarrow \infty} \frac{r_t^{\{0\}}}{t} = \alpha \mbox{ a.s., } \quad \mbox{ where } 
  \alpha \begin{cases} >0, & \mbox{ if } \tau^{\{0\}}=\infty, \cr <0, & \mbox{ if } \tau^{\{0\}}<\infty. \end{cases}
}
In these models, edge speeds characterize critical values. Similar features were for instance recently observed in Bessonov and Durrett~\cite{BD2017} for planar quadratic contact processes (here, two individuals are needed to produce a new one). Under long range interaction, \eqref{equ:contact-prop1}--\eqref{equ:contact-prop2} do not hold true any longer. 

For these reasons, the study of the speed of the right (and thus by symmetry left) marker
\eqn{
\label{equ:def-R0}
  R_0(u(t)) \equiv R_0(t) \equiv \sup\{ x \in \RR: u(t,x)>0 \} \qquad \mbox{ with } \sup\emptyset = -\infty
}
of a solution to \eqref{equ:SPDE} starting in $u_0 \in \CP(\SC_{tem}^+)$ with $\PP(R_0(u_0)<\infty)=1$ is of independent interest and yields new insights into the dynamics of solutions to \eqref{equ:SPDE}. Note that $R_0(t)=-\infty$ if and only if $\tau \leq t$. Extending arguments of Iscoe~\cite{I1988} one can show that $R_0(u(0))<\infty$ implies $R_0(u(t))<\infty$ for all $t>0$. In \cite[Remark~2.8]{K2017}, $\SC_{tem}^+$-valued left- and right-upper measures were derived as analogues to the law of $\xi_t^{(-\infty,0] \cap \ZZ}, \xi_t^{[0,\infty) \cap \ZZ}, t>0$ and first rough estimates on marker-speeds obtained in Section~4.

Let $R_0(t)$ as in \eqref{equ:def-R0}. Using $R_0$ as a \textit{(right) wavefront marker}, we look for so-called \textit{travelling wave solutions} to \eqref{equ:SPDE}, that is solutions with the properties
\eqan{
  & (i) \qquad R_0(u(t)) \in (-\infty,\infty) \mbox{ for all } t \geq 0, \\
  & (ii) \quad\ \ u(t,\fatcdot + R_0(u(t))) \mbox{ is a stationary process in time}. \nn
}
Travelling wave solutions are of interest in models from physics, chemistry and biology (cf. Aronson and Weinberger~\cite{AW1975}). In \cite{T1996}, the existence of travelling wave solutions for $\theta>\theta_c$ with non-negative wave speed, based on solutions to \eqref{equ:SPDE} with Heavyside initial data of the form $H_0(x) \equiv 1 \wedge (-x \vee 0)$ is established. In \cite[Section~4]{T1996} it is established that for $\theta>\theta_c$ any travelling wave solution has an asymptotic (possibly random) \textit{wave speed}
\eqn{
  R_0(u(t))/t \rightarrow A \in \left[ 0,2\theta^{1/2} \right] \mbox{ for } t \rightarrow \infty \mbox{ almost surely.}
}
It is further shown that for $\theta$ big enough, $A$ is close to $2\theta^{1/2}$ with high probability. Strict positivity of $A$ remains an open problem if $\theta$ is of moderate size. Further open problems that arise are for instance if the wave-speed is deterministic or random, the dependence of the speed on the parameter $\theta$, the uniqueness of the distribution of the travelling waves and the shape of the wavefront. In this article, we make substantial progress to resolve the first two problems.

An alternative construction of travelling wave solutions is given in \cite{K2017} in case $\theta>\theta_c$. The initial Heavyside-condition $H_0$ is replaced by an arbitrary non-negative continuous function $g_0 \in \SC_c^+$ with compact support. As extinction (that is $\tau = \inf\{t \geq 0: u_t \equiv 0 \} = \inf\{t \geq 0: \langle u_t,1 \rangle = 0 \} < \infty$) happens with probability $0<\PP_{g_0}(\tau<\infty)<1$, we condition on non-extinction to obtain well-defined travelling wave solutions $\nu^{(g_0)}$. Note that $\nu^{(g_0)}$ denotes any subsequential limit obtained by this construction. The uniqueness of the limiting distribution remained as an open problem.

Recall the discussion below \eqref{equ:def-R0}. For $T>0$, denote by $\upsilon_T$ the \textit{left-upper measure} on $\SC_{tem}^+$ corresponding to $\SL\big(\xi_T^{(-\infty,0] \cap \ZZ}\big)$ (here, $\SL$ denotes ``law") in the contact process setup. By \cite[Remark~2.8]{K2017}, 
\eqn{
\lbeq{equ:laplace-left}
  \int e^{-2 \langle f,g \rangle} \upsilon_T(df) = \PP\!\left( \big\langle \1_{(-\infty,0)}(\fatcdot) , u_T^{(g)} \big\rangle = 0 \right), \quad \mbox{ for } g \in \SC_{tem}^+.
}
Furthermore, for $u_0 \in \SC_{tem}^+$ with $R_0(u_0) \leq 0$ and $T>0$ arbitrarily fixed one obtains the existence of a coupling with a random continuous process $(u_{T+t}^{*,l})_{t \geq 0}$ with values in $\SC_{tem}^+$ such that 
\eqn{
\lbeq{equ:coupling-left}
  u_{T+t}^{(u_0)}(x) \leq u_{T+t}^{*,l}(x) \quad \mbox{ for all } x \in \RR, t \geq 0 \mbox{ almost surely,}
}
where $\SL((u_{T+t}^{*,l})_{t \geq 0})=\PP_{\upsilon_T}$ holds. Note in particular that such a coupling yields
\eqn{
\lbeq{equ:coupling-R-left}
  R_0(u_{T+t}^{(u_0)}) \leq R_0(u_{T+t}^{*,l}) \mbox{ for all } t \geq 0 \mbox{ almost surely.} 
}
By symmetry, analogous results hold for a \textit{right-upper measure}, say $\kappa_T$, where we make use of the notations $L_0(f) \equiv \inf\{ x \in \RR: f(x)>0 \}$ and $u_{T+t}^{*,r}$ instead. In the appendix (cf. \eqref{equ:def-nu-star-lr}) we indicate how to modify the techniques of \cite{K2017} to construct travelling wave solutions $\nu^{*,l}$ respectively $\nu^{*,r}$ from $u^{*,l}$ respectively $u^{*,r}$.

The first main result of this article is a pre-step to a result in the spirit of the first case of \eqref{equ:asconv-markers}, where the almost sure convergence is replaced by $\SL^1$-convergence. 
%
%
\begin{proposition}
\label{PRO:alpha}
For all $\theta>\theta_c$, the limit $B \equiv B(\theta) \equiv \lim_{t \rightarrow \infty} \EE\!\left[ R_0\big( u_t^{*,l} \big) \right] / t$ exists and is strictly positive. Moreover, for all $\theta_c < \underline{\theta} \leq \theta_1 \leq \theta_2 \leq \overline{\theta}$, there exists a constant $C=C(\underline{\theta}, \overline{\theta})$ such that
\eqn{
\label{equ:alpha-incr}
  B(\theta_2)-B(\theta_1) \geq C(\theta_2-\theta_1).
}
\end{proposition}
%
%
We note that the strict positivity of $B(\theta)$ follows from \eqref{equ:alpha-incr}, once $B(\theta) \geq 0$ is established for all $\theta>\theta_c$. Our approach relies on establishing the estimate \eqref{equ:alpha-incr} along the lines of the corresponding result for contact processes in \cite[Lemma~4.2]{D1980}. 

Recall from above that $H_0$ denotes Heavyside initial data of the form $H_0(x) \equiv 1 \wedge (-x \vee 0)$.
%
%
\begin{definition}
Let 
\eqn{
\lbeq{equ:def-mittl}
  \mittl = \{ f \in \SC_{tem}^+: \exists\, x_0 \in \RR, \epsilon>0: f(x) \geq \epsilon H_0(x-x_0) \mbox{ for all } x \in \RR \}
}
and $\mittl^R = \{ f \in \mittl: R_0(f) \in \RR \}$.
\end{definition}
%
%

Our second main result concerns the limiting speeds of several right markers. It establishes in particular the existence of at least one travelling wave with positive deterministic speed.
%
%
\begin{theorem}
\label{THM:speeds}
Let $\theta>\theta_c$. Then 
\eqn{
\label{equ:result2a}
  R_0\big( u_T^{*,l} \big) / T \rightarrow B \mbox{ as } T \rightarrow \infty \mbox{ almost surely and in $\SL^1$.}
}
For any travelling wave solution $\nu^{*,l}$,
\eqn{
\lbeq{equ:result2b}
  R_0\big( u_T^{( \nu^{*,l} )} \big) / T \rightarrow B \mbox{ almost surely as } T \rightarrow \infty
}
and $\big( 0 \vee R_0\big( u_T^{(\nu^{*,l})} \big) \big)/T \rightarrow B$ in $\SL^1$.

For initial conditions $\psi \in \mittl^R$, 
\eqn{
  R_0\big( u_T^{(\psi)} \big)/T \rightarrow B \mbox{ as } T \rightarrow \infty \mbox{ in probability and in } \SL^1. 
}
For any travelling wave solution $\nu^{(\psi)}$,
\eqn{
  R_0\big( u_T^{(\nu^{(\psi)})} \big)/T \rightarrow B \mbox{ almost surely as } T \rightarrow \infty
}
and $\big( 0 \vee R_0\big( u_T^{(\nu^{(\psi)})} \big) \big)/T \rightarrow B$ in $\SL^1$.
\end{theorem}
%

Recall \eqref{equ:asconv-markers}. It remains to prove, for instance, that for $\theta<\theta_c$, $\big( R_0\big( u_T^{*,l} \big) \vee 0 \big) / T$ converges (in some sense) to $0$. In combination with \eqref{equ:result2a} this would then show that the edge speeds of solutions starting in left- or right-upper measures characterize critical values. This is work in progress. 
%

For the remainder, let us recall some notation and Theorem~2.2 from \cite{T1996} that are often used in the present article.

%
\begin{notation}[Notation from \cite{T1996}, also see Subsection~1.2 of \cite{K2017}]
\label{NOTA:tribe}
\begin{enumerate}
\item
Equip $\SC_{tem}^+$ with the topology given by the norms $\normx{f}{\lambda}$ for $\lambda>0$. Note that $d(f,g) \equiv \sum_{n \in \NN} (1 \wedge $ $\normx{f-g}{1/n})$ metrizes this topology and makes $\SC_{tem}^+$ a Polish space. Let $(\SC([0,\infty),\SC_{tem}^+),\SU,\SU_t,U(t))$ be continuous path space, the canonical right continuous filtration and the coordinate variables.
\item
In  \cite[(2.4)--(2.5)]{T1996}, the more general equation
\eqn{
\lbeq{equ:SPDE-Tribe}
  \partial_t u = \partial_{xx} u + \alpha + \theta u - \beta u - \gamma u^2 + u^{\frac{1}{2}} dW, \qquad t>0, x \in \RR, \qquad u(0,x) = u_0(x) \geq 0
}
with $\alpha, \beta, \gamma \in \SC([0,\infty),\SC_{tem}^+)$ is under consideration. We may interpret $\alpha$ as the immigration rate, $\theta-\beta$ as the mass creation-annihilation rate and $\gamma$ as the overcrowding rate.

A solution to \eqref{equ:SPDE-Tribe} consists of a filtered probability space $(\Omega,\SF,\SF_t,\PP)$, an adapted white noise $W$ and an adapted continuous $\SC_{tem}^+$ valued process $u(t)$ such that for all $\phi \in \SC_c^\infty$, the space of infinitely differentiable functions on $\RR$ with compact support,
\eqan{
  \langle u(t),\phi \rangle =& \langle u(0),\phi \rangle + \int_0^t \langle u(s),\phi_{xx}+(\theta-\beta(s)-\gamma(s)u(s)) \phi \rangle ds \\
  & + \int_0^t \langle \alpha(s),\phi \rangle ds + \int_0^t \int |u(s,x)|^{1/2} \phi(x) dW_{x,s}. \nn
}
If in addition $\PP(u(0,x)=f(x))=1$ then we say the solution $u$ starts at $f$.
\end{enumerate}
\end{notation}

\begin{theorem}[Theorem~2.2a)--b) of \cite{T1996}] \hfill
\label{THM:tribe}
\begin{enumerate}
\item[a)] For all $f \in \SC_{tem}^+$ there is a solution to \eqref{equ:SPDE-Tribe} started at $f$.
\item[b)] All solutions to \eqref{equ:SPDE-Tribe} started at $f$ have the same law which we denote by $Q^{f,\alpha,\beta,\gamma}$. The map $(f,\alpha,\beta,\gamma) \rightarrow Q^{f,\alpha,\beta,\gamma}$ is continuous. The laws $Q^{f,\alpha,\beta,\gamma}$ for $f \in \SC_{tem}^+$ form a strong Markov family.
\item[c)] For $R, T>0$ let $\SU_{R,T} = \sigma(U(t,x): t \leq T, |x| \leq R)$. Then the two laws $Q^{f,\alpha,\beta,\gamma}$, $Q^{f,\alpha,0,0}$ are mutually absolutely continuous on $\SU_{R,T}$.
\end{enumerate}
\end{theorem}
%

Note that Tribe \cite{T1996} later uses the notation $Q^f \equiv Q^{f,0,0,1}$ where we use $\PP_f$. Also, when the parameter $\theta$ in \eqref{equ:SPDE-Tribe} is not clear from the context, we write $Q^{f,\alpha,\beta,\gamma}(\theta)$.

Finally, let $\stackrel{\SD}{=}$ denote equality in distribution. Constants may change from line to line. We drop $\theta$ if the context is clear. \medskip

\textit{Outline.}
The paper is organized as follows. Sections~\ref{SEC:prelim}--\ref{SEC:increase-at-front} are dedicated to the proof of Proposition~\ref{PRO:alpha}, that is, the positivity of $B(\theta)$ for all $\theta>\theta_c$. In Subsections~\ref{SUBSEC:terms}--\ref{SUBSEC:prelim-est} the groundwork is laid for the proof of Proposition~\ref{PRO:alpha}. In Subsection~\ref{SUBSEC:estimates-right-markers} we already state the estimate that lies at the heart of the proof of Proposition~\ref{PRO:alpha}, see Proposition~\ref{PRO:exp-alpha-increase}. Its proof follows in Subsection~\ref{SUBSEC:Proof-theta-prop}. A substantial part of the proof goes into an estimate on the gain of mass at the front due to an increase in $\theta$, see Proposition~\ref{PRO:stopping-time-S}. We therefore postpone the proof of the latter to Section~\ref{SEC:increase-at-front}.

Section~\ref{SEC:speed-vg} is dedicated to the proof of Theorem~\ref{THM:speeds}, that is, the convergence of the linear speed of right markers to $B(\theta)$. In the appendix, Section~\ref{SEC:appendix}, the construction of travelling wave solutions from \cite{K2017} is extended to include $\nu^{*,l}$ and $\nu^{(\psi)}$ with initial conditions $\psi \in \mittl^R$ (cf. \eqref{equ:def-mittl} and below). Coupling techniques that are often used are summarized for reference.

%
\section{Preliminary results}
\label{SEC:prelim}
%
%
\subsection{The terms under investigation}
\label{SUBSEC:terms}
%

Let $\theta>\theta_c$ be arbitrarily fixed. Recall the sequence of laws $\big( \upsilon_T \big)_{T>0}$ on $\SC_{tem}^+$ and $\big( u_{T+t}^{*,l} \big)_{t \geq 0}$ for $T>0$ fixed satisfying $\SL\big( u_{T+t}^{*,l} \big) = \upsilon_{T+t}$ from \eqref{equ:laplace-left}--\eqref{equ:coupling-left}. Note that $\upsilon_{\fatcdot}=\upsilon_{\fatcdot}(\theta)$ and $u_{T+\fatcdot}^{*,l}=u_{T+\fatcdot}^{*,l}(\theta)$. From \cite[Corollary~4.7 and Notation~1.3-4.]{K2017} we conclude that the double-integrals below are well-defined with values in $[-\infty,\infty)$. Let
\eqn{
  \alpha_T(\theta) = \alpha_T = \frac{2}{T} \int_0^{T/2} \EE\!\left[ R_0(u_{T/2+s}^{*,l}) \right] ds.
}
In fact, $\EE\big[ R_0\big( u_T^{*,l} \big) \big] / T$ and $\alpha_T/T$ are uniformly bounded in $T \geq 1$ as we conclude from \cite[Corollary~4.7]{K2017} and the next lemma.
%
%
%
\subsection{Estimates on right-markers}
\label{SUBSEC:estimates-right-markers}
%
%

\framebox{Note that in this subsection, for all $\theta_c < \underline{\theta} \leq \theta \leq \overline{\theta}$ the constants to follow only depend on $\theta$ through $\underline{\theta}, \overline{\theta}$.}
%
%
\begin{lemma}
\label{LEM:speed-bound-mittl}
For all $u_0 \in \mittl$, there exists a constant $C=C(u_0)>0$ such that 
\eqn{
  \EE_{u_0}[ 0 \vee ( -R_0( u_t ) ) ] \leq C (1+t)
}
holds uniformly in $t \geq 0$. Moreover, there exist $C_i=C_i(u_0)>0, i=1,2$ such that for all $M>0$,
\eqn{
\lbeq{equ:lower-bd-mittl}
  \EE_{u_0}\!\left[ \frac{-R_0( u_t )}{t} \1_{\big\{ R_0( u_t ) < -Mt \big\}} \right] \leq C_1 e^{-C_2 M}
}
holds uniformly in $t \geq 1$.
\end{lemma}
%
%
\begin{proof}
Let $u_0 \in \mittl$. Recall $\epsilon, x_0$ from the definition of $\mittl$. By domination, that is using 
\cite[Lemma~3.1b)]{T1996}, we assume without loss of generality that $u_0=\epsilon H_0(\fatcdot-x_0)$. We further assume $x_0=0$ by the shift invariance of the dynamics.

We reason as in the proof of \cite[Lemma~3.5]{T1996}. The author uses the wave-marker $R_1(t) = \ln(\langle e^{\fatcdot} , u_t \rangle)$ and Heavyside initial data $H_0$ instead. It is shown that there exist $c=c(\theta), a=a(\theta), \delta=\delta(\theta)>0$ such that $\PP_{H_0}\big( R_1(t) \leq -a-cmt \big) \leq (1-\delta/4)^m$ for all $t \geq 0, m \in \NN$. 

We claim that this holds for $H_0$ replaced by $u_0 = \epsilon H_0$, $R_1(t)$ replaced by $R_0(t)$ and $a$ replaced by $0$ as well. Moreover, the constants $c,\delta$ only depend on $\epsilon, \underline{\theta}$ and $\overline{\theta}$. Indeed, replace $a>0$ by $a=0$. Reason as in the given proof with $R_1(f)$ replaced by $R_0(f)$ and $\psi_0$ replaced by $\psi_0' \equiv \epsilon \psi_0$ until the last set of equations. Choose $r=ct$ for $t \geq 1$ arbitrarily fixed and $r=c$ for $t \in [0,1)$ (and thus $Q^{\psi_0}(T_0(U) \leq t) \leq Q^{\psi_0}(T_0(U) \leq 1) \leq \delta/4$ in the notation of \cite{T1996}). In the last set of equations, use that for a superprocess with initial symmetric condition $\psi_0'$ and law $\overline{\PP}_{\psi_0'}$, there exists $\delta>0$ small enough such that $\overline{\PP}_{\psi_0'}(R_0(u_t) \geq 0) \geq \overline{\PP}_{\psi_0'}(\tau > t)/2 \geq \PP_{\psi_0'}(\langle u_t , 1 \rangle \geq \delta) \geq \delta/2$ to obtain for all $t \geq 1$,
\eqn{
  \PP_{u_0}\big( R_0(t) \leq -cmt \big) \leq (1-\delta/4)^m \ \mbox{ for all } m \in \NN.
}
As a result,
\eqn{
  \EE_{u_0}[ 0 \vee (-R_0( u_t ))]
  \leq ct + \sum_{m \in \NN} (1-\delta/4)^m c(m+1) t
  \leq C(c,\delta) t.
}
For $t \in [0,1)$, the different choice of $r$ yields
\eqn{
  \PP_{u_0}\big( R_0(t) \leq -cm \big) \leq (1-\delta/4)^m \ \mbox{ for all } m \in \NN
}
and
\eqn{
  \EE_{u_0}[ 0 \vee (-R_0( u_t ))]
  \leq c + \sum_{m \in \NN} (1-\delta/4)^m c(m+1)
  \leq C(c,\delta)
}
instead.

By $\lfloor x \rfloor$ we denote the greatest integer that is less than or equal to $x \in \RR$. To obtain the second claim, for $0<M<c$ choose $C_1$ big enough and $C_2$ small enough such that $C(c,\delta) \leq C_1 e^{-C_2 c}$.  For $M \geq c$ and $t \geq 1$, 
\eqan{
  \EE_{u_0}\!\left[ \frac{-R_0( u_t )}{t} \1_{\big\{ R_0( u_t ) < -Mt \big\}} \right]
  &\leq \sum_{m = \lfloor M/c \rfloor}^\infty (1-\delta/4)^m c(m+1) \\
  &= e^{\ln(1-\delta/4) \lfloor M/c \rfloor} \sum_{m = 0}^\infty (1-\delta/4)^m c \big( m+1+ \lfloor M/c \rfloor \big) 
 \leq C_1 e^{-C_2 M} \nn
}
for $C_1=C_1(c,\delta)$ big enough and $C_2=C_2(c,\delta)$ small enough.
\end{proof}
%
%
\begin{corollary}
\label{COR:exp-ustarl-bd-below}
There exists a constant $C>0$ such that 
\eqn{
\lbeq{equ:neg-speed-bound-star}
  \EE\!\left[ 0 \vee \big( -R_0\big( u_t^{*,l} \big) \big) \right] \leq C (1+t)
}
holds uniformly in $t > 0$. Moreover, there exist $C_i>0, i=1,2$ such that for all $M>0$,
\eqn{
\lbeq{equ:exp-ustarl-bd-below}
  \EE\!\left[ \frac{-R_0\big( u_t^{*,l} \big)}{t} \1_{\big\{ R_0\big( u_t^{*,l} \big) < -Mt \big\}} \right] \leq C_1 e^{-C_2 M}
}
holds uniformly in $t \geq 1$.
\end{corollary}
%
%
\begin{proof}
The result follows again by domination, this time using \eqref{equ:coupling-R-left} and $u_0 \in \mittl^R$ with $R_0(u_0) \leq 0$ arbitrary.
\end{proof}
%
%
\begin{corollary}
\label{COR:e-bd}
$\EE_{u_0}[ | R_0(u_T) | ]/T$, $u_0 \in \mittl^R$ and $\EE\big[ \big| R_0\big( u_T^{*,l} \big) \big| \big] / T$ are uniformly bounded in $T \geq 1$ (constants may depend on $u_0$).
\end{corollary}
%
%
\begin{proof}
Combine Lemma~\ref{LEM:speed-bound-mittl} respectively Corollary~\ref{COR:exp-ustarl-bd-below} with \cite[Lemma~4.6]{K2017}.
\end{proof}
%
%
\begin{corollary}
\label{COR:e-bd-small-t}
There exists a constant $C>0$ such that $\EE\big[ \big| R_0\big( u_s^{*,l} \big) \big| \big] \leq C$ for all $0 < s \leq 1$. Moreover, for every $u_0 \in \mittl^R$ there exists a constant $C(u_0)>0$ such that $\EE[|R_0(u_s)|] \leq C(u_0)$ for all $0 \leq s \leq 1$.
\end{corollary}
%
%
\begin{proof}
Fix $0< s \leq 1$. Then $\EE\big[ 0 \vee R_0\big( u_s^{*,l} \big)\big] \leq C$ follows from \cite[Proposition~4.5]{K2017} and $\EE\big[ 0 \vee \big(-R_0\big( u_s^{*,l} \big) \big) \big] \leq C$ from Corollary~\ref{COR:exp-ustarl-bd-below}. For $u_0 \in \mittl^R$, the bound for the positive part follows by domination and shift invariance, that is, $\EE[0 \vee R_0(u_s)] \leq |R_0(u_0)| + \EE\big[ 0 \vee R_0\big( u_s^{*,l} \big)\big] \leq |R_0(u_0)| + C = C(u_0)$. The bound for the negative part follows from Lemma~\ref{LEM:speed-bound-mittl}.
%
%
%
%
\end{proof}
%
%
\begin{corollary}
\label{COR:e-bd-2}
$\alpha_T/T$ is uniformly bounded in $T \geq 1$.
\end{corollary}
%
%
\begin{remark}
\label{RMK:ustar-survives}
The definition of the marker $R_0$ together with Corollaries~\ref{COR:e-bd}--\ref{COR:e-bd-small-t} yields $\inf\{ t > 0: u_t^{*,l} \equiv 0 \} = + \infty$ a.s., that is, the process $u^{*,l}$ does not die out in finite time. Thus, if we consider $u^{*,l}$, we do not have to bother with conditioning on non-extinction.
\end{remark}
%

The existence of the following limit will turn out to be crucial in the following chapters. Non-negativity of the limit follows below.
%
%
\begin{lemma}
\label{LEM:alpha_lim}
The limit
\eqn{
  B=B(\theta)
  = \lim_{T \rightarrow \infty} \frac{ \EE\!\left[ R_0\big( u_T^{*,l} \big) \right] }{T}
  = \inf_{T \geq 1} \frac{ \EE\!\left[ R_0\big( u_T^{*,l} \big) \right] }{T}
  \in (-\infty,\infty)
}
exists.
\end{lemma}
%
%
\begin{proof}
We work with $\PP_{\upsilon_1}$, that is randomize the initial condition according to the law of $u_1^{*,l} \in \SC_{tem}^+$. By the strong Markov property of the process we have for arbitrary $1 \leq s, t$,
\eqan{
  \EE\!\left[ R_0\big( u_{s+t}^{*,l} \big) \right] 
  &= \EE\!\left[ \EE\!\left[ R_0\big( u_{s+t}^{*,l} \big) \mid \SF_t \right] \right] 
  = \EE\!\left[ R_0\Big( u_s^{(u^{*,l}_t)} \Big) \right] \\
  &= \EE\!\left[ R_0\Big( u_s^{\big( u^{*,l}_t (\fatcdot + R_0( u_t^{*,l})) \big)} \Big) \right] + \EE\!\left[ R_0\big( u_t^{*,l} \big) \right]. \nn
}
Use monotonicity, that is \eqref{equ:coupling-R-left} to further obtain
\eqn{
  \EE\!\left[ R_0\big( u_{s+t}^{*,l} \big) \right] 
  \leq \EE\!\left[ R_0\big( u^{*,l}_s \big) \right] + \EE\!\left[ R_0\big( u^{*,l}_t \big) \right]. 
}
By subadditivity (cf. for instance Liggett \cite[Theorem~B22]{bL1999}) and from the uniform boundedness of $\EE\big[ \big| R_0\big( u_T^{*,l} \big) \big| \big]/T$ in $T \geq 1$, we conclude,
\eqn{
  \lim_{T \rightarrow \infty} \frac{\EE\!\left[ R_0\big( u_T^{*,l} \big) \right] }{T} = \inf_{T>0} \frac{\EE\!\left[ R_0\big( u_T^{*,l} \big) \right] }{T} \mbox{ exists in } (-\infty,\infty).
}
\end{proof}
%
%
\begin{corollary}
\label{COR:alpha_lim}
The limit $\lim_{T \rightarrow \infty} \frac{\alpha_T}{T} = \frac{3}{4} B$ exists.
\end{corollary}
%
%
\begin{proof}
For all $\epsilon>0$ there exists $T_0 \geq 1$ such that for all $T \geq T_0$,
\eqn{
  \limsup_{T \rightarrow \infty} \frac{\alpha_T}{T}
  = \limsup_{T \rightarrow \infty} \frac{2}{T^2} \int_0^{T/2} \EE\!\left[ R_0(u_{T/2+s}^{*,l}) \right] ds 
  \leq \limsup_{T \rightarrow \infty} \frac{2}{T^2} \int_0^{T/2} (\epsilon+B)  (T/2+s) ds
  = \frac{3}{4}( \epsilon + B).
}
Analogous reasoning for a lower bound concludes the proof. 
\end{proof}
%

The limit is indeed non-negative.

%
\begin{lemma}
\label{LEM:alpha_nonneg}
The limit $B=\lim_{T \rightarrow \infty} \frac{ \EE\!\left[ R_0\big( u_T^{*,l} \big) \right] }{T}$ from Lemma~\ref{LEM:alpha_lim} is non-negative.
\end{lemma}
%
%
\begin{proof}
Let $\nu \in \CP(\SC_{tem}^+)$ be such that $\nu(\{ f: R_0(f)=0 \}) = 1$ and $\PP_\nu$ is the law of a travelling wave. For $\theta>\theta_c$, existence follows from \cite[Theorem~3.8 and (3.29)]{T1996} and the shift invariance of the dynamics. By \cite[Proposition~4.1]{T1996}, $R_0\big( u_t^{(\nu)} \big) / t$ converges a.s. to a (possibly random) limit $A^{(\nu)} \geq 0$. By monotonicity, that is by \eqref{equ:coupling-R-left}, we have $R_0\big( u_t^{*,l} \big) / t \geq R_0\big( u_t^{(\nu)} \big) / t$ for all $t \geq 1$ a.s. and thus $\liminf_{t \rightarrow \infty} R_0\big( u_t^{*,l} \big) / t \geq 0$ a.s. 

Let $\epsilon>0$ arbitrary. By Corollary~\ref{COR:exp-ustarl-bd-below} there exist constants $C_1, C_2>0$ such that for $M>0$ satisfying $C_1 e^{-C_2 M} < \epsilon$,
\eqan{
  B
  \geq \limsup_{T \rightarrow \infty} \frac{\EE\big[ R_0\big( u_T^{*,l} \big) \1_{\big\{ R_0\big( u_T^{*,l} \big) \geq -MT \big\}} \big]}{T} - \epsilon 
  \geq \EE\!\left[ \liminf_{T \rightarrow \infty} \frac{ R_0\big( u_T^{*,l} \big) }{T} \1_{\big\{ R_0\big( u_T^{*,l}) \big) \geq -MT \big\}} \right] - \epsilon 
  \geq -\epsilon 
}
where we applied Fatou's lemma.
\end{proof}
%

We now formulate the main result of this section. The proof is deferred to Subsection~\ref{SUBSEC:Proof-theta-prop}.
%
%
\begin{proposition}
\label{PRO:exp-alpha-increase}
Let $\theta_c < \underline{\theta} < \overline{\theta}$. Then there exists $C=C\big( \underline{\theta}, \overline{\theta} \big)>0$ and $T_0=T_0\big( \underline{\theta}, \overline{\theta} \big) \geq 1$ such that for all $T \geq T_0$ and $\underline{\theta} \leq \theta_1< \theta_2 \leq \overline{\theta}$,
\eqn{
  \frac{ \alpha_T(\theta_2) - \alpha_T(\theta_1) }{T} \geq C(\theta_2-\theta_1).
}
\end{proposition}
%
%
\begin{corollary}
\label{COR:Bg0}
For all $\theta>\theta_c$, 
\eqn{
\lbeq{equ:result}
  B = B(\theta) = \lim_{T \rightarrow \infty} \frac{\EE\!\left[ R_0\big( u_T^{*,l}(\theta) \big) \right] }{T} > 0.
}
\end{corollary}
%
%
\begin{proof}
By definition of $\alpha_T$, Corollary~\ref{COR:alpha_lim} and Lemma~\ref{LEM:alpha_nonneg}, Proposition~\ref{PRO:exp-alpha-increase} implies that for all $\theta>\theta_c$, 
\eqn{
  \lim_{T \rightarrow \infty} \frac{\EE\!\left[ R_0\big( u_T^{*,l}(\theta) \big) \right] }{T}
  = B(\theta) = \frac{4}{3} \lim_{T \rightarrow \infty} \frac{\alpha_T(\theta)}{T}
  > \frac{4}{3} \lim_{T \rightarrow \infty} \frac{\alpha_T(\theta_c + (\theta-\theta_c)/2)}{T} 
  \geq 0.
}
\end{proof}
%
%

We conclude this subsection with two more results that we need for later estimates.
%
%
\begin{lemma}
\label{LEM:marker-nonneg}
Let $\theta_c < \underline{\theta}$, then there exists $\tilde{\delta}>0$ such that
\eqn{
  \PP\big( R_0\big( u_T^{*,l}(\theta) \big) \geq 0 \big) \geq \tilde{\delta}
}
for all $T \geq 1$ and $\theta \geq \underline{\theta}$.
\end{lemma}
%
%
\begin{proof}
Use a \textit{$\theta$-$*$-coupling} to see that it suffices to show the claim for $\underline{\theta}$ fixed. Note that for $T \geq 1$ arbitrarily fixed, $\PP\big( R_0\big( u_T^{*,l}(\underline{\theta}) \big) \geq 0 \big) > 0$. Therefore, in the following proof by contradiction we only need to suppose to the contrary that there exists a sequence $(T_n)_{n \in \NN}$ such that $T_n \rightarrow \infty$ for $n \rightarrow \infty$ and $\lim_{n \rightarrow \infty} \PP\big( R_0\big( u_{T_n}^{*,l} \big) \geq 0 \big) = 0$. Let $H_0(x) = 1 \wedge (-x \vee 0)$ be Heavyside initial data and set $f_1(x) = H_0(x+1) + H_0(-x-1)$. Then $f_1 \in \SC_{tem}^+$ and $\mbox{supp}(f_1) = (-\infty,-1] \cup [1,\infty)$. Let $f_2(x) = 0 \vee (1-|x|)$, then $f_2 \in \SC_c^+$ with $\mbox{supp}(f_2)=[-1,1]$. $f_1$ fulfills condition \cite[(6)]{HT2004} and hence \cite[Theorem~1]{HT2004} yields $u_t^{(f_1)} \Rightarrow \mu$ for $t \rightarrow \infty$. Using a \textit{coupling with two independent processes} in combination with the construction of \cite[Remark~2.8(ii)]{K2017}, we construct two independent processes $(u_t^{*,l})_{t \geq 1}$ and $(u_t^{*,r})_{t \geq 1}$ such that $\SL(( u_t^{*,l} )_{t \geq 1}) = \PP_{\upsilon_1}$, $\SL(( u_t^{*,r} )_{t \geq 1}) = \PP_{\kappa_1}$ and
\eqn{
  u_t^{(f_1)} \leq u_t^{*,l}(\fatcdot + 1) + u_t^{*,r}(\fatcdot - 1) \ \mbox{ for all } \ t \geq 1, x \in \RR \ \mbox{ almost surely.}
}
By \cite[Theorem~1]{HT2004} and \cite[Theorem~1]{MT1994},
\eqn{
  \int_{\SC_{tem}^+} e^{-2 \langle g , f_2 \rangle} \mu(dg) = \PP_{f_2}\big( \tau<\infty \big) < 1. 
}
Note that the additional factor of $2$ in the exponent results from the use of a different scaling constant in the original SPDE. We obtain by the weak convergence of $u_t^{(f_1)}$ to $\mu$,
\eqan{
\lbeq{equ:marker-geq-zero}
  1 
  & > \lim_{n \rightarrow \infty} \int_{\SC_{tem}^+} e^{-2 \langle g , f_2 \rangle} u_{T_n}^{(f_1)}(dg)
  = \lim_{n \rightarrow \infty} \EE\!\left[ e^{-2 \langle u_{T_n}^{(f_1)} , f_2 \rangle} \right] \\
  & \geq \lim_{n \rightarrow \infty} \EE\!\left[ e^{-2 \langle u_{T_n}^{*,l}(\fatcdot + 1) + u_{T_n}^{*,r}(\fatcdot - 1) , f_2 \rangle} \right]
  = \lim_{n \rightarrow \infty} \EE\!\left[ e^{-2 \langle u_{T_n}^{*,l}(\fatcdot + 1) , f_2 \rangle} \right] \EE\!\left[ e^{-2 \langle u_{T_n}^{*,r}(\fatcdot - 1) , f_2 \rangle} \right]. \nn
}
The assumption $\lim_{n \rightarrow \infty} \PP\big( R_0\big( u_{T_n}^{*,l} \big) \geq 0 \big) = 0$ yields by symmetry and by the shift invariance of the dynamics, $\lim_{n \rightarrow \infty} \PP\big( R_0\big( u_{T_n}^{*,l}(\fatcdot + 1 \big) \geq -1 \big) = 0 = \lim_{n \rightarrow \infty} \PP\big( L_0\big( u_{T_n}^{*,r}(\fatcdot - 1 \big) \leq 1 \big) = 0$. Use a coupling $u_{T_n}^{*,r} \leq u_{T_n}^*$ with $u_{T_n}^* \Rightarrow \mu \in \CP(\SC_{tem}^+)$ (cf. \cite[(2.34) and Proposition~2.4]{K2017}) to conclude by using dominated convergence that the right hand side in \eqref{equ:marker-geq-zero} is equal to $1$, a contradiction.
\end{proof}
%
%
\begin{lemma}
\label{LEM:marker-pos-part-second-moment}
For all $\theta_c < \theta \leq \overline{\theta}, T \geq 1$, 
\eqn{
  \EE\!\left[ \left( 0 \vee R_0\big( u_T^{*,l} \big) \right)^2 \right] \leq C\big( \overline{\theta} \big) T^2.
}
\end{lemma}
%
%
\begin{proof}
In what follows, constants $C=C(\overline{\theta})$ may change from line to line. Note that for $a_i \geq 0, i=1,\ldots,n, n \in \NN$, $\big( \sum_{i=1}^n a_i \big)^2 \leq n \sum_{i=1}^n a_i^2$. 

We first show the claim for $T \in \NN$. Reason as in \cite[Lemma~4.2--Proposition~4.5]{K2017} to show that for $T=1$, $\EE\big[ \big( 0 \vee R_0\big( u_1^{*,l} \big) \big)^2 \big] \leq C$. Then reason as in \cite[Lemma~4.6]{K2017} to show the claim for $T \in \NN$ by induction.

Next, we extend this result to $T \geq 1$. As $\SL\big( u_T^{*,l} \big) \in \CP(\SC_{tem}^+ \backslash \{0\})$ for all $T>0$ we use \cite[Remark~2.8]{K2017} to get for $T \geq 1$ arbitrary,
\eqn{
  \EE\!\left[ \left( 0 \vee R_0\big( u_T^{*,l} \big) \right)^2 \right]
  = \EE\!\left[ \left( 0 \vee R_0\Big( u_{T - \lfloor T \rfloor}^{\big( u_{\lfloor T \rfloor}^{*,l} \big)} \Big) \right)^2 \right].
}
By \cite[Remark~A.1]{K2017} and symmetry,
\eqn{
  \EE\!\left[ \left( 0 \vee R_0\Big( u_{T - \lfloor T \rfloor}^{\big( u_{\lfloor T \rfloor}^{*,l} \big)} \Big) \right)^2 \Bigmid u_{\lfloor T \rfloor}^{*,l} \right]
  \leq \left( 0 \vee \big( R_0\big( u_{\lfloor T \rfloor}^{*,l} \big) + 2 \big) \right)^2 + C \int_{0 \vee \big(R_0\big( u_{\lfloor T \rfloor}^{*,l} \big) + 2 \big)}^\infty 2R \big\langle e^{-\frac{(\fatcdot-(R-1))^2}{4 (T - \lfloor T \rfloor)}} , u_{\lfloor T \rfloor}^{*,l} \big\rangle dR.
}
Take expectations and use \cite[Corollaries~2.6 and 2.9]{K2017} to conclude that
\eqn{
  \EE\!\left[ \left( 0 \vee R_0\Big( u_{T - \lfloor T \rfloor}^{\big( u_{\lfloor T \rfloor}^{*,l} \big)} \Big) \right)^2 \right]
  \leq C \lfloor T \rfloor^2 + C \int_0^\infty 2R \big\langle e^{-\frac{(\fatcdot-(R-1))^2}{4}} , 1 \big\rangle dR
  \leq CT^2
}
as claimed. 
\end{proof}
%
%
\subsection{A preliminary estimate}
\label{SUBSEC:prelim-est}
%
%

The following two lemmas yield, in combination, a lower bound on the expected increase of the right front marker at time $T+t, T>0, t \geq 0$ resulting from an increase of $\psi \in \SC_{tem}^+$ in the initial density of a solution to \eqref{equ:SPDE}.

Recall the construction of the left upper invariant measure $\upsilon_T$ and the process $\big( u_{T+t}^{*,l} \big)_{t \geq 0}$ for $T>0$ fixed from \cite{K2017} (cf. the corresponding construction for the upper invariant measure $\mu_T$ from \cite{K2017}, Proposition~2.2 and Corollary~2.6 as well as Remark~2.8). For arbitrarily fixed (to be chosen later) $\psi \in \SC_{tem}^+$, write
\eqn{
\lbeq{equ:def-Phi-Psi}
  \Phi(x) \equiv \begin{cases} \infty, & x < 0, \cr 0, & \mbox{otherwise} \end{cases}
  \quad \mbox{ and } \quad 
  \Psi(x) \equiv \begin{cases} \infty, & x < 0, \cr \psi(x), & x \geq 0, \cr 0, & \mbox{otherwise.} \end{cases}
}
In what follows consider couplings of solutions $\big( u_{T+t}^{(\phi)} \big)_{t \geq 0}$, $\phi \in \SC_{tem}^+$ and $\big( u_{T+t}^{(\phi+\psi)} \big)_{t \geq 0}$, $\psi \in \SC_{tem}^+$ with processes $\big( u_{T+t}^{(\Phi)} \big)_{t \geq 0}$ and $\big( u_{T+t}^{(\Psi)} \big)_{t \geq 0}$ for $T>0$ arbitrarily fixed. Note that by a slight abuse of notation "$\Psi = \Phi + \psi$". The two latter processes are to be understood in the spirit of the construction of $\upsilon_T$, that is as in Corollary~2.6 we choose sequences $( \Psi_N )_{N \in \NN}$ and $( \Phi_N )_{N \in \NN}$ such that $\Psi_N \uparrow \Psi$ and $\Phi_N \uparrow \Phi$ for $N \rightarrow \infty$ to obtain $u_{T+t}^{(\Psi)}(x) \equiv \uparrow \lim_{N \rightarrow \infty} u_{T+t}^{(\Psi_N)}(x)$ and $u_{T+t}^{(\Phi)} \equiv \uparrow \lim_{N \rightarrow \infty} u_{T+t}^{(\Phi_N)}(x)$ on a common probability space.
%
%
\begin{lemma} 
\label{LEM:claim2}
Let $\psi \in \SC_{tem}^+$ arbitrarily fixed and $\Phi, \Psi$ be as above. Let $\phi \in \SC_{tem}^+$ arbitrary with $R_0(\phi) \leq 0$. Then, for arbitrary $T>0, t \geq 0$, there exists a coupling of processes $\big( u^{(\Phi)}_{T+t} \big)_{t \geq 0}, \big( u^{(\Psi)}_{T+t} \big)_{t \geq 0}$ and solutions $\big( u^{(\phi)}_{T+t} \big)_{t \geq 0}, \big( u^{(\phi+\psi)}_{T+t} \big)_{t \geq 0}$ such that
\eqn{
\lbeq{equ:claim2}
  \EE\!\left[ R_0\big( u_{T+t}^{(\Psi)} \big)-R_0\big( u_{T+t}^{(\Phi)} \big) \right]
  \leq \EE\!\left[ R_0\big( u_{T+t}^{(\phi+\psi)} \big)-R_0\big( u_{T+t}^{(\phi)} \big) \right]
}
for all $t \geq 0$ almost surely. On the right hand side we consider a \textit{monotonicity-coupling} and set $R_0\big( u_{T+t}^{(\phi+\psi)} \big)-R_0\big( u_{T+t}^{(\phi)} \big) = 0$ on $\big\{ \tau^{(\phi+\psi)} \leq T+t \big\}$.
\end{lemma}
%
%
\begin{remark}
\label{RMK:claim2}
Note that the expectations on the left hand side of \eqref{equ:claim2} are well-defined by Lemma~\ref{LEM:speed-bound-mittl} and Corollaries~\ref{COR:e-bd}--\ref{COR:e-bd-small-t}. Indeed, note that if $f_n \uparrow f$ in $\SC_{tem}^+$, then $R_0(f_n) \uparrow R_0(f)$ for $n \rightarrow \infty$. Now use approximating sequences $\Psi_N, \Phi_N \in \mittl^R$ for $\Psi$ respectively $\Phi$ from below as in \cite[Remark~2.8(i)]{K2017} in combination with dominated convergence.
\end{remark}
%
%
\begin{proof}
\textit{Step 1.} Let $\phi, \psi$ as in the statement above. We first show the claim for $T=0$ and $\Phi, \Psi \in \SC_{tem}^+$ satisfying $R_0(\Phi)=0$, $\Phi \geq \phi$ and $\Psi = \Phi+\psi$. Consider the following coupling. Let $u_1 = u^{(\phi)}$ be a non-negative solution to
\eqn{
  \frac{\partial u_1}{\partial t} = \Delta u_1 + (\theta -u_1) u_1 + \sqrt{u_1} \dot{W_1}, \qquad u_1(0)=\phi,
}
and $v_2$ be a non-negative solution to
\eqn{
  \frac{\partial v_2}{\partial t} = \Delta v_2 + (\theta -v_2-2u_1) v_2 + \sqrt{v_2} \dot{W_2}, \qquad v_2(0)=\Phi-\phi
}
with $W_2$ a white noise independent of $W_1$. For the construction of the latter proceed as in Remark~\ref{RMK:monotonicity-coupling} on \textit{monotonicity-couplings}. Then $u_1+v_2 \stackrel{\SD}{=} u^{(\Phi)}$, that is, $u_1+v_2$ solves \eqref{equ:SPDE} with initial condition $\Phi$. Let $v_3$ be a non-negative solution to
\eqn{
  \frac{\partial v_3}{\partial t} = \Delta v_3 + (\theta -v_3 - 2(u_1+v_2)) v_3 + \sqrt{v_3} \dot{W_3}, \qquad v_3(0)=\psi
}
with $W_3$ a white noise independent of $W_1, W_2$. Then $u_1+v_2+v_3 \stackrel{\SD}{=} u^{(\Phi+\psi)}$ follows as above, and using that $\Psi=\Phi+\psi$,
\eqan{
\lbeq{equ:diff_R0_1}
  R_0\big( u_t^{(\Psi)} \big) - R_0\big( u_t^{(\Phi)} \big) 
  &\stackrel{\SD}{=} R_0\big( (u_1+v_2+v_3)_t \big) - R_0\big( (u_1+v_2)_t \big) \\
  &= \big( R_0\big( (v_3)_t \big) - R_0\big( (u_1+v_2)_t) \big) \vee 0 \nn
}
for all $t \geq 0$ a.s., where we set $R_0\big( u_t^{(\Psi)} \big)-R_0\big( u_t^{(\Phi)} \big) = 0$ on $\big\{ \tau^{(\Psi)} \leq t \big\}$. Finally, let $d_4$ be a non-negative solution to
\eqn{
  \frac{\partial d_4}{\partial t} = \Delta d_4 + 2 v_2 v_3+ (\theta -d_4 - 2(u_1+v_3)) d_4 + \sqrt{d_4} \dot{W_4}, \qquad d_4(0)=0
}
with $W_4$ independent of $W_1, W_2, W_3$ and where the term $2 v_2 v_3$ can be interpreted as an additional immigration term. Then $u_1+v_3+d_4 \stackrel{\SD}{=} u^{(\phi+\psi)}$ and
\eqan{
  R_0\big( u_t^{(\phi+\psi)} \big) - R_0\big( u_t^{(\phi)} \big)
  &= R_0\big( (u_1+v_3+d_4)_t \big) - R_0\big( (u_1)_t \big) \\
  &= \big( R_0\big( (v_3+d_4)_t \big) - R_0\big( (u_1)_t \big) \big) \vee 0 \nn\\
  &\geq \big( R_0\big( (v_3)_t \big) - R_0\big( (u_1+v_2)_t \big) \big) \vee 0, \nn
}
the last by the non-negativity of the solutions $d_4$ and $v_2$. 

The second part of the claim now follows from the above and \eqref{equ:diff_R0_1}. For the first part of the claim, use that $\Psi = \Phi + \psi$ and
\eqn{
  u_1 \stackrel{\SD}{=} u^{(\phi)}, \quad u_1+v_2 \stackrel{\SD}{=} u^{(\Phi)}, \quad u_1+v_2+v_3 \stackrel{\SD}{=} u^{(\Phi+\psi)}, \quad u_1+v_3+d_4 \stackrel{\SD}{=} u^{(\phi+\psi)}
}
to obtain a coupling satisfying
\eqn{
  u_t^{(\Psi)}-u_t^{(\Phi)} = \big( u_1+v_2+v_3\big) - \big( u_1+v_2 \big) = v_3 \leq v_3 + d_4 =  \big( u_1+v_3+d_4 \big) - u_1 = u_t^{(\phi+\psi)}-u_t^{(\phi)}
}
as claimed.

\textit{Step 2.} Fix $T>0$. Let $\Phi_N \uparrow \Phi$, $\Phi$ as in \eqref{equ:def-Phi-Psi}, satisfy $R_0(\Phi_N)=0$ and $\Phi_1 \geq \phi, \Phi_1 \in \mittl^R$. Set $\Psi_N = \Phi_N + \psi$. By Step~1, there exists a coupling of solutions $\big( u^{(\Phi_N)}_{T+t} \big)_{t \geq 0}, \big( u^{(\Psi_N)}_{T+t} \big)_{t \geq 0}$ to \eqref{equ:SPDE} such that \eqref{equ:claim2} holds with $\Phi, \Psi$ replaced by $\Phi_N, \Psi_N$ for $N \in \NN$ arbitrarily fixed. 

Define $u_{T+t}^{(\Phi)}(x) = \uparrow \lim_{N \rightarrow \infty} u_{T+t}^{(\Phi_N)}(x)$ and $u_{T+t}^{(\Psi)}(x) = \uparrow \lim_{N \rightarrow \infty} u_{T+t}^{(\Psi_N)}(x)$ on a common probability space (cf. \cite[Remark~2.8(i)]{K2017}). By taking limits in $N \rightarrow \infty$, the claim now follows for $\Phi, \Psi$ as well by dominated convergence (cf. Remark~\ref{RMK:claim2} above). 
\end{proof}
%
%
\begin{lemma} \label{LEM:claim1}
For $t>0$ fixed and $\psi \in \SC_{tem}^+$,
\eqn{
  \EE\!\left[ R_0\big( u_t^{(\Psi)} \big) \vee 0 - R_0\big( u_t^{(\Phi)} \big) \vee 0 \right] 
  \geq \int_0^\infty \EE\!\left[ \1_{\big\{ -x \leq L_0\big( u_t^{*,r} \big) < -x+1 \big\}} \left( 1 - e^{-2 \langle \1_{[0,1)} \psi , u_t^{*,r}(\fatcdot - x) \rangle } \right) \right] dx
}
holds.
\end{lemma}
%
%
\begin{proof}
By partial integration, for $\phi \in \SC_{tem}^+, t>0$ arbitrary,
\eqn{
\lbeq{equ:prob-right-marker}
  \EE_\phi\!\left[ R_0\big( u_t \big) \vee 0 \right] 
  = \int_0^\infty \PP_\phi\!\left( R_0\big( u_t \big) > x \right) dx.
}
By \eqref{equ:laplace-left}, symmetry and by the shift invariance of the dynamics,
\eqn{
  \PP_\phi\!\left( R_0\big( u_t \big) \leq x \right) 
  = \EE\!\left[ e^{-2 \langle \phi , u_t^{*,r}(\fatcdot - x) \rangle } \right]
  = \EE\!\left[ e^{-2 \langle \phi(\fatcdot + x) , u_t^{*,r} \rangle } \right].
}
Hence,
\eqn{
\lbeq{equ:exp_right_marker-2}
  \EE_\phi\!\left[ R_0\big( u_t \big) \vee 0 \right] 
  = \int_0^\infty \EE\!\left[ \left( 1 - e^{-2 \langle \phi , u_t^{*,r}(\fatcdot - x) \rangle } \right) \right] dx
  = \int_0^\infty \EE\!\left[ \1_{\big\{ u_t^{*,r} |_{\mbox{\tiny supp}(\phi(\fatcdot + x))} \not\equiv 0 \big\}} \left( 1 - e^{-2 \langle \phi , u_t^{*,r}(\fatcdot - x) \rangle } \right) \right] dx. 
}
In the following we use $\Phi$ and $\Psi=\Phi+\psi$ as initial conditions or test functions to facilitate notation. This notation is understood as an abbreviation for taking limits of non-decreasing approximating sequences of initial conditions as explained above and using monotone convergence to obtain the respective results.
 
Reason as in Remark~\ref{RMK:claim2} to see that the following integrals are well-defined. The Theorem of Fubini-Tonelli yields
\eqan{
  \EE\!\left[ R_0\big( u_t^{(\Psi)} \big) \vee 0 - R_0\big( u_t^{(\Phi)} \big) \vee 0  \right] 
  &= \int_0^\infty \EE\!\left[ e^{-2 \langle \Phi(\fatcdot + x) , u_t^{*,r} \rangle } \right] - \EE\!\left[ e^{-2 \langle \Psi(\fatcdot + x) , u_t^{*,r} \rangle } \right] dx \\
  &= \int_0^\infty \EE\!\left[ e^{-2 \langle \Phi(\fatcdot + x) , u_t^{*,r} \rangle } \left( 1 - e^{-2 \langle \psi(\fatcdot + x) , u_t^{*,r} \rangle } \right) \right] dx \nn\\
  &= \int_0^\infty \EE\!\left[ \1_{\big\{ L_0\big( u_t^{*,r} \big) \geq -x \big\}} \left( 1 - e^{-2 \langle \psi(\fatcdot + x) , u_t^{*,r} \rangle } \right) \right]  dx \nn\\
  &\geq \int_0^\infty \EE\!\left[ \1_{\big\{ -x \leq L_0\big( u_t^{*,r} \big) < -x+1 \big\}} \left( 1 - e^{-2 \langle \1_{[0,1)} \psi , u_t^{*,r}(\fatcdot - x) \rangle } \right) \right] dx. \nn
}
This completes the proof. 
\end{proof}
%
%
%
\subsection{Proof of Proposition~\ref{PRO:exp-alpha-increase}}
\label{SUBSEC:Proof-theta-prop}
%
%
Let $T \geq 1$ be arbitrarily fixed. For $\theta_c < \underline{\theta} < \overline{\theta}$ arbitrary let 
\eqn{
\lbeq{equ:def-delta-epsilon}
  \delta = \frac{ \overline{\theta}-\underline{\theta} }{M}, \mbox{ where } M>0 \mbox{ is arbitrarily large with } MT \in \NN \mbox{ and } \theta_m = \underline{\theta} + m \frac{\delta}{T},\ m \in \{0, 1, \ldots, MT\}.
}
For ease of notation, we only prove the case $\theta_1=\underline{\theta}, \theta_2=\overline{\theta}$. Note that if we let $\delta=(\theta_2-\theta_1)/M$ instead and consider the difference $\alpha_T(\theta_2) - \alpha_T(\theta_1)$ in what follows, the proof remains unchanged.

We proceed to observe that $\theta_0=\underline{\theta}, \theta_{MT}=\overline{\theta}$ and that we therefore rewrite
\eqan{
  \alpha_T\big( \overline{\theta} \big) - \alpha_T\big( \underline{\theta} \big)
  &= \sum_{m=1}^{MT} \left\{ \alpha_T(\theta_m) - \alpha_T(\theta_{m-1}) \right\} \\
  &= \sum_{m=1}^{MT} \frac{2}{T} \int_0^{T/2} \left\{ \EE\!\left[ R_0\big( u_{T/2+s}^{*,l}(\theta_m) \big) \right] - \EE\!\left[ R_0\big( u_{T/2+s}^{*,l}(\theta_{m-1}) \right] \right\} ds. \nn
}
Let $\xi>0$ arbitrary and $S=S(\omega,m)$, $m \in \NN$ with $\xi \leq S \leq T/2-\xi$ be random stopping times to be made more precise later on. Then, by the strong Markov property of the processes involved,
\eqn{
\lbeq{equ:diff-alpha-1}
  \alpha_T\big( \overline{\theta} \big) - \alpha_T\big( \underline{\theta} \big) 
  = \sum_{m=1}^{MT} \frac{2}{T} \int_0^{T/2} \EE\!\left[ \EE_{u_S^{*,l}(\theta_m)}\!\left[ R_0\big( u_{T/2-S+s}(\theta_m) \big) \right] - \EE_{u_S^{*,l}(\theta_{m-1})}\!\left[ R_0\big( u_{T/2-S+s}(\theta_{m-1}) \big) \right] \right] ds. 
}
The expectations are well-defined by Corollaries~\ref{COR:e-bd}--\ref{COR:e-bd-small-t}. Using a \textit{$\theta$-coupling} we bound \eqref{equ:diff-alpha-1} from below by
\eqn{
  \sum_{m=1}^{MT} \frac{2}{T} \int_0^{T/2} \EE\!\left[ \EE_{u_S^{*,l}(\theta_m)}\!\left[ R_0\big(u_{T/2-S+s}(\theta_{m-1}) \big) \right] - \EE_{u_S^{*,l}(\theta_{m-1})}\!\left[ R_0\big(u_{T/2-S+s}(\theta_{m-1}) \big) \right] \right] ds.
}
A shift in space, using the shift invariance of the dynamics, further allows to rewrite this to
\eqan{
  & \sum_{m=1}^{MT} \frac{2}{T} \int_0^{T/2} \EE\!\left[ \EE_{u_S^{*,l}(\theta_m)\big( \fatcdot + R_0\big(u_S^{*,l}(\theta_{m-1}) \big) \big)}\!\left[ R_0\big(u_{T/2-S+s}(\theta_{m-1}) \big) \right] \right. \\
  &\quad\qquad \left. - \EE_{u_S^{*,l}(\theta_{m-1})\big( \fatcdot + R_0\big(u_S^{*,l}(\theta_{m-1}) \big) \big)}\!\left[ R_0\big(u_{T/2-S+s}(\theta_{m-1}) \big) \right] \right] ds. \nn
}

For $S \geq \xi > 0$, use a \textit{$\theta$-$*$-coupling} to obtain 
\eqn{
\lbeq{equ:def-Delta}
  0 \leq
  \Delta_S^{*,l}\big( \theta_{m-1}, \theta_m \big)
  \equiv u_S^{*,l}(\theta_m)\big( \fatcdot + R_0\big(u_S^{*,l}(\theta_{m-1}) \big) \big) - u_S^{*,l}(\theta_{m-1})\big( \fatcdot + R_0\big(u_S^{*,l}(\theta_{m-1}) \big) \big)
  \in \SC_{tem}^+. 
}
Hence, we use the strong Markov property of the family of laws $\PP_f, f \in \SC_{tem}^+$ to apply Lemma~\ref{LEM:claim2}, using that $T/2-S \geq \xi > 0$ and $S \geq \xi>0$, to see that 
\eqan{
  &\alpha_T\big( \overline{\theta} \big) - \alpha_T\big( \underline{\theta} \big) \\
  & \geq \sum_{m=1}^{MT} \frac{2}{T} \int_0^{T/2} \EE\Bigg[ \EE\!\left[ \EE_{u_S^{*,l}(\theta_m)\big( \fatcdot + R_0\big(u_S^{*,l}(\theta_{m-1}) \big) \big)}\!\left[ R_0\big(u_{T/2-S+s}(\theta_{m-1}) \big) \right] \right. \nn\\
  &\quad\qquad \left. - \EE_{u_S^{*,l}(\theta_{m-1})\big( \fatcdot + R_0\big(u_S^{*,l}(\theta_{m-1}) \big) \big)}\!\left[ R_0\big(u_{T/2-S+s}(\theta_{m-1}) \big) \right] \Bigmid \SF_S \right] \Bigg] ds \nn\\
  & \geq \sum_{m=1}^{MT} \frac{2}{T} \int_0^{T/2} \EE\Bigg[ \EE\!\left[ R_0\big( u_{T/2-S+s}^{\big( \Phi + \Delta_S^{*,l}\big( \theta_{m-1}, \theta_m \big) \big)}(\theta_{m-1}) \big) \right] - \EE\!\left[ R_0\big( u_{T/2-S+s}^{(\Phi)}(\theta_{m-1}) \big) \right] \Bigg] ds \nn\\
  & \geq \sum_{m=1}^{MT} \frac{2}{T} \int_0^{T/2} \EE\!\left[ R_0\big( u_{T/2-S+s}^{\big( \Phi +  \Delta_S^{*,l}\big( \theta_{m-1}, \theta_m \big) \big) }(\theta_{m-1}) \big) \vee 0 - R_0\big( u_{T/2-S+s}^{(\Phi)}(\theta_{m-1}) \big) \vee 0 \right] ds. \nn
}
With the help of Lemma~\ref{LEM:claim1} we further bound this from below by
\eqn{
  \sum_{m=1}^{MT} \frac{2}{T} \int_0^{T/2}\!\!\!\int_0^\infty \EE\!\left[ \1_{\big\{ -x \leq L_0\big( u_{T/2-S+s}^{*,r}(\theta_{m-1}) \big) < -x+1 \big\}} \left( 1 - e^{-2 \big\langle \1_{[0,1)} \Delta_S^{*,l}\big( \theta_{m-1}, \theta_m \big) , \big( u_{T/2-S+s}^{*,r}(\theta_{m-1}) \big)(\fatcdot - x) \big\rangle } \right) \right] dx ds.
}

For $d_0, m_0>0$, let
\eqn{
  \tilde{M}(d_0,m_0) = \big\{ f \in \SC_{tem}^+: \mbox{ there exist } 0 \leq l_0 < r_0 \leq 1/2 \mbox{ with } |r_0-l_0|=d_0 \mbox{ such that } f \geq m_0 \1_{[l_0,r_0]} \big\}.
}
Fix $\epsilon>0$ arbitrary and let $d_0=d_0(\epsilon), m_0=m_0(\epsilon)>0$ as in Corollary~\ref{COR:uniform-properties}, where we note that instead of considering right markers we now consider left markers. We obtain as a further lower bound to the above 
\eqan{
\lbeq{equ:before-prop}
  & \sum_{m=1}^{MT} \frac{2}{T} \int_0^{T/2} \int_0^\infty \EE\!\left[ \1_{\big\{ -x \leq L_0\big( u_{T/2-S+s}^{*,r}(\theta_{m-1}) \big) < -x+1 \big\}} \1_{\left\{ u_{T/2-S+s}^{*,r}(\theta_{m-1})\big( \fatcdot + L_0\big( u_{T/2-S+s}^{*,r}(\theta_{m-1}) \big) \big) \in \tilde{M}(d_0,m_0) \right\}} \right. \\
  & \quad \left. \times \left( 1 - e^{-2 \big\langle \1_{[0,1)} \Delta_S^{*,l}\big( \theta_{m-1}, \theta_m \big) , \big( u_{T/2-S+s}^{*,r}(\theta_{m-1}) \big)(\fatcdot - x) \big\rangle } \right) \right] dx ds \nn
}
for all $T \geq 1$. We next make use of the following crucial observation. Recall that $\theta_m-\theta_{m-1}=\delta/T$ for $m \in \{1,\ldots,MT\}$ with $\delta=\big(\overline{\theta}-\underline{\theta}\big)/M$.
%
%
\begin{proposition}
\label{PRO:stopping-time-S}
For all $\xi>0$ and $\varphi \in \SC_{tem}^+$ with $L_0(\varphi) \in (0,1)$ there exist $T_0>0$ big enough and $\rho, C_0, C_1>0$ small enough, all constants only dependent on $\xi,\underline{\theta},\overline{\theta},\varphi$, such that for all $T \geq T_0$ and $m \in \{0, 1, \ldots, MT\}, M \in \NN$ there exist stopping times $\xi \leq S=S(m,\varphi) \leq T/2-\xi$ such that
\eqn{
  \PP\!\left( \big\langle \Delta_{S(m,\varphi)}^{*,l}\big( \theta_{m-1}, \theta_m \big) , \varphi \big\rangle \geq \rho \right) \geq C_0 \big( 1 - \exp( - C_1 \delta ) \big).
} 
\end{proposition}
%

The proof of the proposition follows in Section~\ref{SEC:increase-at-front} below. First, we finish the proof of Proposition~\ref{PRO:exp-alpha-increase}. We obtain as a lower bound to the term in \eqref{equ:before-prop} with $\varphi = m_0 \1_{[1/2,1/2+d_0/2]}$,
\eqan{
  & C_0 \big( 1 - e^{-C_1 \delta} \big) \left( 1 - e^{-2 \rho} \right) \sum_{m=1}^{MT} \frac{2}{T} \int_0^{T/2} \int_0^\infty \EE\!\left[ \1_{\big\{ -x \leq L_0\big( u_{T/2-S+s}^{*,r}(\theta_{m-1}) \big) < -x+1 \big\}}  \right. \\
  & \quad \left. \times \1_{\left\{ u_{T/2-S+s}^{*,r}(\theta_{m-1})\big( \fatcdot + L_0\big( u_{T/2-S+s}^{*,r}(\theta_{m-1}) \big) \big) \in \tilde{M}(d_0,m_0) \right\}} \1_{\left\{ \1_{[0,1)}(\fatcdot) \big( u_{T/2-S+s}^{*,r}(\theta_{m-1}) \big)(\fatcdot - x) \geq m_0 \1_{[1/2,1/2+d_0/2]}(\fatcdot) \right\}} \right] dx ds \nn
}
for all $T \geq T_0$. By definition of $\tilde{M}(d_0,m_0)$, using the Theorem of Fubini-Tonelli, this is bounded from below by
\eqan{
  & C_0 \big( 1 - e^{-C_1 \delta} \big) \left( 1 - e^{-2 \rho} \right) \frac{d_0}{2} \sum_{m=1}^{MT} \frac{2}{T} \int_0^{T/2} \EE\!\left[ \1_{\big\{ L_0\big( u_{T/2-S+s}^{*,r}(\theta_{m-1}) \big) < 0 \big\}} \right. \\
  & \quad \left. \times \1_{\left\{ u_{T/2-S+s}^{*,r}(\theta_{m-1})\big( \fatcdot + L_0\big( u_{T/2-S+s}^{*,r}(\theta_{m-1}) \big) \big) \in \tilde{M}(d_0,m_0) \right\}} \right] ds. \nn
}
By symmetry and Lemma~\ref{LEM:marker-nonneg}, we have for $T$ big enough,
\eqn{
  \frac{2}{T} \int_0^{T/2} \EE\!\left[ \1_{\big\{ L_0\big( u_{T/2-S+s}^{*,r}(\theta_{m-1}) \big) < 0 \big\}} \right] ds 
  \geq \tilde{\delta}/2>0
}
with $\tilde \delta$ as in Lemma~\ref{LEM:marker-nonneg}. Recall the definition of $\big( \nu_T^{*,l}(\theta) \big)$ from \eqref{equ:def-nu-star-lr}. We conclude using Corollary~\ref{COR:uniform-properties} and symmetry that
\eqan{
  \alpha_T\big( \overline{\theta} \big) - \alpha_T\big( \underline{\theta} \big) 
  &\geq C_0 \big( 1 - e^{-C_1 \delta} \big) \left( 1 - e^{-2 \rho} \right) \frac{d_0}{2} \sum_{m=1}^{MT} (\tilde{\delta}/2-\epsilon) \\
  &= C_0 \big( 1 - e^{-C_1 \delta} \big) \left( 1 - e^{-2 \rho} \right) \frac{d_0}{2} (\tilde{\delta}/2-\epsilon) M T \nn
}
for all $T \geq T_0$. Choose $\epsilon$ small enough and recall that $\delta = \big( \overline{\theta}-\underline{\theta} \big) / M$ respectively $M = \big( \overline{\theta}-\underline{\theta} \big) / \delta$ to conclude that
\eqn{
  \frac{\alpha_T\big( \overline{\theta} \big) - \alpha_T\big( \underline{\theta} \big)}{T}
  \geq C_0 \left( 1 - e^{-2 \rho} \right) \frac{d_0}{2} (\tilde{\delta}/2-\epsilon) \big( 1 - e^{-C_1 \big( \overline{\theta} - \underline{\theta} \big) / M} \big) M.
}
Let $M \rightarrow \infty$ to obtain $C_0 \left( 1 - e^{-2 \rho} \right) \frac{d_0}{2} (\tilde{\delta}/2-\epsilon) \big( \overline{\theta} - \underline{\theta} \big)$ as a lower bound to the left hand side. This finishes the proof.
%
%
%
\subsection{Proof of Proposition~\ref{PRO:alpha}}
\label{SUBSEC:Proof-result-1}
%
%
Lemma~\ref{LEM:alpha_lim} yields the existence of the limit $B=B(\theta)$. Its positivity follows from Corollary~\ref{COR:Bg0}. Combine Proposition~\ref{PRO:exp-alpha-increase} and Corollary~\ref{COR:alpha_lim} to obtain \eqref{equ:alpha-incr} by taking $T \rightarrow \infty$. This concludes the proof.
%
%
\section{Proof of Proposition~\ref{PRO:stopping-time-S}}
\label{SEC:increase-at-front}
%
%
In this section we prove Proposition~\ref{PRO:stopping-time-S}. We start out by giving the main idea of the proof.

%
\subsection{Idea of proof}
\label{SUBSEC:strategy}
%
%
Let $T_1, T_2>0$, $m \in \NN$ and $\xi>0$ be arbitrarily fixed. Let $\theta_m$ as in \eqref{equ:def-delta-epsilon} and suppose that $\theta_c < \underline{\theta} \leq \theta_{m-1}<\theta_m \leq \overline{\theta}$. For $t \in [\xi,T/2-\xi-T_1-T_2]$ fixed, on a time-interval of length $T_1+T_2$, we look for a (random) point in time $S=S(m) \in [t,t+T_1+T_2]$ such that 
\eqn{
\lbeq{equ:S-goal}
  \1_{[0,1)} \Delta_S^{*,l}\big( \theta_{m-1}, \theta_m \big) \geq \rho \1_{[0,1)},
}
where
\eqn{
 \Delta_S^{*,l}\big( \theta_{m-1}, \theta_m \big)
  \equiv u_S^{*,l}(\theta_m)\big( \fatcdot + R_0\big(u_S^{*,l}(\theta_{m-1}) \big) \big) - u_S^{*,l}(\theta_{m-1})\big( \fatcdot + R_0\big(u_S^{*,l}(\theta_{m-1}) \big) \big)
}
as in \eqref{equ:def-Delta}.

We investigate the difference between the solutions $u^{*,l}(\theta_m)$ and $u^{*,l}(\theta_{m-1})$ over time with the goal of finding $S$ such that \eqref{equ:S-goal} holds. For $t$ fixed as above, condition on $\SF_t$. Aside from the shift in space, by monotonicity, the difference on the time-interval $[t,t+T_1+T_2]$ is greater or equal to the difference of solutions $u^{*,l}_{t+\fatcdot}(\theta_m)$ and $u^{*,l}_{t+\fatcdot}(\theta_{m-1})$ with common initial condition $u^{*,l}_t(\theta_{m-1})$ at time $t$.

\textit{First step:} Start out with density $u^{*,l}_t(\theta_{m-1})$. Use a time-interval of length $T_1$ to gain additional mass-density $v_{T_1}$ of height of order $O(\epsilon)$, $\epsilon \equiv \theta_m-\theta_{m-1}$ on the support of $u_t^{*,l}(\theta_{m-1})$ with probability of order $O(1)$. This amount is due to an immigration term of order $\epsilon u_{t+s}^{*,l}(\theta_{m-1})$, $s \in [0,T_1]$ in a \textit{$\theta$-coupling}. For $T_1$ not too big, the mass created, $v_s, s \in [0,T_1]$ remains small and immigration dominates the annihilation term of order $v_s$. 

\textit{Second step:} Use a \textit{monotonicity-coupling} to compare the original solution $u_{t+T_1+\fatcdot}^{*,l}(\theta_{m-1})$ with parameter $\theta_{m-1}$ for a time-interval of length $T_2$ with a solution with the same parameter $\theta_{m-1}$ but with mass $v_{T_1}$ (cf. \textit{first step}) added to the initial condition (at time $t+T_1$). With probability of order $O(\epsilon)$ the mass $v_{T_1}$ gets a constant distance and an amount of mass $O(1)$ in front of the original solution $u^{*,l}_{t+T_1+\fatcdot}(\theta_{m-1})$ after a time-period of length $T_2$. To be more precise, we use this time-period of length $T_2$ twofold. Firstly, we show that the mass stays "ahead" with probability of order $O(\epsilon)$ and secondly, that if it stays "ahead", then it has acquired a size of order $O(1)$ at the front. 

We now give the \textit{mathematical framework} for the coupling-techniques mentioned above. Let $T_1>0$ be arbitrarily fixed and $\theta_c < \underline{\theta} \leq \theta_1 < \theta_2 \leq \overline{\theta} < \infty$ with $\epsilon \equiv \theta_2-\theta_1>0$. For $u_0 \in \CP(\SC_{tem}^+)$ fixed, use a \textit{$\theta$-coupling} to construct solutions $u_s^{(1)}(x) = u_s(\theta_1)(x), u_s^{(2)}(x) = u_s(\theta_2)(x), 0 \leq s \leq T_1$ to \eqref{equ:SPDE} such that $u_s^{(1)}(x) \leq u_s^{(2)}(x)$ for all $s \in [0,T_1], x \in \RR$ a.s. solve
\eqn{
\lbeq{equ:increase-step-1}
  u_s(\theta_2)(x) = u_s(\theta_1)(x) + v_s(x) \mbox{ with } v_s(x) \geq 0 \mbox{ for all } s \in [0,T_1], x \in \RR \mbox{ a.s. }
}
and $v$ as in \eqref{equ:SPDE-v-theta-coupling}, that is, \textit{conditional on $\sigma(u_s(\theta_1)): 0 \leq s \leq T_1)$}, $v$ has distribution $Q^{0,(\theta_2-\theta_1)u(\theta_1),2u(\theta_1),1}(\theta_2)$ (cf. \eqref{equ:SPDE-Tribe} and Theorem~\ref{THM:tribe}) on $[0,T_1]$.

Let $T_2>0$ be arbitrarily fixed. Extend the above coupling to include a process $(w_s)_{s \in [0,T_1+T_2]}$ such that
\eqn{
\lbeq{equ:u1wu2}
  u_s(\theta_1)(x) \leq w_s(x) \leq u_s(\theta_2)(x) \mbox{ for all } s \in [0,T_1+T_2], x \in \RR \mbox{ a.s.}
}
as follows. Set
\eqn{
\lbeq{equ:def-v-w}
  w_s(x) \equiv u_s(\theta_1)(x) + v_s(x) \equiv
  \begin{cases}
    u_s^{(u_0)}(\theta_2)(x) = u_s^{(u_0)}(\theta_1)(x) + v_s(x) & \mbox{ for } 0 \leq s \leq T_1, \cr 
    u_{s-T_1}^{(w_{T_1})}(\theta_1)(x) = u_{s-T_1}^{\big( u_{T_1}^{(u_0)}(\theta_1) + v_{T_1} \big)}(\theta_1)(x) & \mbox{ for } s \geq T_1. 
  \end{cases}
}
That is, \textit{conditional on $\SF_{T_1}$}, $w_{T_1+\fatcdot}$ has distribution $Q^{w_{T_1},0,0,1}(\theta_1)$. Indeed, to construct the coupling for the case $s>T_1$, condition on $\SF_{T_1}$ and use a combination of a \textit{monotonicity-coupling} and a \textit{$\theta$-coupling}. To be more precise, use a \textit{monotonicity-coupling} based on two independent white noises $W_1, W_2$ to construct 
\eqn{
\lbeq{equ:v-coupling-part-2}
  u_{T_1+r}^{(u_0)}(\theta_1)(x) = u_r^{(u_{T_1}(\theta_1))}(\theta_1)(x) \leq u_r^{(w_{T_1})}(\theta_1)(x) \equiv u_r^{(u_{T_1}(\theta_1))}(\theta_1)(x) + v_{T_1+r}(x) \mbox{ for all } r \geq 0, x \in\RR
}
almost surely, with $v_{T_1+\fatcdot}$ solving \eqref{equ:SPDE-v-monotonicity-coupling}. Then use a \textit{$\theta$-coupling} to obtain 
\eqn{
  u_r^{(w_{T_1})}(\theta_1)(x) \leq u_r^{(w_{T_1})}(\theta_2)(x) \equiv u_r^{(w_{T_1})}(\theta_1)(x) + \hat{v}_{T_1+r}(x) \mbox{ for all } r \geq 0, x \in \RR
}
almost surely, where the difference process $\hat{v}$ solves \eqref{equ:SPDE-v-theta-coupling} with a white noise $W_3$ independent of $W_1, W_2$ from above. As a result,
\eqan{
  u_{T_1+r}^{(u_0)}(\theta_1)(x)
  &\leq u_{T_1+r}^{(u_0)}(\theta_1)(x) + v_{T_1+r}(x) = w_{T_1+r}(x) \\
  &\leq u_{T_1+r}^{(u_0)}(\theta_1)(x) + v_{T_1+r}(x) + \hat{v}_{T_1+r}(x) 
  = u_r^{(w_{T_1})}(\theta_2)(x)
  = u_r^{(u_{T_1}^{(u_0)}(\theta_2))}(\theta_2)(x) = u_{T_1+r}^{(u_0)}(\theta_2)(x) \nn
}
holds indeed true.
%
%
\subsection{A first estimate}
%
%

The following estimate is fundamental in the \textit{first step} of the construction. Recall \eqref{equ:increase-step-1} and thus compare the following SPDE with \eqref{equ:SPDE-v-theta-coupling} from the \textit{$\theta$-coupling} which quantifies the gain in density due to an increase in $\theta$. 

Let 
\eqn{
  \Upsilon = \big\{ f \in \SC(\RR,\RR): \parallel f \parallel_\lambda = \sup\{ |f(x)| \exp(-\lambda|x|): x \in \RR \} < \infty \mbox{ for some } \lambda<0 \big\}
}
be the set of continuous functions with exponential decay. For existence and uniqueness of solutions to all of the SPDEs mentioned in the proof below, see Theorem~\ref{THM:tribe}.  

Also let
\eqn{
\lbeq{equ:test-functions}
  \tilde{\Upsilon} \equiv \Big\{ \psi \in \SC^{1,2} \mbox{ and } \sup_{t \in [0,T]} |\psi_t(\fatcdot)| \wedge \Big| \frac{\partial \psi_t(\fatcdot)}{\partial t} \Big| \wedge |\Delta\psi_t(\fatcdot)| \in \Upsilon \Big\}.
} 
%
%
\begin{lemma}
\label{LEM:a-first-estimate}
Let $T>0, \zeta \in \SC([0,T],\SC_{tem}^+) \backslash \{0\}$, $W$ a white noise and $\epsilon>0$ be arbitrarily fixed. Let $v=v(\epsilon,\theta,\zeta)$ be a solution to
\eqn{
\lbeq{equ:equ-for-v}
  \frac{\partial v}{\partial t} = \Delta v + \epsilon \zeta + (\theta -v - 2 \zeta) v + \sqrt{v} \dot{W}, \qquad v(0)=0, \quad t \in [0,T].
}
For $g \in \Upsilon, g \geq 0, g \not\equiv 0$ fixed, 
\eqn{
\lbeq{equ:expo-exp}
  0 
  < c_-(\zeta,g,\theta,T) = \liminf_{\epsilon \downarrow 0^+} \frac{ \EE\!\left[ 1 - e^{-2\langle v_T , g \rangle } \right] }{\epsilon}
  \leq \limsup_{\epsilon \downarrow 0^+} \frac{ \EE\!\left[ 1 - e^{-2\langle v_T , g \rangle } \right] }{\epsilon} = c_+(\zeta,g,\theta,T)
  < \infty 
}
holds true.
\end{lemma}
%
%
\begin{proof}
\textit{The lower bound.} 
Fix $\zeta, g, \theta$ and $T$ as above. Let 
\eqn{
  I(\epsilon) 
  = \EE\!\left[ 1 - e^{-2\langle v_T , g \rangle } \right].
}
Subsequently, dominate $v=v(\epsilon)$ by the sum of two independent solutions (cf. the construction of the \textit{coupling with two independent processes} in Remark~\ref{RMK:coupling-independent} below) satisfying
\eqn{
  \frac{\partial v^{(i)}}{\partial t} 
  = \Delta v^{(i)} + \tfrac{\epsilon}{2} \zeta + \big( \theta - v^{(i)} - 2\zeta) v^{(i)} + \sqrt{v^{(i)}} \dot{W}_i, \quad v^{(i)}(0) = 0, \quad i=1,2, \quad t \geq 0
}
such that $v(t,x) \leq v^{(1)}(t,x) + v^{(2)}(t,x)$ for all $t \geq 0, x \in \RR$ a.s. Note that $v^{(i)} = v^{(i)}(\epsilon) \stackrel{\SD}{=} v(\epsilon/2), i=1,2$. We obtain by the independence and the identical distribution of the two non-negative solutions, for all $\epsilon, T>0$,
\eqn{
\lbeq{equ:I-epsilon}
  I(\epsilon)
  \leq \EE\!\left[ 1 - e^{-2\langle v^{(1)}_T + v^{(2)}_T , g \rangle } \right] 
  = \EE\!\left[ 1 - e^{-2\langle v^{(1)}_T , g \rangle } \right] \EE\!\left[ 1 + e^{-2\langle v^{(2)}_T , g \rangle } \right] 
  \leq 2 I(\epsilon/2) 
  \iff \left( \tfrac{I(\epsilon)}{\epsilon} \leq \tfrac{I(\epsilon/2)}{\epsilon/2} \right). 
}

By Theorem~\ref{THM:tribe}b), $I(\epsilon)$ is continuous in $\epsilon$. Hence, to establish the lower bound, it is enough to show that there exists $\epsilon_0>0$ such that $I(\epsilon)>0$ for all $\epsilon \in [\epsilon_0,2\epsilon_0]$. Indeed, by the continuity of $I$ and \eqref{equ:I-epsilon}, it then follows that
\eqn{
\lbeq{equ:I-epsilon-12}
  \inf_{\epsilon \in (0,2\epsilon_0]} \tfrac{I(\epsilon)}{\epsilon}
  \geq \inf_{\epsilon \in [\epsilon_0,2\epsilon_0]} \tfrac{I(\epsilon)}{\epsilon} 
  > 0.
} 
By reasoning as for an \textit{immigration-coupling}, it follows that $I(\epsilon)$ is monotonically increasing in $\epsilon$. It is therefore enough to find $\epsilon_0=\epsilon_0(T)>0$ such that $I(\epsilon_0)>0$. By definition of $v$, this holds true for arbitrary $T>0$. Indeed, use for instance Theorem~\ref{THM:tribe}c) to see that with $\PP(v)$ denoting the distribution of $v$, $\PP(v)=Q^{0,\epsilon \zeta, 2\zeta, 1}$ and $Q^{0, \epsilon \zeta, 0 , 0}$ are mutually absolutely continuous on $\SU_{R,T}$ (recall the notation from Theorem~\ref{THM:tribe}) for $R,T>0$ arbitrarily fixed. Here, the law $Q^{0, \epsilon \zeta, 0 ,0}$ is the law of the solution to 
\eqn{
\lbeq{equ:first-est-superp}
  \frac{\partial w}{\partial t} = \Delta w + \epsilon \zeta + \theta w + \sqrt{w} \dot{W}, \qquad w(0)=0, \quad t \geq 0.
}
The latter is a superprocess with immigration and thus satisfies $\PP(\langle w_T , g \rangle > 0) > 0$. 

\textit{The upper bound.}
We now derive the upper bound in \eqref{equ:expo-exp}. Couple a solution $v$ of \eqref{equ:equ-for-v} with a solution $V$ of
\eqn{
  \frac{\partial V}{\partial t} = \Delta V + \epsilon \zeta + \theta V + \sqrt{V} \dot{W_3}, \qquad V(0)=0, \quad t \geq 0,
}
such that $v(t,x) \leq V(t,x)$ for all $t \geq 0, x \in \RR$ a.s. Here, $W_3$ is an appropriate white noise and we use techniques as in Subsection~\ref{SUBSEC:coupling} of the appendix. See the beginning of \cite[Section~2]{T1996} for the theoretical background of what follows. We obtain for test functions $\psi(t,x)=\psi_t(x), t \geq 0, x \in \RR$ satisfying $\psi \in \tilde{\Upsilon}$ and for $t \geq 0$ arbitrary, by an application of It\^o's formula,
\eqn{
\lbeq{exp-mart}
  e^{-2\langle V_t , \psi_t \rangle} 
  = 1 + M_t -2 \int_0^t e^{-2\langle V_s , \psi_s \rangle} \big\{ \epsilon \big\langle \,\zeta_s , \psi_s \big\rangle 
  + \big\langle V_s , \frac{\partial \psi_s}{\partial s} + \Delta\psi_s + \theta \psi_s - \tfrac{2}{2} \psi_s^2 \big\rangle \big\} ds, 
}
where $(M_t)_{t \geq 0}$ is a local martingale with quadratic variation
\eqn{
  \langle M_{\fatcdot} \rangle_t 
  = 4 \int_0^t \int e^{-4\langle V_s , \psi_s \rangle} V_s(x) \psi_s^2(x) dx ds.
}

For $T>0$ fixed and $0 \leq s \leq T$, choose $\psi(s,z) \equiv \Psi(T-s,z)$, where $\Psi(s,x)=\Psi_s(x)$ is the unique non-negative solution to the partial differential equation (PDE)
\eqn{
\lbeq{equ:PDE-Psi}
  \frac{\partial \Psi_s}{\partial s} = \Delta\Psi_s + \theta \Psi_s - \Psi_s^2, \qquad \Psi_0 = g, \qquad 0 \leq s \leq T
}
(cf. Iscoe \cite[Theorem~A of the appendix]{I1986} with $A\psi=\Delta \psi+\theta \psi, g(x)=x^2$ and $\SD(A)=\{ f \in \SC^2(\RR,\RR): f, \Delta f \in \Upsilon$). Then $\psi \in \tilde{\Upsilon}$ and we obtain for $0 \leq t \leq T$,
\eqn{
  e^{-2\langle V_t , \psi_t \rangle} 
  = 1 + M_t -2 \epsilon \int_0^t e^{-2\langle V_s , \psi_s \rangle} \big\langle \,\zeta_s , \psi_s \big\rangle ds.
}
Note that the integral on the right hand side is finite as $\zeta \in \SC([0,\infty),\SC_{tem}^+)$ and $\sup_{s \in [0,T]} |\psi_s(\fatcdot)| \in \Upsilon$. Let $t=T$. In case $(M_t)_{t \in [0,T]}$ is a martingale, take expectations to conclude
\eqn{
  \EE\!\left[ 1 - e^{-2\langle V_T , g \rangle } \right] 
  = 2 \epsilon \int_0^T \EE\!\left[ e^{-2\langle V_s , \psi_s \rangle} \big\langle \,\zeta_s , \psi_s \big\rangle \right] ds. 
}
In case $(M_t)_{t \in [0,T]}$ is only a local martingale, take a sequence of increasing stopping times $\tau_n \uparrow T$ such that $(M_{t \wedge \tau_n})_{t \in [0,T]}$ is a martingale for each $n \in \NN$ fixed. Take expectations and subsequently use dominated convergence to obtain the same conclusion.

The coupling of $v$ and $V$ yields
\eqn{
\lbeq{equ:1-estimate-ub}
  \EE\!\left[ 1 - e^{-2\langle v_T , g \rangle } \right] 
  \leq \EE\!\left[ 1 - e^{-2\langle V_T , g \rangle } \right] 
  = 2 \epsilon \int_0^T \EE\!\left[ e^{-2\langle V_s , \psi_s \rangle} \big\langle \,\zeta_s , \psi_s \big\rangle \right] ds 
  \leq 2 \epsilon \int_0^T \big\langle \,\zeta_s , \Psi_{T-s} \big\rangle ds. 
}
It thus remains to show that $\int_0^T \big\langle \,\zeta_s , \Psi_{T-s} \big\rangle ds < \infty$. The latter follows from the assumption $\zeta \in \SC([0,\infty),\SC_{tem}^+)$ and as $(\psi_s)_{s \in [0,T]}=(\Psi_{T-s})_{s \in [0,T]}$ satisfies \eqref{equ:test-functions}.  
\end{proof}
%
%
\subsection{Increase of the right marker}
%
%
We now follow the strategy as outlined in Subsection~\ref{SUBSEC:strategy}. We start by investigating the increase of the right marker of a solution due to an increase in $\theta$.

Let $f \in \SC_{tem}^+$ with $R_0(f)<\infty$ and $\PP_f(\tau=\infty)=1$. Recall the notation from Subsection~\ref{SUBSEC:strategy}, in particular the definition of $u, v, w=u^{(f)}(\theta_1)+v$ with $u_0=f$ from \eqref{equ:def-v-w}. In what follows, write $u_{\fatcdot}(x) = u_{\fatcdot}^{(f)}(\theta_1)(x)$ and set $\SF_T^u = \sigma(u_t: 0 \leq t \leq T)$ for $T>0$ arbitrary. Note that in the proofs to follow we will often only write $\EE$ or $\PP$ when the context is clear. In the main statements, the indices are kept however. This will allow us to avoid changes in indexing when using duality relations.
%
%
\begin{lemma}
\label{LEM:increase-of-right-marker-L1}
Let $T_1, T_2>0$ be arbitrarily fixed and $\theta_c < \underline{\theta} \leq \theta_1 < \theta_2 \leq \overline{\theta}$. For all $\deltap>0$ there exists $\eta_1=\eta_1\big( \deltap, T_1, T_2, \underline{\theta}, \overline{\theta} \big)>0$ small enough such that 
\eqn{
\lbeq{equ:bound-on-first-moment}
  \int_{\SC_{tem}^+} \PP_f\bigg( \EE_f\!\left[ 0 \vee \big( R_0\big( w_{T_1+T_2} \big) \wedge 1 \big) - 0 \vee \big( R_0\big( u_{T_1+T_2} \big) \wedge 1 \big) \Bigmid \SF^u_{T_1} \right] \geq \eta_1 (\theta_2-\theta_1) \bigg) \big( \nu_T^{*,l}(\theta_1) \big)(df) \geq 1-2\deltap
}
for all $T>1$.
\end{lemma}
%
%
\begin{remark}
\label{RMK:initial-survives}
Note that $\PP_{\nu_T^{*,l}(\theta_1)}(\tau=\infty)=1$ by \eqref{equ:def-nu-star-lr} and Remark~\ref{RMK:ustar-survives}.
\end{remark}
%
%
\begin{proof}
Recall $Q^{f,\alpha,\beta,\gamma} \equiv Q^{f,\alpha,\beta,\gamma}_u(\theta)$ from Notation~\ref{NOTA:tribe} and Theorem~\ref{THM:tribe}. With regards to the following conditional expectation $\EE\!\left[ 0 \vee \big( R_0\big( w_{T_1+T_2} \big) \wedge 1 \big) - 0 \vee \big( R_0\big( u_{T_1+T_2} \big) \wedge 1 \big) \bigmid \SF^u_{T_1} \right]$, recall that for the coupling from \eqref{equ:def-v-w}, the difference $(0 \vee ( R_0(w_t) \wedge 1 )) - (0 \vee ( R_0(u_t) \wedge 1 ))$ is non-negative for all $t \geq 0$ almost surely. Moreover, \textit{conditional on $\SF_{T_1}^u$}, $v$ has law $Q^{0,(\theta_2-\theta_1)u,2u,1}(\theta_2) \equiv Q_v^{0,(\theta_2-\theta_1)u,2u,1}(\theta_2)$ on $[0,T_1]$ and \textit{conditional on $\SF_{T_1}$}, $w$ has law $Q^{u_{T_1}+v_{T_1},0,0,1}(\theta_1) \equiv Q_w^{u_{T_1}+v_{T_1},0,0,1}(\theta_1)$ on $[T_1,T_1+T_2]$. Thus, as the laws $Q^{f,\alpha,\beta,\gamma}$ for $f \in \SC_{tem}^+$ form a strong Markov family by Theorem~\ref{THM:tribe}b), and by \eqref{equ:u1wu2},
\eqan{
  & \EE\!\left[ 0 \vee \big( R_0\big( w_{T_1+T_2} \big) \wedge 1 \big) - 0 \vee \big( R_0\big( u_{T_1+T_2} \big) \wedge 1 \big) \bigmid \SF^u_{T_1} \right] \\
  & = Q_v^{0,(\theta_2-\theta_1)u,2u,1}(\theta_2)\!\left[ Q_w^{u_{T_1}+v_{T_1},0,0,1}(\theta_1)\!\left[ 0 \vee \big( R_0\big( u_{T_2} \big) \wedge 1 \big) \right] \right] - Q_u^{u_{T_1},0,0,1}(\theta_1)\!\left[ 0 \vee \big( R_0\big( u_{T_2} \big) \wedge 1 \big) \right] \mbox{ a.s.} \nn
}
Note that by Remark~\ref{RMK:initial-survives}, $u$ survives almost surely. Recall \eqref{equ:prob-right-marker}--\eqref{equ:exp_right_marker-2} to rewrite 
\eqan{
\lbeq{equ:front-sth-2}
  & \EE\!\left[ 0 \vee \big( R_0\big( w_{T_1+T_2} \big) \wedge 1 \big) - 0 \vee \big( R_0\big( u_{T_1+T_2} \big) \wedge 1 \big) \Bigmid \SF^u_{T_1} \right] \\
  &= \int_0^1 \PP\!\left( R_0\big( w_{T_1+T_2} \big) > x \Bigmid \SF^u_{T_1} \right) dx - \int_0^1 \PP\!\left( R_0\big( u_{T_1+T_2} \big) > x \Bigmid \SF^u_{T_1} \right) dx \nn\\
  &= \int_0^1 \EE\!\left[ 1-e^{-2 \big\langle w_{T_1}(\fatcdot + x) , u_{T_2}^{*,r}(\theta_1) \big\rangle } \bigmid \SF^u_{T_1} \right] dx - \int_0^1 \EE\!\left[ 1-e^{-2 \big\langle u_{T_1}(\fatcdot + x) , u_{T_2}^{*,r}(\theta_1) \big\rangle } \Bigmid \SF^u_{T_1} \right] dx \nn\\
  &= \int_0^1 \EE\!\left[ e^{-2 \big\langle u_{T_1}(\fatcdot + x) , u_{T_2}^{*,r}(\theta_1) \big\rangle } \left( 1 - e^{-2 \big\langle v_{T_1}(\fatcdot + x) , u_{T_2}^{*,r}(\theta_1) \big\rangle } \right) \Bigmid \SF^u_{T_1} \right] dx. \nn
}
Use a \textit{$\theta$-$*$-coupling} to conclude that this is bounded below by
\eqan{
\lbeq{equ:front-sth-3}
  & \int_0^1 \EE\!\left[ e^{-2 \big\langle u_{T_1}(\fatcdot + x) , u_{T_2}^{*,r}( \overline{\theta} ) \big\rangle } \left( 1 - e^{-2 \big\langle v_{T_1}(\fatcdot + x) , u_{T_2}^{*,r}( \underline{\theta} ) \big\rangle } \right) \Bigmid \SF^u_{T_1} \right] dx \\
  & = \int_0^1 \EE\Bigg[ e^{-2 \big\langle u_{T_1}(\fatcdot + x) , u_{T_2}^{*,r}( \overline{\theta} ) \big\rangle } \EE\!\left[ 1 - e^{-2 \big\langle v_{T_1}(\fatcdot + x) , u_{T_2}^{*,r}( \underline{\theta} ) \big\rangle } \Bigmid \sigma\Big( \SF^{u^{*,r}} , \SF^u_{T_1} \Big) \right] \Biggmid \SF^u_{T_1} \Bigg] dx, \nn
}
where we let $\SF^{u^{*,r}} = \sigma\big( u_t^{*,r}( \underline{\theta} ), u_t^{*,r}( \overline{\theta} ): t \geq 0 \big)$. 

Let $\epsilon = \theta_2-\theta_1$. We now randomize the initial condition. Recall $\nu_T^{*,l}(\theta_1)$ as defined in \eqref{equ:def-nu-star-lr} of the appendix. Let $\eta_1,T>0$ be arbitrarily fixed. The quantity we are interested in is
\eqn{
\lbeq{equ:exp-as-int-av}
  I_1
  \equiv \int_{\SC_{tem}^+} \PP_f\bigg( \EE\!\left[ 0 \vee \big( R_0\big( w_{T_1+T_2} \big) \wedge 1 \big) - 0 \vee \big( R_0\big( u_{T_1+T_2} \big) \wedge 1 \big) \Bigmid \SF^u_{T_1} \right] \geq \eta_1 \epsilon \bigg) \big( \nu_T^{*,l}(\theta_1) \big)(df).
}
We get with the help of \eqref{equ:front-sth-2}--\eqref{equ:front-sth-3} as a lower bound to \eqref{equ:exp-as-int-av},
\eqn{
\lbeq{equ:I1-step1}
  \int_{\SC_{tem}^+} \PP_f\Bigg(  \int_0^1 \EE\bigg[ e^{-2 \big\langle u_{T_1}(\fatcdot + x) , u_{T_2}^{*,r}( \overline{\theta} ) \big\rangle } \EE\Big[ 1 - e^{-2 \big\langle v_{T_1}(\fatcdot + x) , u_{T_2}^{*,r}( \underline{\theta} ) \big\rangle } \Bigmid \sigma\Big( \SF^{u^{*,r}} , \SF^u_{T_1} \Big) \Big] \biggmid \SF^u_{T_1} \bigg] dx \geq \eta_1 \epsilon \Bigg) \big( \nu_T^{*,l}(\theta_1) \big)(df).
}
Here we note that $(v_t)_{t \in [0,T_1]}$ solves \eqref{equ:SPDE-v-theta-coupling} respectively \eqref{equ:equ-for-v} with $\theta=\theta_2$ and $(\zeta_t)_{t \in [0,T_1]} = (u_t)_{t \in [0,T_1]} = \big( u_t^{(f)} \big)(\theta_1)_{t \in [0,T_1]}$ and $f$ drawn according to $\nu_T^{*,l}(\theta_1)$. By \eqref{equ:uniform-tightness} and Corollary~\ref{COR:uniform-properties}, for every $\deltap>0$ there exist a compact set $K_{\deltap} \subset \SC_{tem}^+$ and $d_0=d_0(\deltap), m_0=m_0(\deltap)>0$ such that
\eqn{
\lbeq{equ:inf-comp}
  \inf_{\underline{\theta} \leq \theta \leq \overline{\theta}} \big( \nu_T^{*,l}(\theta) \big)(K_{\deltap} \cap M(d_0,m_0)) \geq 1-\deltap \quad \mbox{ for all } T>1
}
with  
\eqn{
\lbeq{equ:def-M}
  M(d_0,m_0) \equiv \big\{ f \in \SC_{tem}^+: \mbox{ there exist } -1/2 \leq l_0 < r_0 \leq 0 \mbox{ with } |r_0-l_0|=d_0 \mbox{ such that } f \geq m_0 \1_{[l_0,r_0]} \big\}
}
for $d_0, m_0>0$. Observe first that $M(d_0,m_0) \cap K_{\deltap}$ is compact in $\SC_{tem}^+$. Indeed, use that if $(f_n)_n \subset M(d_0,m_0) \cap K_{\deltap}$, then there exists a subsequence $(f_{n_k})_k \subset M(d_0,m_0)$ that converges to a limit in $K_{\deltap}$. Let $x_{n_k} = l_{n_k}+d_0/2$ such that $f_{n_k} \geq m_0 \1_{[x_{n_k}-d_0/2, x_{n_k}+d_0/2]}$ and $l_{n_k} \leq x_{n_k} \leq r_{n_k}, |r_{n_k}-l_{n_k}|=d_0, l_{n_k}, r_{n_k} \in [-1/2,0]$. By the compactness of $[-1/2,0]$ there exists a subsequence $x_{n_{k_l}} \rightarrow x_0 \in [-1/2,0]$ for $l \rightarrow \infty$ and as a result, $f_{n_{k_l}}$ converges to a limit in $M(d_0,m_0) \cap K_{\deltap}$. 

Conditional on $\sigma\Big( \SF^{u^{*,r}} , \SF^u_{T_1} \Big)$, we now apply the lower bound of Lemma~\ref{LEM:a-first-estimate} for some $0 \not\equiv g=g_x \in \Upsilon, 0 \leq g \leq u_{T_2}^{*,r}\big( \underline{\theta} \big)(\fatcdot - x)$. Recall that $u_t=u_t(\theta_1)$. From below \eqref{equ:I-epsilon-12} it follows that it is enough to show for $\epsilon_0>0$ arbitrarily fixed that
\eqan{
\lbeq{equ:I1-step2}
  & \inf_{\theta_1 \in \big[ \underline{\theta}, \overline{\theta} \big]} \int_{K_{\deltap} \cap M(d_0,m_0)} \PP_f\Bigg( \int_0^1 \EE\bigg[ e^{-2 \big\langle u_{T_1}(\fatcdot + x) , u_{T_2}^{*,r}( \overline{\theta} ) \big\rangle } \\
  & \qquad\qquad \times \EE\Big[ 1-e^{-2 \big\langle v_{T_1}\big( \epsilon_0, \theta_2, (u_t)_{t \in [0,T_1]} \big) , g_x \big\rangle } \Bigmid \sigma\Big( \SF^{u^{*,r}} , \SF^u_{T_1} \Big) \Big]  \Bigmid \SF^u_{T_1} \bigg]  dx \geq \eta_1 \epsilon_0 \Bigg) \big( \nu_T^{*,l}(\theta_1) \big)(df) \geq 1-2\deltap \nn
}
for $\eta_1$ small enough and $\theta_2=\theta_1+\epsilon_0$. The left hand side in the above can be bounded from below by
\eqan{
\lbeq{equ:I1-step3}
  & (1-\deltap) \inf_{\theta \in \big[ \underline{\theta}, \overline{\theta} \big]} \inf_{f \in K_{\deltap} \cap M(d_0,m_0)} \PP_f\Bigg( \int_0^1 \EE\bigg[ e^{-2 \big\langle u_{T_1}(\fatcdot + x) , u_{T_2}^{*,r}( \overline{\theta} ) \big\rangle } \\
  & \qquad\qquad\qquad \times \EE\Big[ 1-e^{-2 \big\langle v_{T_1}\big( \epsilon_0, \theta_2, (u_t)_{t \in [0,T_1]} \big) , g_x \big\rangle } \Bigmid \sigma\Big( \SF^{u^{*,r}} , \SF^u_{T_1} \Big) \Big] \biggmid \SF^u_{T_1} \bigg]  dx \geq \eta_1 \epsilon_0 \Bigg), \nn
}
where we used \eqref{equ:inf-comp}. 

The map $(f,\alpha,\beta,1) \mapsto Q^{f,\alpha,\beta,1}$ is continuous by Theorem~\ref{THM:tribe}b). Hence, the law $\PP_f(\theta_1)=Q^{f,0,0,1}(\theta_1)=Q^{f,0,(\overline{\theta}-\theta_1),1}(\overline{\theta})$ of $u$ is continuous in $f$ and $\theta_1$. Furthermore, by the continuous mapping theorem, the law of $v(\epsilon_0,\theta_2,u)$, that is $Q^{0,\epsilon_0 u, 2u,1}(\theta_1+\epsilon_0)=Q^{0,\epsilon_0 u, 2u + (\overline{\theta}-\theta_1),1}\big( \overline{\theta}+\epsilon_0\big)$ is also continuous in $f$ and $\theta_1$. As $\big[ \underline{\theta}, \overline{\theta} \big]$ is a compact interval and $M(d_0,m_0) \cap K_{\deltap}$ is compact in $\SC_{tem}^+$, the infimum is attained for some $\theta' \in \big[ \underline{\theta}, \overline{\theta} \big], f' \in M(d_0,m_0) \cap K_{\deltap}$. Let $\theta', f'$ be arbitrarily fixed. The innermost expectation is non-zero almost surely by reasoning as in \eqref{equ:first-est-superp} of the proof of the lower bound in Lemma~\ref{LEM:a-first-estimate}. Let $x \in [0,1]$ be arbitrarily fixed. Then \eqref{equ:laplace-left} and symmetry yield for $u_t=u_t^{(f')}(\theta')$,
\eqn{
  \EE\!\left[ e^{-2 \big\langle u_{T_1}(\fatcdot + x) , u_{T_2}^{*,r}( \overline{\theta} ) \big\rangle } \Bigmid \SF^u_{T_1} \right]
  = \PP\Big( \big\langle \1_{(0,\infty)}(\fatcdot) , u_{T_2}^{( u_{T_1}(\fatcdot + x) )}(\overline{\theta}) \big\rangle = 0 \Bigmid \SF^u_{T_1} \Big)
  = \PP\Big( R_0\big( u_{T_2}^{( u_{T_1}(\fatcdot + x) )}(\overline{\theta}) \big) \leq 0 \Bigmid \SF^u_{T_1} \Big).
}
The latter is non-zero almost surely. Thus, using dominated convergence, we can choose $\eta_1>0$ small enough such that $I_1 \geq (1-\deltap)^2 \geq 1-2\deltap$.
\end{proof}
%
%
\begin{corollary}
\label{COR:in-front-at-least}
Let $T_1, T_2>0$ be arbitrarily fixed and $\theta_c < \underline{\theta} \leq \theta_1 < \theta_2 \leq \overline{\theta}$. For all $\deltap>0$ there exists $\eta_1=\eta_1\big( \deltap, T_1, T_2, \underline{\theta}, \overline{\theta} \big)>0$ small enough such that 
\eqn{
\lbeq{equ:ahead-lower}
  \int_{\SC_{tem}^+} \PP_f\bigg( \PP_f\!\left( R_0\big( w_{T_1+T_2} \big) > R_0\big( u_{T_1+T_2} \big)  \bigmid \SF^u_{T_1} \right) \geq \eta_1 (\theta_2-\theta_1) \bigg) \big( \nu_T^{*,l}(\theta_1) \big)(df) \geq 1-2\deltap
}
for all $T>1$.
\end{corollary}
%
%
\begin{proof}
Use that $\1_{\{X > Y \}} \geq 0 \vee (X \wedge 1) - 0 \vee (Y \wedge 1)$ for $X \geq Y$.
\end{proof}
%
%
\begin{lemma}
Let $T_1, T_2>0$ be arbitrarily fixed and $\theta_c < \underline{\theta} \leq \theta_1 < \theta_2 \leq \overline{\theta}$. For all $\deltap>0$ there exists $\eta_2=\eta_2\big( \deltap, T_1, T_2, \underline{\theta}, \overline{\theta} \big)>0$ big enough such that 
\eqn{
\lbeq{equ:ahead-upper}
  \int_{\SC_{tem}^+} \PP_f\bigg( \PP_f\!\left( R_0\big( w_{T_1+T_2} \big) > R_0\big( u_{T_1+T_2} \big)  \bigmid \SF^u_{T_1} \right) \leq \eta_2 (\theta_2-\theta_1) \bigg) \big( \nu_T^{*,l}(\theta_1) \big)(df) \geq 1-4\deltap
}
for all $T>1$.
\end{lemma}
%
%
\begin{proof}
First note that $w=u+v$ as in \eqref{equ:def-v-w} and thus
\eqn{
  \PP\!\left( R_0\big( w_{T_1+T_2} \big) > R_0\big( u_{T_1+T_2} \big) \Bigmid \SF^u_{T_1} \right)
  = \PP\!\left( R_0\big( v_{T_1+T_2} \big) > R_0\big( u_{T_1+T_2} \big) \Bigmid \SF^u_{T_1} \right).
}
Recall the construction of $v_{T_1+r}, r \in [0,T_2]$ by means of a \textit{monotonicity-coupling} from \eqref{equ:v-coupling-part-2}. Extend this coupling as follows.
\eqan{
  \frac{\partial u}{\partial t} &= \Delta u + \big( \theta_1 - u \big) u + \sqrt{u} \dot{W_1}, \quad && u(T_1)=u_{T_1}^{(u_0)}(\theta_1), \\
  \frac{\partial v}{\partial t} &= \Delta v + \big( \theta_1 - v - 2u \big) v + \sqrt{v} \dot{W_2}, \quad && v(T_1)=v_{T_1}=w_{T_1}-u_{T_1}^{(u_0)}(\theta_1), \nn\\
  \frac{\partial d}{\partial t} &= \Delta d + 2uv +\big( \theta_1 - d -2v \big) d + \sqrt{d} \dot{W_3}, \quad && d(T_1)=0, 
}
$t \geq T_1$, with $W_i, i=1,2,3$ independent white noises. Then $U \equiv v+d$ solves, conditional on $\SF^u_{T_1+T_2}$,
\eqn{
  \frac{\partial U}{\partial t} = \Delta U + (\theta_1-U) U + \sqrt{d} \dot{W}_4, \qquad U(T_1)=v_{T_1}, \qquad t \geq T_1
}
for some white noise $W_4$ independent of $W_1$. By construction, $v_{T_1+t}(x) \leq U_{T_1+t}(x)$ for all $x \in \RR, t \in [0,T_2]$ almost surely and the law of $U$ only depends on $\SF_{T_1}^u$ through the initial condition.

Now reason similarly to \eqref{equ:front-sth-2}--\eqref{equ:front-sth-3} to obtain
\eqan{
  &\PP\!\left( R_0\big( v_{T_1+T_2} \big) > R_0\big( u_{T_1+T_2} \big) \Bigmid \SF^u_{T_1} \right) 
  \leq \PP\!\left( R_0\big( U_{T_1+T_2} \big) > R_0\big( u_{T_1+T_2} \big) \Bigmid \SF^u_{T_1} \right) \\
  &= \PP\bigg( \PP\!\left( R_0\big( U_{T_1+T_2} \big) > R_0\big( u_{T_1+T_2} \big) \Bigmid \SF^u_{T_1+T_2} \right) \biggmid \SF^u_{T_1} \bigg) \nn\\
  &= \EE\Bigg[ \EE\!\left[ 1-e^{-2 \big\langle U_{T_1}\big(\fatcdot + R_0\big( u_{T_1+T_2} \big) \big) , u_{T_2}^{*,r}(\theta_1) \big\rangle } \Bigmid \SF^u_{T_1+T_2} \right] \Biggmid \SF^u_{T_1} \Bigg] \nn\\
  &= \EE\!\left[ 1-e^{-2 \big\langle v_{T_1}\big(\fatcdot + R_0\big( u_{T_1+T_2} \big) \big) , u_{T_2}^{*,r}(\theta_1) \big\rangle } \Bigmid \SF^u_{T_1} \right] \nn\\
  &\leq \EE\!\left[ 1 - e^{-2 \big\langle v_{T_1}\big(\fatcdot + R_0\big( u_{T_1+T_2} \big) \big) , u_{T_2}^{*,r}\big( \overline{\theta} \big) \big\rangle } \Bigmid \SF^u_{T_1} \right] \nn\\
  &= \EE\Bigg[ \EE\!\left[ 1 - e^{-2 \big\langle v_{T_1} , \big( u_{T_2}^{*,r}( \overline{\theta} ) \big)\big(\fatcdot - R_0\big( u_{T_1+T_2} \big) \big) \big\rangle } \Bigmid \sigma\Big( \SF^{u^{*,r}} , \SF^u_{T_1} \Big) \right] \Biggmid \SF^u_{T_1} \Bigg]. \nn
}
In the third equality we used that $U(T_1)=v_{T_1}$.

Use Lemma~\ref{LEM:analogue-49} to obtain that for all $\theta \in \big[ \underline{\theta}, \overline{\theta} \big]$, $T_1, T_2>0, A>0, T \geq 1$,
\eqn{
  \PP_{\nu_T^{*,l}(\theta)}\big( \big| R_0\big( u_{T_1+T_2} \big) \big| \geq A \big)
  \leq \frac{C\big( \underline{\theta}, \overline{\theta}, T_1+T_2 \big)}{A}.
}
By \eqref{equ:uniform-tightness} it follows that for every $\deltap>0$ there exist $A_{\deltap}>0$ big enough and a compact set $K_{\deltap} \subset \SC_{tem}^+$ such that 
\eqn{
\lbeq{equ:I2-compact}
  \inf_{\underline{\theta} \leq \theta \leq \overline{\theta}} \big( \nu_T^{*,l}(\theta) \big)\Big( K_{\deltap} \cap \big\{ \big| R_0\big( u_{T_1+T_2} \big) \big| < A_{\deltap} \big\} \Big) \geq 1-\deltap \quad \mbox{ for all } T>1.
}
We obtain for $I_2 \equiv \int_{\SC_{tem}^+} \PP_f\bigg( \PP\!\left( R_0\big( w_{T_1+T_2} \big) > R_0\big( u_{T_1+T_2} \big)  \bigmid \SF^u_{T_1} \right) \leq \eta_2 (\theta_2-\theta_1) \bigg) \big( \nu_T^{*,l}(\theta_1) \big)(df)$ as in \eqref{equ:ahead-upper},
\eqan{
  I_2 
  \geq & \int_{K_{\deltap}} \PP_f\Bigg( \sup_{ a \in [-A_{\deltap},A_{\deltap}] } \EE\bigg[ \EE\Big[ 1 - e^{-2 \big\langle v_{T_1} , \big( u_{T_2}^{*,r}( \overline{\theta} ) \big)(\fatcdot - a) \big\rangle } \Bigmid \sigma\Big( \SF^{u^{*,r}} , \SF^u_{T_1} \Big) \Big] \biggmid \SF^u_{T_1} \bigg] \leq \eta_2 (\theta_2-\theta_1) \Bigg) \\
  & \hspace{13.cm} \big( \nu_T^{*,l}(\theta_1) \big)(df) - \deltap. \nn
}

By symmetry, Corollaries~\ref{COR:e-bd}--\ref{COR:e-bd-small-t} and the Markov inequality, for $T_2>0$ fixed we can choose $l>0$ big enough such that
\eqn{
  \PP\!\left( \big| L_0\big( u_{T_2}^{*,r}( \overline{\theta} ) \big) \big| \geq l \right) \leq \deltap.
}
Recall \eqref{equ:u1wu2}--\eqref{equ:def-v-w} and use monotonicity to conclude that for $T_1>0$ fixed, for all $r>0, T>1$,
\eqn{
  \PP_{\nu_T^{*,l}(\theta_1)}\!\left( 0 \vee R_0\big( v_{T_1} \big) \geq r \right)
  \leq \PP_{\nu_T^{*,l}(\theta_1)}\!\left( 0 \vee R_0\big( u_{T_1}(\theta_2) \big) \geq r \right)
  \leq \frac{C\big( \underline{\theta}, \overline{\theta}, T_1 \big)}{r}.
}
Thus, for $\deltap>0$ fixed, we can pick $l, r>0$ big enough such that
\eqan{
  I_2 
  \geq& \int_{K_{\deltap}} \PP_f\Bigg( \sup_{ a \in [-A_{\deltap},A_{\deltap}] } \EE\bigg[ \EE\Big[ 1 - e^{-2 \big\langle v_{T_1} , \big( \1_{[-l,r]} u_{T_2}^{*,r}( \overline{\theta} ) \big)(\fatcdot - a) \big\rangle } \Bigmid \sigma\Big( \SF^{u^{*,r}} , \SF^u_{T_1} \Big) \Big] \biggmid \SF^u_{T_1} \bigg] \\
  & \qquad \qquad \quad \leq \eta_2 (\theta_2-\theta_1) \Bigg) \big( \nu_T^{*,l}(\theta_1) \big)(df) - 3\deltap. \nn
}
Now reason as from above \eqref{equ:I1-step2} to below \eqref{equ:I1-step3}, this time using \eqref{equ:1-estimate-ub} from the proof of the upper bound from Lemma~\ref{LEM:a-first-estimate}, to obtain the claim.
\end{proof}
%
%

Recall the following observation for the coupling from \eqref{equ:def-v-w} from the beginning of the proof of Lemma~\ref{LEM:increase-of-right-marker-L1}. Conditional on $\SF^u_{T_1}$, $v$ has law $Q_v^{0,(\theta_2-\theta_1)u,2u,1}(\theta_2)$ on $[0,T_1]$ and conditional on $\SF_{T_1}$, $w$ has law $Q_w^{u_{T_1}+v_{T_1},0,0,1}(\theta_1)$ on $[T_1,T_1+T_2]$. Finally, recall that $w=u+v$.
%
%
\begin{lemma}
\label{LEM:ahead-how-much}
Let $T_1, T_2>0$ be arbitrarily fixed, $\theta_c < \underline{\theta} \leq \theta_1 < \theta_2 \leq \overline{\theta}$. For all $\deltap>0$ there exists $\eta_3=\eta_3\big( \deltap, T_1, T_2, \underline{\theta}, \overline{\theta} \big)>0$ small enough such that 
\eqn{
\lbeq{equ:ahead-how-much}
  \int_{\SC_{tem}^+} \EE_f\!\left[ \1_{\bigg\{ \PP_f\Big( \big\{ R_0\!\left( w_{T_1+T_2} \right) -R_0\!\left( u_{T_1+T_2} \right) \geq \eta_3  \big\} \Bigmid \SF_{T_1}^u \Big) \geq \eta_3^2 \PP_f\Big( \big\{ R_0\!\left( w_{T_1+T_2} \right) > R_0\!\left( u_{T_1+T_2} \right) \big\} \Bigmid \SF_{T_1}^u \Big) \bigg\}} \right] \big( \nu_T^{*,l}(\theta_1) \big)(df) \geq 1-6\deltap 
}
for all $T>1$.
\end{lemma}
%
%
\begin{remark}
Note that by \eqref{equ:ahead-lower}, $\PP\big( \big\{ R_0\!\left( w_{T_1+T_2} \right) > R_0\!\left( u_{T_1+T_2} \right) \big\} \bigmid \SF_{T_1}^u \big) > 0$ almost surely under $\nu_T^{*,l}(\theta_1)$.
%
\end{remark}
%
%
\begin{proof}
Let $X \geq 0$ be a random variable on some probability space $(\Omega,\SF,\tilde\PP)$. Then (cf. \cite[Proof of Lemma~3]{HT2004})
\eqn{
\lbeq{equ:lower-general-bound}
  \tilde{\PP}\!\left( X > \frac{1}{2} \tilde{\EE}[X] \right) \geq \frac{ \big( \tilde{\EE}[X] \big)^2 }{ 4\tilde{\EE}\big[ X^2 \big] }.
}
In the coupling from above, let
\eqn{
  0 \leq X \equiv 0 \vee \left( R_0\big( w_{T_1+T_2} \big) \wedge 1 \right) - 0 \vee \left( R_0\big( u_{T_1+T_2} \big) \wedge 1 \right) \leq 1.
}
In what follows, we make use of regular conditional distributions. For $\PP\big( \big\{ R_0\big( w_{T_1+T_2} \big) > R_0\big( u_{T_1+T_2} \big) \big\} \bigmid \SF_{T_1}^u \big) > 0$, set
\eqn{
  \tilde{\PP}(\{\fatcdot\}) = \frac{\PP\big( \big\{ \fatcdot \big\} \cap \big\{ R_0\big( w_{T_1+T_2} \big) > R_0\big( u_{T_1+T_2} \big) \big\} \bigmid \SF_{T_1}^u \big)}{\PP\big( \big\{ R_0\big( w_{T_1+T_2} \big) > R_0\big( u_{T_1+T_2} \big) \big\} \bigmid \SF_{T_1}^u \big)}.
}

Now apply \eqref{equ:lower-general-bound} to get for $\eta_3 \leq \frac{1}{2} \tilde{\EE}[X]$,
\eqan{
  & \tilde{\PP}\!\!\left( \left\{ R_0\big( w_{T_1+T_2} \big) - R_0\big( u_{T_1+T_2} \big) \geq \eta_3 \right\} \right) \\
  &\geq \tilde{\PP}(X \geq \eta_3) 
  \geq \tilde{\PP}\!\left(X > \frac{1}{2} \tilde{\EE}[X] \right)
  \geq \frac{\big( \tilde{\EE}[X] \big)^2}{4 \tilde{\EE}\big[ X^2 \big]}
  \geq \frac{ \big( \tilde{\EE}[X] \big)^2 }{4} \geq \eta_3^2. \nn
}
Therefore it suffices to show that there exists $\eta_3>0$ small enough such that
\eqn{
  \int_{\SC_{tem}^+} \EE_f\!\left[ \1_{ \big\{ \eta_3 \leq \tilde{\EE}[X]/2 \big\} } \right] \big( \nu_T^{*,l}(\theta_1) \big)(df)
  \geq 1-6\deltap.
}
By \eqref{equ:bound-on-first-moment} and \eqref{equ:ahead-upper} we have for $\eta_3=\eta_1/(2 \eta_2)$,
\eqan{
  & \int_{\SC_{tem}^+} \EE_f\!\left[ \1_{ \big\{ \eta_3 \leq \tilde{\EE}[X]/2 \big\} } \right] \big( \nu_T^{*,l}(\theta_1) \big)(df) \\
  & = \int_{\SC_{tem}^+} \EE_f\!\left[ \1_{\bigg\{ \EE\Big[ 0 \vee ( R_0( w_{T_1+T_2} ) \wedge 1) - 0 \vee ( R_0( u_{T_1+T_2} ) \wedge 1 ) \Bigmid \SF_{T_1}^u \Big] \geq 2 \eta_3 \PP\Big( \big\{ R_0( w_{T_1+T_2} ) > R_0( u_{T_1+T_2} ) \big\} \Bigmid \SF_{T_1}^u \Big) \bigg\}} \right] \big( \nu_T^{*,l}(\theta_1) \big)(df) \nn\\
  & \geq \int_{\SC_{tem}^+} \EE_f\!\left[ \1_{\Big\{ \EE\Big[ 0 \vee ( R_0( w_{T_1+T_2} ) \wedge 1) - 0 \vee ( R_0( u_{T_1+T_2} ) \wedge 1 ) \Bigmid \SF_{T_1}^u \Big] \geq \eta_1 (\theta_2-\theta_1) \Big\}} \right] \big( \nu_T^{*,l}(\theta_1) \big)(df) \nn\\
  &\quad - \int_{\SC_{tem}^+} \EE_f\!\left[ \1_{\bigg\{ \PP\Big( \big\{ R_0( w_{T_1+T_2} ) > R_0( u_{T_1+T_2} ) \big\} \Bigmid \SF_{T_1}^u \Big) \geq \frac{\eta_1 (\theta_2-\theta_1)}{2 \eta_3} \bigg\}} \right] \big( \nu_T^{*,l}(\theta_1) \big)(df) \nn\\
  &\geq 1-2\deltap-4\deltap. \nn
}
\end{proof}
%
%
\begin{lemma}
\label{LEM:ahead-against-function}
Let $\varphi \in \SC_{tem}^+$ with $L_0(\varphi) \in (0,1)$ be arbitrarily fixed. Further let $T_1, T_2>0$ be arbitrarily fixed and $\theta_c < \underline{\theta} \leq \theta_1 < \theta_2 \leq \overline{\theta}$. For all $\deltap>0$ there exists $\eta_4=\eta_4\big( \varphi, \deltap, T_1, T_2, \underline{\theta}, \overline{\theta} \big)>0$ small enough such that 
\eqan{
\lbeq{equ:ahead-against-function}
  &\int_{\SC_{tem}^+} \EE_f\!\left[ \1_{ \bigg\{ \PP_f\Big( \big\{ \big\langle v_{T_1+T_2}\big( \fatcdot + R_0\!\left( u_{T_1+T_2} \right) \big) , \varphi(\fatcdot) \big\rangle \geq \eta_4 \big\} \Bigmid \SF_{T_1}^u \Big) \geq \big(1-e^{-2\eta_4} \big)^2 \PP_f\Big( \big\{ R_0\!\left( w_{T_1+T_2} \right) > R_0\!\left( u_{T_1+T_2} \right) \big\} \Bigmid \SF_{T_1}^u \Big) \bigg\} } \right] \\
  & \hspace{13cm} \big( \nu_T^{*,l}(\theta_1) \big)(df) \geq 1-\deltap \nn
}
for all $T>1$.
\end{lemma}
%
%
\begin{proof}
Let $\varphi \in \SC_{tem}^+$ with $L_0(\varphi) \in (0,1)$ and $\eta_4>0$ be arbitrarily fixed. Note that in what follows we condition first on $\SF^u_{T_1+T_2}$ rather than on $\SF^u_{T_1}$. Next rewrite
\eqan{
\lbeq{equ:rewrite-B}
  &\PP\bigg( \Big\{ \big\langle v_{T_1+T_2}\big( \fatcdot + R_0\big( u_{T_1+T_2} \big) \big) , \varphi(\fatcdot) \big\rangle \geq \eta_4 \Big\} \Bigmid \SF^u_{T_1+T_2} \bigg) \\
  &= \PP\bigg( \Big\{ 1-e^{-2 \big\langle v_{T_1+T_2}\big( \fatcdot + R_0\big( u_{T_1+T_2} \big) \big) , \varphi(\fatcdot) \big\rangle} \geq 1-e^{-2\eta_4} \Big\} \Bigmid \SF^u_{T_1+T_2} \bigg). \nn
}
Recall from \eqref{equ:v-coupling-part-2} that by means of a \textit{monotonicity-coupling}, $u_{T_1+t}(x) = u_t^{( u_{T_1}(\theta_1) )}(\theta_1)(x) \leq w_{T_1+t}(x) = u_t^{(w_{T_1})}(\theta_1)(x) = u_{T_1+t}(x) + v_{T_1+t}(x)$ for $0 \leq t \leq T_2$ with $v$ solving (cf. \eqref{equ:SPDE-v-monotonicity-coupling})
\eqn{
  \frac{\partial v}{\partial t} = \Delta v + \big( \theta_1 - v - 2u \big) v + \sqrt{v} \dot{W_2}, \qquad v(T_1)=w_{T_1}-u_{T_1}, \qquad T_1 \leq t \leq T_1+T_2.
}
By Corollary~\ref{COR:duality} we have
\eqan{
  &\EE\!\left[ 1-e^{-2 \big\langle v_{T_1+T_2}\big( \fatcdot + R_0\big( u_{T_1+T_2} \big) \big) , \varphi(\fatcdot) \big\rangle} \Bigmid \SF^u_{T_1+T_2} \right] 
  = \EE\!\left[ 1-e^{-2 \big\langle v_{T_1+T_2} , \varphi\big( \fatcdot - R_0\big( u_{T_1+T_2} \big) \big) \big\rangle} \Bigmid \SF^u_{T_1+T_2} \right] \\
  &= \EE\!\left[ 1-e^{-2 \big\langle v_{T_1} , z_{T_1+T_2} \big\rangle} \Bigmid \SF^u_{T_1+T_2} \right] \nn
}
where $z$ solves
\eqn{
  \frac{\partial z}{\partial t} = \Delta z + \big( \theta - z - 2u_{2T_1+T_2-\fatcdot} \big) z + \sqrt{z} \dot{W}_3, \qquad z(T_1)=\varphi\big( \fatcdot - R_0\big( u_{T_1+T_2} \big) \big), \qquad T_1 \leq t \leq T_1+T_2,
}
where $W_2, W_3$ are independent white noises. That is, conditional on $\SF_{T_1+T_2}^u$, $(z_t)_{T_1 \leq t \leq T_1+T_2}$ has law $\PP(z) \equiv Q^{z(T_1),0,2u_{2T_1+T_2-\fatcdot},1}$. By Theorem~\ref{THM:tribe}c), $\PP(z)$ and $Q^{z(T_1),0,0,0}$ are mutually absolutely continuous on $\SU_{R,T}$ for $R, T>0$ arbitrarily fixed. The latter is the law of a superprocess with non-zero initial condition and thus is non-zero with positive probability at time $T_1+T_2$. Similarly, $v_{T_1}$ is non-zero with positive probability. 

Now reason as in \eqref{equ:front-sth-3}--\eqref{equ:I1-step3} with the following modifications. Use a \textit{$\theta$-coupling} for $z$ to obtain \eqref{equ:front-sth-3}. Then investigate 
\eqn{
\lbeq{equ:I1-step1*}
  \EE\!\left[ \EE\!\left[ 1-e^{-2 \big\langle v_{T_1} , z_{T_1+T_2}( \underline{\theta} ) \big\rangle} \Bigmid \sigma\!\left( \SF^z, \SF^u_{T_1+T_2} \right) \right] \SF_{T_1}^u \right]
}
instead of the (outer) conditional expectation in \eqref{equ:I1-step1}. Only apply the lower bound of Lemma~\ref{LEM:a-first-estimate} in case $z_{T_1+T_2} \not\equiv 0$. This way we obtain a result in the spirit of \eqref{equ:bound-on-first-moment}. Here we do not require $z_{T_1+T_2}$ respectively $v_{T_1}$ to be non-zero a.s. as we do not have to multiply the (inner) conditional expectation from \eqref{equ:I1-step1*} with a front factor as in \eqref{equ:I1-step1}. 

Note in particular, that the final statement is phrased in terms of conditioning on $\SF_{T_1}^u$. Analogous reasoning to the proof of Lemma~\ref{LEM:ahead-how-much}, using that $L_0(\varphi) \in (0,1)$ and thus $\big\langle v_{T_1+T_2}\big( \fatcdot + R_0\!\left( u_{T_1+T_2} \right) \big) , \varphi(\fatcdot) \big\rangle = 0$ if $R_0\!\left( w_{T_1+T_2} \right) \leq R_0\!\left( u_{T_1+T_2} \right)$, completes the proof.
\end{proof}
%
%
\begin{lemma}
\label{LEM:summary-for-ahead}
Let $\varphi \in \SC_{tem}^+$ with $L_0(\varphi) \in (0,1)$ and $T_1, T_2>0$ be arbitrarily fixed and $\theta_c < \underline{\theta} \leq \theta_1 < \theta_2 \leq \overline{\theta}$. For all $\deltapp>0$ there exists $\eta_5=\eta_5\big( \varphi, \deltapp, T_1, T_2, \underline{\theta}, \overline{\theta} \big)>0$ small enough such that 
\eqn{
\lbeq{equ:ahead-against-function-summary}
  \int_{\SC_{tem}^+} \PP_f\bigg( \PP_f\Big( \Big\{ \big\langle v_{T_1+T_2}( \fatcdot + R_0( u_{T_1+T_2} ) ) , \varphi(\fatcdot) \big\rangle \geq \eta_5 \Big\} \Bigmid \SF_{T_1}^u \Big) \geq \eta_5 (\theta_2-\theta_1) \bigg) \big( \nu_T^{*,l}(\theta_1) \big)(df) \geq 1-\deltapp 
}
for all $T>1$.
\end{lemma}
%
%
\begin{proof}
By Corollary~\ref{COR:in-front-at-least}, with $\nu_T^{*,l}(\theta_1)(df)$-measure of at least $1-2\delta'$,
\eqn{
  \PP_f\bigg( \PP_f\!\left( R_0\big( w_{T_1+T_2} \big) > R_0\big( u_{T_1+T_2} \big)  \bigmid \SF^u_{T_1} \right) \geq \eta_1 (\theta_2-\theta_1) \bigg) 
  \geq 1-2\deltap.
}
Hence, with $\nu_T^{*,l}(\theta_1)(df) \PP_f(d\omega)$-measure of at least $(1-2\deltap)^2 \geq 1-4\deltap$,
\eqn{
\lbeq{equ:cor-plugin}
  \PP_f\!\left( R_0\big( w_{T_1+T_2} \big) > R_0\big( u_{T_1+T_2} \big)  \bigmid \SF^u_{T_1} \right)\!(\omega) \geq \eta_1 (\theta_2-\theta_1).
}
Also, by Lemma~\ref{LEM:ahead-against-function}, with $\nu_T^{*,l}(\theta_1)(df)$-measure of at least $1-\delta'$,
\eqn{
  \EE_f\Bigg[ \1_{ \bigg\{  \PP_f\Big( \big\{ \big\langle v_{T_1+T_2}\big( \fatcdot + R_0\!\left( u_{T_1+T_2} \right) \big) , \varphi(\fatcdot) \big\rangle \geq \eta_4 \big\} \Bigmid \SF_{T_1}^u \Big) \geq \big(1-e^{-2\eta_4} \big)^2 \PP_f\Big( \big\{ R_0\!\left( w_{T_1+T_2} \right) > R_0\!\left( u_{T_1+T_2} \right) \big\} \Bigmid \SF_{T_1}^u \Big) \bigg\} } \Bigg] 
  \geq 1-\deltap.
}
Hence, with $\nu_T^{*,l}(\theta_1)(df) \PP_f(d\omega)$-measure of at least $(1-\deltap)^2 \geq 1-2\deltap$, 
\eqan{
  & \PP_f\big( \big\{ \langle v_{T_1+T_2}( \fatcdot + R_0( u_{T_1+T_2} ) ) , \varphi(\fatcdot) \rangle \geq \eta_4 \big\} \bigmid \SF_{T_1}^u \big)\!(\omega) \\
  & \geq \big(1-e^{-2\eta_4} \big)^2 \PP_f\big( \big\{ R_0( w_{T_1+T_2} ) > R_0( u_{T_1+T_2} ) \big\} \bigmid \SF_{T_1}^u \big)\!(\omega). \nn
}
Together with \eqref{equ:cor-plugin} this yields that with $\nu_T^{*,l}(\theta_1)(df) \PP_f(d\omega)$-measure of at least $1-6\deltap$,
\eqn{
  \PP_f\big( \big\{ \langle v_{T_1+T_2}( \fatcdot + R_0( u_{T_1+T_2} ) ) , \varphi(\fatcdot) \rangle \geq \eta_4 \big\} \bigmid \SF_{T_1}^u \big)(\omega) 
  \geq \big(1-e^{-2\eta_4} \big)^2 \eta_1 (\theta_2-\theta_1). 
}
The claim now follows.
\end{proof}
%
%

\textit{Proof of Proposition~\ref{PRO:stopping-time-S}.}
Let $\varphi \in \SC_{tem}^+$ with $L_0(\varphi) \in (0,1), T_1, T_2,\xi>0$ and $\theta_{m-1},\theta_m$ be arbitrarily fixed. For ease of notation, write $\theta_1, \theta_2$ instead of $\theta_{m-1}, \theta_m$ and set $\epsilon=\theta_2-\theta_1$. By Lemma~\ref{LEM:summary-for-ahead} and the definition of $\nu_T^{*,l}(\theta_1)$ (cf. \eqref{equ:def-nu-star-lr}), for all $\deltapp>0$ there exists $\eta_6>0$ small enough and $T_0>0$ big enough, all constants only dependent on $\varphi, \deltapp,T_1,T_2,\underline{\theta},\overline{\theta}$, such that 
\eqn{
  \frac{1}{T} \int_0^T \PP\Bigg( \bigg\{ \PP_{u_s^{*,l}(\fatcdot + R_0(s))}\Big( \big\langle v_{T_1+T_2}\big( \fatcdot + R_0\!\left( u_{T_1+T_2} \right) \big) , \varphi(\fatcdot) \big\rangle \geq \eta_6 \Big) \geq \eta_6 (\theta_2-\theta_1) \bigg\} \Bigg) ds
  \geq 1-\deltapp 
}
for all $T \geq T_0$ and $\theta_c < \underline{\theta} \leq \theta_1 < \theta_2 \leq \overline{\theta}$. Hence, using Fubini-Tonelli's theorem, there exists a set $\Omega'$ with $\PP(\Omega') \geq 1-\deltapp$, such that for all $\omega\in\Omega'$,
\eqn{
\lbeq{equ:on-T-average}
  \frac{1}{T} \int_0^T \1_{\Big\{ \PP_{u_s^{*,l}(\fatcdot + R_0(s))}\big( \big\langle v_{T_1+T_2}\big( \fatcdot + R_0\!\left( u_{T_1+T_2} \right) \big) , \varphi(\fatcdot) \big\rangle \geq \eta_6 \big) \geq \eta_6 \epsilon \Big\}} ds
  \geq 1-\deltapp. 
}
For all $\omega\in\Omega'$, there exists 
\eqn{
  s_1=s_1(\omega) \equiv \inf\!\bigg\{ s \geq \xi: \PP_{u_s^{*,l}(\fatcdot + R_0(s))}\Big( \big\langle v_{T_1+T_2}\big( \fatcdot + R_0\!\left( u_{T_1+T_2} \right) \big) , \varphi(\fatcdot) \big\rangle \geq \eta_6 \Big) \geq \eta_6 \epsilon \bigg\}
}
satisfying $s_1 \leq T/2-\xi-T_1-T_2$. In case $\Big\langle v^{\big( u_{s_1}^{*,l}(\fatcdot + R_0(s_1)) \big)}_{T_1+T_2}\Big( \fatcdot + R_0\Big( u^{\big( u_{s_1}^{*,l}(\fatcdot + R_0(s_1)) \big)}_{T_1+T_2} \Big) \Big) , \varphi(\fatcdot) \Big\rangle \geq \eta_6$, set $S = s_1+T_1+T_2$ and call this a success. In case of no success, by \eqref{equ:on-T-average}, there exists
\eqn{
  s_2=s_2(\omega) \equiv \inf\!\bigg\{ s \geq s_1+T_1+T_2: \PP_{u_s^{*,l}(\fatcdot + R_0(s))}\Big( \big\langle v_{T_1+T_2}\big( \fatcdot + R_0\!\left( u_{T_1+T_2} \right) \big) , \varphi(\fatcdot) \big\rangle \geq \eta_6 \Big) \geq \eta_6 \epsilon \bigg\}
}
satisfying $s_2 \leq T/2-\xi-T_1-T_2$. Continue as above with $s_2$ instead of $s_1$. By choosing $C>0$ small enough, we can repeat this procedure $\lceil CT \rceil$ times. If the above procedure fails, which can only happen if $\omega \not\in \Omega'$ or $\omega \in \Omega'$ but there was no success in $\lceil CT \rceil$ trials, set $S = T/2-\xi$.

Note that for $\eta,\deltapp \in (0,1)$ arbitrarily fixed, $\max_{x \in [0,\deltapp]} (x+(1-x)\eta) = \deltapp + (1-\deltapp)\eta$. As a result, using the strong Markov property of the family of laws $\PP_f, f \in \SC_{tem}^+$, we get
\eqn{
  \PP\!\left( \nexists S: \xi \leq S \leq T/2-\xi: \, \big\langle \Delta_S^{*,l}\big( \theta_1, \theta_2 \big) , \varphi \big\rangle \geq \eta_6 \right) 
  \leq \deltapp + (1-\deltapp) \left( 1 - \eta_6 \epsilon \right)^{\lceil CT \rceil}.
}
Recall from \eqref{equ:def-delta-epsilon} that $\epsilon = \delta/T$ to conclude that 
\eqn{
  \PP\!\left( \exists S: \xi \leq S \leq T/2-\xi: \, \big\langle \Delta_S^{*,l}\big( \theta_1, \theta_2 \big) , \varphi \big\rangle \geq \eta_6 \right) 
  \geq (1-\deltapp) \left( 1 - \left( 1 - \eta_6 \delta / T \right)^{\lceil CT \rceil} \right).
}
For $T \rightarrow \infty$ this bound approaches $(1-\deltapp) \big( 1 - \exp\big( - C \eta_6 \delta \big) \big) > 0$. This completes the proof of the claim. \qed
%
%
\section{The speed of the right marker}
\label{SEC:speed-vg}
%
%

Note the construction of travelling wave solutions from Theorem~\ref{THM:analogue-16} respectively Remark~\ref{RMK:analogue-16} of the appendix. Let $\nu_T^{*,l} \in \CP(\SC_{tem}^+)$ be given as in \eqref{equ:def-nu-star-lr} and denote any arbitrary subsequential limit of the tight set $\{ \nu_T^{*,l}: T \geq 1 \}$ by $\nu=\nu^{*,l}$ in what follows. This limit yields a travelling wave solution to \eqref{equ:SPDE}. By Proposition~\ref{PRO:analogue-17}, $\nu^{*,l}(\{ f: R_0(f) = 0 \})=1$ and $\PP_{\nu^{*,l}}(u(t) \not \equiv 0)=1$ for all $t \geq 0$. Denote with $\nu^{(u_0)}$ any subsequential limit that is obtained as in Remark~\ref{RMK:analogue-16} for $u_0 \in \mittl$ with analogous properties.

Recall from \cite[Proposition~4.1]{T1996} that for $\theta>\theta_c$ and $(u_t^{(\nu)})_{t \geq 0}$ a travelling wave solution to \eqref{equ:SPDE},
\eqn{
\lbeq{equ:Tribe-tv-speed}
  R_0\big( u_t^{(\nu)} \big)/t \rightarrow A^{(\nu)}=A^{(\nu)}(\omega) \in [0,2\theta^{1/2}] \mbox{ almost surely as } t \rightarrow \infty
}
holds. This convergence also holds in $\SL^1$ if we replace $R_0(u(t))$ by $0 \vee R_0(u(t))$ as we see below. 

In this section we show that the limiting speed of the dominating right marker $R_0\big( u_t^{*,l} \big)$ and that of any travelling wave solution $\nu^{*,l}$ coincide. Moreover, the speed is deterministic, namely it equals $B=B(\theta)$ from Lemma~\ref{LEM:alpha_lim}. We extend this result to right front markers of solutions to \eqref{equ:SPDE} with initial conditions $\psi$ satisfying $\psi \in \mittl^R$, where the convergence is now in probability and $\SL^1$. For the right front marker of a corresponding travelling wave solution we obtain almost sure convergence to $A^{(\nu^{(\psi)})}=B$.
%
%
\begin{lemma}
\label{LEM:wave-cvg-L1}
Let $\theta>\theta_c$. Then, for any $(u_t^{(\nu)})_{t \geq 0}$ a travelling wave solution to \eqref{equ:SPDE}, 
\eqn{
\lbeq{equ:wave-speed-cvg-L1}
  \big( 0 \vee R_0\big( u_t^{(\nu)} \big) \big)/t \rightarrow A^{(\nu)}=A^{(\nu)}(\omega) \in \big[ 0, 2\theta^{1/2} \big] \quad \mbox{ as } t \rightarrow \infty \mbox{ in } \SL^1. 
}
Moreover, $\EE\big[ A^{(\nu)} \big] \leq B$.
\end{lemma}
%
%
\begin{proof}
By \eqref{equ:Tribe-tv-speed}, for all $N \in \NN$, $\big( \big( 0 \vee R_0\big( u_t^{(\nu)} \big) \big) \wedge (Nt) \big)/t \rightarrow A^{(\nu)} \wedge N$ almost surely for $t \rightarrow \infty$. By dominated convergence, 
\eqn{
\lbeq{equ:wave-speed-cvg-L1-bounded}
  \big( \big( 0 \vee R_0\big( u_t^{(\nu)} \big) \big) \wedge (Nt) \big)/t \rightarrow A^{(\nu)} \wedge N \quad \mbox{ as } t \rightarrow \infty \mbox{ in } \SL^1.
}
By \cite[(2.33)]{K2017} and Lemma~\ref{LEM:marker-pos-part-second-moment}, for $m>0$ arbitrary,
\eqn{
\lbeq{equ:L2-bd-wave}
  \PP\!\left( R_0\big( u_t^{(\nu)} \big) / t \geq m \right)
  \leq C m^{-2}
}
uniformly in $t \geq 1$. In combination with \eqref{equ:Tribe-tv-speed} this gives $\PP(A^{(\nu)} \geq 2m) \leq Cm^{-2}$. Thus $A^{(\nu)} \in \SL^1$. For $N \in \NN$ arbitrary, we get
\eqn{
  \frac{ \EE\!\left[ \big( 0 \vee R_0\big( u_t^{(\nu)} \big) \big) \wedge (Nt) \right] }{t}
  \leq \frac{ \EE\!\left[ 0 \vee R_0\big( u_t^{(\nu)} \big) \right] }{t}
  \leq \frac{ \EE\!\left[ \big( 0 \vee R_0\big( u_t^{(\nu)} \big) \big) \wedge (Nt) \right] }{t} + \int_N^\infty Cm^{-2} dm
}
and similarly, $| \EE\big[ A^{(\nu)} \wedge N \big]-\EE\big[ A^{(\nu)} \big] | \leq 2C/N$. Hence,
\eqn{
  \lim_{t \rightarrow \infty} \EE\!\left[ \left| \frac{ 0 \vee R_0\big( u_t^{(\nu)} \big) }{t} - A^{(\nu)} \right| \right]
  \leq \lim_{t \rightarrow \infty} \EE\!\left[ \left| \frac{\big( 0 \vee R_0\big( u_t^{(\nu)} \big) \big) \wedge (Nt) }{t} - A^{(\nu)} \wedge N \right| \right] + C/N + 2C/N = 3C/N
}
and the first claim follows after taking $N \rightarrow \infty$.

Moreover, for all $N \in \NN$, using once more \cite[(2.33)]{K2017} and the $\SL^1$-convergence of the first claim,
\eqan{
  \EE\big[ A^{(\nu)} \big]
  \leq \EE\big[ A^{(\nu)} \wedge N \big] + 2C/N
  &= \lim_{T \rightarrow \infty} \frac{ \EE\!\left[ \big( 0 \vee R_0\big( u_T^{(\nu)} \big) \big) \wedge (NT) \right] }{T} + 2C/N \\
  &\leq \lim_{T \rightarrow \infty} \frac{ \EE\!\left[ 0 \vee R_0\big( u_T^{*,l} \big) \right] }{T}  + 2C/N
  \leq B + 2C/N. \nn
}
Take $N \rightarrow \infty$ to conclude that $\EE\big[ A^{(\nu)} \big] \leq B$.
\end{proof}
%
%

Recall from Lemma~\ref{LEM:alpha_lim} and Corollary~\ref{COR:Bg0} that 
\eqn{
\lbeq{equ:def-B}
  \lim_{T \rightarrow \infty} \frac{ \EE\!\left[ R_0\big( u_T^{*,l} \big) \right] }{T}
  = \inf_{T \geq 1} \frac{ \EE\!\left[ R_0\big( u_T^{*,l} \big) \right] }{T} 
  \equiv B
  \in (0,\infty).
}
Then
\eqn{
  B = \lim_{T \rightarrow \infty} \frac{ \EE\!\left[ 0 \vee R_0\big( u_T^{*,l} \big) \right] }{T}
}
holds as well. Indeed, let $(u_t^{(\nu)})_{t \geq 0}$ be an arbitrary travelling wave solution with $R_0(\nu)=0$ almost surely. By Corollary~\ref{COR:exp-ustarl-bd-below} and \cite[(2.33)]{K2017}, \eqref{equ:Tribe-tv-speed}, for $M \in \NN$ arbitrary,
\eqan{
  \lim_{T \rightarrow \infty} \EE\!\left[ 0 \vee \big(-R_0\big( u_T^{*,l} \big) \big) \right]/T
  &\leq \lim_{T \rightarrow \infty} \EE\!\left[ \big( 0 \vee \big(-R_0\big( u_T^{*,l} \big) \big) \big) \wedge (MT) \right]/T + C_1e^{-C_2 M} \\
  &\leq \lim_{T \rightarrow \infty} \EE\!\left[ \big( 0 \vee \big(-R_0\big( u_T^{(\nu)} \big) \big) \big) \wedge (MT) \right]/T + C_1e^{-C_2 M} \nn\\
  &= \EE[0 \vee (-A^{(\nu)} \wedge M)] + C_1e^{-C_2 M} 
  = C_1e^{-C_2 M}. \nn
}
Take $M \rightarrow \infty$ and the claim follows.
%
%
\begin{lemma}
\label{LEM:EA-B}
Let $\theta>\theta_c$. With the notations from above, $\EE\big[ A^{(\nu^{*,l})} \big]=B$ holds true. Moreover, if $\psi \in \mittl^R$ satisfies $R_0\big( u_T^{(\psi)} \big)/T \rightarrow B$ for $T \rightarrow \infty$ in $\SL^1$, then $\EE\big[ A^{(\nu^{(\psi)})} \big]=B$ holds as well.
\end{lemma}
%
%
\begin{remark}
Note that we will show at the end of this section (Corollary~\ref{COR:speed-mittl}) that every $\psi \in \mittl$ indeed satisfies the above assumption.
\end{remark}
%
%
\begin{proof}
Let $\nu=\nu^{*,l}$ or $\nu=\nu^{(\psi)}$ with $\psi$ as above. By Lemma~\ref{LEM:wave-cvg-L1}, $\EE\big[ A^{(\nu)} \big] \leq B$. It remains to show that $\EE\big[ A^{(\nu)} \big] \geq B$. In what follows, we provide a proof in case $\nu=\nu^{*,l}$. The proof in case $\nu=\nu^{(\psi)}$ is analogous except for the changes indicated below.

Fix $T \geq 1$ arbitrary. Review the definitions and comments in \cite[(5.18)--(5.20)]{K2017}. Note in particular that for fixed $N \in \NN$ and $m_0 > 0$, $R_{m_0}^N$ is a continuous function on $\SC_{tem}^+$ with $|R_{m_0}^N(f)| \leq N$ and $R^{m_0,N}(f) \leq R_{m_0}^N(f) \leq R_0(f)$ on $\{ R_0(f) \geq -N \}$ and $R_{m_0}^N(f)=-N$ on $\{ R_0(f) < -N \}$. Hence, for $m_0=m_0(T), N=N(T)$, by Lemma~\ref{LEM:wave-cvg-L1},
\eqn{
  \EE\big[ A^{(\nu^{*,l})} \big]
  = \lim_{T \rightarrow \infty} \EE_{\nu^{*,l}}\!\left[ 0 \vee R_0(u_T) \right] / T
  \geq \lim_{T \rightarrow \infty} \EE_{\nu^{*,l}}\!\left[ 0 \vee R_{m_0}^N(u_T) \right] / T.
}
By the definition of tightness respectively weak convergence and the continuity of $f \mapsto \PP_f$, \eqref{equ:def-nu-star-lr} yields for $\nu_{T_n}^{*,l} \Rightarrow \nu^{*,l}$ ($n \rightarrow \infty$),
\eqan{
\lbeq{equ:speed-orig-versus-wave}
  & \EE_{\nu^{*,l}}\!\left[ 0 \vee R_{m_0}^N(u_T) \right] 
  = \lim_{n \rightarrow \infty} \frac{1}{T_n} \int_0^{T_n} \EE\!\left[ 0 \vee R_{m_0}^N\big(u_{s+T}^{*,l}\big(\fatcdot + R_0\big(u_s^{*,l}\big) \big) \big) \right] ds \\
  &\geq \liminf_{n \rightarrow \infty} \left\{ \frac{1}{T_n} \int_0^{T_n} \EE\!\left[ R_0\big(u_{s+T}^{*,l}\big(\fatcdot + R_0\big(u_s^{*,l}\big) \big) \big) \right] ds \right. \nn\\
  &\qquad \left. - \frac{1}{T_n} \int_0^{T_n} \EE\!\left[ 0 \vee \big\{ R_0\big(u_{s+T}^{*,l}\big(\fatcdot + R_0\big(u_s^{*,l}\big) \big) \big) - R_{m_0}^N\big(u_{s+T}^{*,l}\big(\fatcdot + R_0\big(u_s^{*,l}\big) \big) \big) \big\} \right] ds \right\} \nn\\
  &\geq \liminf_{n \rightarrow \infty} \left\{ \frac{1}{T_n} \int_0^{T_n} \EE\!\left[ R_0\big(u_{s+T}^{*,l}\big(\fatcdot + R_0\big(u_s^{*,l}\big) \big) \big) \right] ds \right\} \nn\\
  &\qquad - \limsup_{n \rightarrow \infty} \left\{ \frac{1}{T_n} \int_0^{T_n} \EE\!\left[ 0 \vee \big\{ R_0\big(u_{s+T}^{*,l}\big(\fatcdot + R_0\big(u_s^{*,l}\big) \big) \big) - R_{m_0}^N\big(u_{s+T}^{*,l}\big(\fatcdot + R_0\big(u_s^{*,l}\big) \big) \big) \big\} \right] ds \right\} \nn\\
  &\equiv I_1(T)-E_1(T,m_0(T),N(T)) 
  = I_1(T)-E_1(T). \nn
}
We obtain for $I_1$, using Corollaries~\ref{COR:e-bd}--\ref{COR:e-bd-small-t}, that
\eqan{
\lbeq{equ:I1}
  I_1(T)
  &= \liminf_{n \rightarrow \infty} \frac{1}{T_n} \int_0^{T_n} \EE\!\left[ R_0\big(u_{s+T}^{*,l}\big) - R_0\big( u_s^{*,l} \big) \right] ds \\
  &= \liminf_{n \rightarrow \infty} \frac{1}{T_n} \left( \int_{T_n}^{T_n+T} \EE\!\left[ R_0\big(u_s^{*,l}\big) \right] ds - \int_0^T \EE\!\left[ R_0\big(u_s^{*,l}\big) \right] ds \right)
  = T B \nn
}
by \eqref{equ:def-B} respectively the assumption $R_0\big( u_T^{(\psi)} \big)/T \rightarrow B$ for $T \rightarrow \infty$ in $\SL^1$ for $\psi \in \mittl^R$. It therefore remains to show that $\limsup_{T \rightarrow \infty} E_1(T)/T = 0$.

For $\epsilon>0$ arbitrarily fixed, recall the definition of $R^{m_0,N}(f)$ from \cite[(5.18)]{K2017}. Also recall from above that $|R_{m_0}^N(f)| \leq N$ and $R^{m_0,N}(f) \leq R_{m_0}^N(f) \leq R_0(f)$ on $\{ R_0(f) \geq -N \}$ and $R_{m_0}^N(f)=-N$ on $\{ R_0(f) < -N \}$ to obtain that
\eqan{
  E_1(T)
  &\leq \limsup_{n \rightarrow \infty} \frac{1}{T_n} \int_0^{T_n} \EE\!\left[ 2 R_0\big(u_{s+T}^{*,l}\big(\fatcdot + R_0\big(u_s^{*,l}\big) \big) \big) \1_{\left\{ R_0\big(u_{s+T}^{*,l}\big(\fatcdot + R_0\big(u_s^{*,l}\big) \big) \big) > N \right\}} \right] ds \\
  & \quad + \limsup_{n \rightarrow \infty} \frac{1}{T_n} \int_0^{T_n} \EE\!\left[ 2N \1_{\left\{ \big\langle u_{s+T}^{*,l}\big(\fatcdot + R_0\big(u_{s+T}^{*,l}\big) \big) , \1_{[-\epsilon T,\infty)} \big\rangle < m_0 \right\}} \right] ds + \epsilon T \nn\\
  &= 2 \limsup_{n \rightarrow \infty} \EE_{\nu_{T_n}^{*,l}}\!\left[ R_0(u_T) \1_{\big\{ R_0(u_T) > N \big\}} \right] + 2N \limsup_{n \rightarrow \infty} \PP_{\nu_{T_n}^{*,l}}\!\left( \big\langle u_T(\fatcdot + R_0(T)) , \1_{[-\epsilon T,\infty)} \big\rangle < m_0 \right) + \epsilon T. \nn
}
By \cite[(2.33)]{K2017} and Lemma~\ref{LEM:marker-pos-part-second-moment} we choose $N>CT/\epsilon$ such that 
\eqn{
  \EE_{\nu_{T_n}^{*,l}}\!\left[ R_0(u_T) \1_{\big\{ R_0(u_T) > N \big\}} \right] 
  \leq \EE\!\left[ R_0\big(u_T^{*,l} \big) \1_{\big\{ R_0\big(u_T^{*,l}\big) > N \big\}} \right]
  \leq N^{-1} \EE\!\left[ \left( 0 \vee R_0\big(u_T^{*,l}\big) \right)^2 \right] 
  \leq N^{-1} C T^2
  < \epsilon T 
}
for all $n \in \NN$. Finally, by Lemma~\ref{LEM:A-small-mass-at-front} with $a=\epsilon T/2$, we choose $b=b(T)$ and $\tilde{m}=m_0(T)$ small enough such that 
\eqn{
  \PP_{\nu_{T_n}^{*,l}}\!\left( \big\langle u_T(\fatcdot + R_0(T)) , \1_{[-\epsilon T,\infty)} \big\rangle < m_0 \right)
  \leq \frac{\epsilon T}{2N(T)}
}
for all $n \in \NN$. Thus, $E_1(T)/T \leq 4\epsilon$ for all $T \geq 1$ and the claim follows after taking $\epsilon \downarrow 0^+$.
\end{proof}
%
%
\begin{proposition}
\label{PRO:cvg-upper-A-B}
Let $\theta>\theta_c$. Then $A^{(\nu^{*,l})} \equiv \EE\big[ A^{(\nu^{*,l})} \big] = B$ almost surely and
\eqn{
  R_0\big( u_T^{*,l} \big) / T \rightarrow B \mbox{ almost surely as } T \rightarrow \infty.
}
In particular, $A^{(\nu)} \leq B$ almost surely for all $A^{(\nu)}$ as in \eqref{equ:Tribe-tv-speed}.
\end{proposition}
%
%
\begin{proof}
The last claim follows immediately from \cite[(2.33)]{K2017}.

Fix $T_0>0$. By \cite[(2.33)]{K2017} and as $\nu^{*,l}(\{ f: R_0(f) = 0 \})=1$ by Proposition~\ref{PRO:analogue-17}, there exists a coupling such that
\eqn{
  R_0\big( u_{T_0+t}^{( \nu^{*,l} )} \big) \leq R_0\!\left( u_{T_0+t}^{*,l} \right) \quad \mbox{ for all } t \geq 0 \mbox{ almost surely}.
}
By Corollaries~\ref{COR:e-bd}--\ref{COR:e-bd-small-t}, $\EE\big[ \big| R_0\big(u_{T_0}^{*,l} \big) \big| \big] \leq C(T_0)$ and thus, using Lemma~\ref{LEM:wave-cvg-L1} and \eqref{equ:Tribe-tv-speed},
\eqn{
  A^{(\nu^{*,l})} 
  = \lim_{T \rightarrow \infty} \frac{0 \vee R_0\big( u_T^{( \nu^{*,l} )} \big)}{T}
  = \lim_{T \rightarrow \infty} \frac{R_0\big( u_T^{( \nu^{*,l} )} \big)}{T}
  \leq \liminf_{T \rightarrow \infty} \frac{R_0\big( u_T^{*,l} \big)}{T} \quad \mbox{ a.s.,}
}
where the left equality also holds in $\SL^1$.

Note that by reasoning as in the proof of \cite[Proposition~4.1a)]{T1996}, for $T>0$ fixed, once we bound $\limsup_{n \rightarrow \infty} R_0\big(u_{T_0+nT}^{*,l}\big)/nT$, the same bound holds for $\limsup_{t \rightarrow \infty} R_0(u_t^{*,l})/t$ almost surely. We therefore fix $T>0$ and rewrite
\eqan{
  & \frac{1}{nT} R_0\big( u_{T_0+nT}^{*,l} \big) - \frac{1}{nT} R_0\big( u_{T_0}^{*,l} \big) \\
  &= \frac{1}{nT} \sum_{i=1}^n \left( R_0\big( u_{T_0+iT}^{*,l} \big) - R_0\big( u_{T_0+(i-1)T}^{*,l} \big) \right)
  = \frac{1}{nT} \sum_{i=1}^n R_0\Big( u_T^{\big( u_{T_0+(i-1)T}^{*,l}\big( \fatcdot + R_0\big( u_{T_0+(i-1)T}^{*,l} \big) \big) \big)} \Big). \nn
}
Fix $i \in \NN$. By \eqref{equ:coupling-R-left}, there exists a coupling such that
\eqn{
  R_0\Big( u_T^{\big( u_{T_0+(i-1)T}^{*,l}\big( \fatcdot + R_0\big( u_{T_0+(i-1)T}^{*,l} \big) \big) \big)} \Big)
  \leq R_0\big( u_T^{*,l}(i) \big) \ \mbox{ almost surely,}
}
where $\SL(u_T^{*,l}(i))=\SL(u_T^{*,l})$ for all $i \in \NN$. By construction, the $\SL(u_T^{*,l}(i)), i \in \NN$ are independent. Indeed, we show this by induction. Let $u_{T_0+(i-1)T}^{*,l}$ be given. Then $\zeta_1 \geq u_{T_0+(i-1)T}^{*,l}\big( \fatcdot + R_0\big( u_{T_0+(i-1)T}^{*,l} \big) \big)$ is chosen in the construction. Nevertheless, as $\zeta_N(x) \uparrow \infty$ for $x<0$ and $\zeta_N(x)=0$ for $x \geq 0$, the law of $u_T^{*,l}(i)$ conditional on $u_{T_0+(i-1)T}^{*,l}$ remains $\SL(u_T^{*,l})$. Thus
\eqn{
  \frac{1}{nT} R_0\big( u_{T_0+nT}^{*,l} \big) - \frac{1}{nT} R_0\big( u_{T_0}^{*,l} \big)
  \leq \frac{1}{nT} \sum_{i=1}^n R_0\big( u_T^{*,l}(i) \big),
}
where $\big( R_0\big( u_T^{*,l}(i) \big) \big)_{i \in \NN}$ is an i.i.d. sequence of real valued random variables with $R_0\big( u_T^{*,l}(1) \big) \stackrel{\SD}{=} R_0(u_T^{*,l})$. By Corollaries~\ref{COR:e-bd}--\ref{COR:e-bd-small-t}, $R_0(u_T^{*,l}) \in \SL^1$.

By the ergodic theorem (cf. for instance Klenke \cite[Theorems~20.14, 20.16 and Example~20.12]{bK2014}),
\eqn{
  \frac{1}{nT} \sum_{i=1}^n R_0\big( u_T^{*,l}(i) \big) \rightarrow \EE\big[ R_0\big( u_T^{*,l} \big) \big]/T \quad \mbox{ almost surely and in $\SL^1$ for } n \rightarrow \infty.
}
As a result, 
\eqn{
  \limsup_{n \rightarrow \infty} \frac{ R_0\big( u_{T_0+nT}^{*,l} \big) }{ nT } 
  \leq \limsup_{n \rightarrow \infty} \frac{ R_0\big( u_{T_0}^{*,l} \big) }{ nT } + \frac{ \EE\big[ R_0\big( u_T^{*,l} \big) \big] }{ T } = \frac{ \EE\big[ R_0\big( u_T^{*,l} \big) \big] }{ T }
}
for all $T > 0$. Take $T \rightarrow \infty$ to conclude that
\eqn{
  A^{(\nu^{*,l})} 
  \leq \liminf_{T \rightarrow \infty} \frac{R_0\big( u_T^{*,l} \big)}{T}
  \leq \limsup_{T \rightarrow \infty} \frac{R_0\big( u_T^{*,l} \big)}{T}
  \leq \limsup_{T \rightarrow \infty} \frac{\EE\big[ R_0\big( u_T^{*,l} \big) \big]}{T}
  = B = \EE\big[ A^{(\nu^{*,l})} \big]
}
and therefore $\lim_{T \rightarrow \infty} R_0\big( u_T^{*,l} \big) / T = B$ almost surely.
\end{proof}
%
%
\begin{corollary}
\label{COR:as-limit-upper-wave}
Let $\theta>\theta_c$. Then 
\eqn{
  R_0\big( u_T^{( \nu^{*,l} )} \big) / T \rightarrow B \mbox{ almost surely as } T \rightarrow \infty
}
and
\eqn{
  R_0\big( u_T^{*,l} \big) / T \rightarrow B \mbox{ in $\SL^1$ as } T \rightarrow \infty.
}
\end{corollary}
%
%
\begin{proof}
The first claim follows from \eqref{equ:Tribe-tv-speed} and Proposition~\ref{PRO:cvg-upper-A-B}.

By Proposition~\ref{PRO:cvg-upper-A-B} we have $(R_0(u_T^{*,l})/T \wedge N) \vee (-N) \rightarrow B$ for $T \rightarrow \infty$ almost surely for all $N \in \NN, N>B$ fixed. As these random variables are bounded, dominated convergence implies convergence in $\SL^1$. We conclude that for all $N \in \NN, N>B$,
\eqn{
  \limsup_{T \rightarrow \infty} \EE\!\left[ \left| \frac{ R_0\big( u_T^{*,l} \big) }{T} - B \right| \right]
  \leq 2 \limsup_{T \rightarrow \infty} \EE\!\left[ \left| \frac{ R_0\big( u_T^{*,l} \big) }{T} \right| \1_{\left\{ \left| R_0\big( u_T^{*,l} \big) \right| > NT \right\}} \right].
}
The second claim now follows from the bounds on the positive part from Lemma~\ref{LEM:marker-pos-part-second-moment} and on the negative part from \eqref{equ:exp-ustarl-bd-below} for $N \rightarrow \infty$.
\end{proof}
%

Finally, we consider initial conditions $\psi \in \mittl$ with $\mittl$ as in \eqref{equ:def-mittl}. 
%
%
\begin{lemma}
\label{LEM:speed-mittl}
For initial conditions $\psi \in \mittl^R$, $R_0\big( u_T^{(\psi)} \big)/T \stackrel{\SD}{\rightarrow} B$ as $T \rightarrow \infty$. 
\end{lemma}
%
%
\begin{proof}
Without loss of generality we assume that $\psi(x) = \epsilon H_0(x-x_0)$ for some $x_0 \in \RR, \epsilon>0$. Indeed, by definition of $\mittl^R$, for every $\psi \in \mittl$ there exist $x_0 \in \RR, \epsilon>0$ such that $\psi \geq \epsilon H_0(\fatcdot-x_0)$ and $R_0(\psi) \in \RR$. Reason as in \cite[Remark~2.8(ii)]{K2017} to construct a coupling such that for $T_0>0$ arbitrarily fixed, $u^{(\epsilon H_0(\fatcdot-x_0))}(t,x) \leq u^{(\psi)}(t,x) \leq u^{*,l}(t,x-R_0(\psi))$ for all $t \geq T_0, x \in \RR$ almost surely. Then 
\eqn{
\lbeq{equ:R0-comp-1}
  R_0\big(u_t^{(\epsilon H_0(\fatcdot-x_0))} \big)/t \leq R_0\big( u_t^{(\psi)} \big)/t \leq \big( R_0\big( u_t^{*,l} \big) + R_0(\psi) \big) /t \ \mbox{ for all } t \geq T_0 \mbox{ almost surely }
}
and by Proposition~\ref{PRO:cvg-upper-A-B}, 
\eqn{
\lbeq{equ:R0-comp-2}
  \lim_{t \rightarrow \infty} \big( R_0\big( u_t^{*,l} \big) + R_0(\psi) \big) /t = B \ \mbox{ almost surely.}
}
Hence, $R_0\big( u_t^{(\epsilon H_0(\fatcdot-x_0))} \big)/t \stackrel{\SD}{\rightarrow} B$ implies $R_0\big( u_t^{(\psi)} \big)/t \stackrel{\SD}{\rightarrow} B$. By the shift invariance of the dynamics, assume further that $x_0=1$.

Let $c \in \RR$ be arbitrary. For $c>B$, $\lim_{T \rightarrow \infty} \PP\big( R_0\big( u_T^{(\psi)} \big) \geq cT \big) = 0$ follows from \eqref{equ:R0-comp-1}--\eqref{equ:R0-comp-2}. Note that $\psi(x)=\epsilon H_0(x-1) \geq \epsilon \1_{(-\infty,0]}(x)$ for all $x \in \RR$. By \eqref{equ:laplace-left}, symmetry and by the shift invariance of the dynamics, for all $T > 0$, 
\eqan{
  \PP\big( R_0\big( u_T^{(\psi)} \big) > cT \big) 
  &= 1 - \EE\!\left[ e^{-2\langle \psi , u_T^{*,r}(\fatcdot-cT) \rangle} \right] 
  \geq 1 - \EE\!\left[ e^{-2\epsilon \langle \1_{(-\infty,0]} , u_T^{*,r}(\fatcdot-cT) \rangle} \right] \\
  &\geq 1 - \PP\!\left( \big\langle \1_{(-\infty,0]} , u_T^{*,r}(\fatcdot-cT) \big\rangle < N \right) - e^{-2 \epsilon N} \nn
}
for all $N \in \NN$. Suppose $c<B$. Let $\delta>0$ and choose $N$ big enough such that $e^{-2 \epsilon N} < \delta$. As we will show below, for $N$ fixed, 
\eqn{
\lbeq{equ:front-ustar}
  \lim_{T \rightarrow \infty} \PP\!\left( \big\langle \1_{(-\infty,0]} , u_T^{*,r}(\fatcdot-cT) \big\rangle < N \right)
  =0
}
and therefore $\PP\big( R_0\big( u_T^{(\psi)} \big) > cT \big) \geq 1-2\delta$ for all $T$ big enough. Then $\lim_{T \rightarrow \infty} \PP\big( R_0\big (u_T^{(\psi)} \big) \geq cT \big) = \1_{(-\infty,B)}(c)$ for $c \neq B$ arbitrary follows.

It thus remains to show \eqref{equ:front-ustar} for $0<c<B$ and $N$ arbitrarily fixed. Let $\Delta = (B-c)/2, 0 < \Delta < B$. By symmetry, $L_0\big( u_T^{*,r} \big)/T \rightarrow -B$ almost surely. A \textit{coupling with two independent processes} at time $T$ yields 
\eqan{
  & \PP\!\left( L_0\big( u_{T+1}^{*,r} \big) \geq (-B+\Delta) T \right) \\
  &\geq \EE\!\left[ \1_{\big\{ \big\langle \1_{(-\infty,-cT]} , u_T^{*,r}(\fatcdot) \big\rangle < N \big\}} 
  \PP_{\1_{(-\infty,-cT]} u_T^{*,r}}( \tau \leq 1 ) \PP_{\1_{[-cT,\infty)} u_T^{*,r}}(L_0(u_1) \geq (-B+\Delta) T) \right]. \nn
}
By \cite[(2.11)]{K2017}, on $\big\{ \big\langle \1_{(-\infty,-cT]} , u_T^{*,r}(\fatcdot) \big\rangle < N \big\}$, $\PP_{\1_{(-\infty,-cT]} u_T^{*,r}}( \tau \leq 1 ) \geq \exp\!\left( \frac{-2\theta N}{1-e^{-\theta}} \right)$. Note that $-B+\Delta=-c-\Delta$ to further get by symmetry and domination, $\PP_{\1_{[-cT,\infty)} u_T^{*,r}}(L_0(u_1) \geq (-B+\Delta) T) \geq \PP(R_0(u_1^{*,l}) \leq \Delta T)$. Hence,
\eqn{
  \PP\big( L_0\big( u_{T+1}^{*,r} \big) \geq (-B+\Delta) T \big) 
  \geq \exp\!\left( \frac{-2\theta N}{1-e^{-\theta}} \right) \PP\big( \big\langle \1_{(-\infty,-cT]} , u_T^{*,r}(\fatcdot) \big\rangle < N \big) \PP\big(R_0\big(u_1^{*,r}\big) \leq \Delta T\big). 
}
As $\lim_{T \rightarrow \infty} \PP\big( L_0\big( u_{T+1}^{*,r} \big) \geq (-B+\Delta) T \big) = 0$ and $\lim_{T \rightarrow \infty} \PP\big(R_0\big(u_1^{*,r}\big) \leq \Delta T\big) = 1$ by \cite[Lemma~4.6]{K2017} and the Markov inequality, \eqref{equ:front-ustar} follows.
\end{proof}
%
%
\begin{corollary}
\label{COR:speed-mittl}
For initial conditions $\psi \in \mittl^R$, $R_0\big( u_T^{(\psi)} \big)/T \rightarrow B$ for $T \rightarrow \infty$ in probability and in $\SL^1$. 
\end{corollary}
%
%
\begin{proof}
As the limit is deterministic, the convergence in distribution (cf. Lemma~\ref{LEM:speed-mittl}) implies convergence in probability (cf. Grimmett and Stirzaker \cite[Theorem~7.2.(4)(a)]{bGS2001}). Use \eqref{equ:lower-bd-mittl} for the negative part and Lemma~\ref{LEM:marker-pos-part-second-moment} and domination for the positive part of $R_0( u_T^{(\psi)} )$ to see that the family $\big\{ R_0\big( u_T^{(\psi)} \big)/T, T \geq 1 \big\}$ is bounded in $\SL^2$ and thus uniformly integrable (cf. \cite[Corollary~6.21]{bK2014}). By \cite[Theorem~6.25 (and Definition~6.2)]{bK2014}, the convergence in $\SL^1$ now follows from the convergence in probability of $R_0\big( u_T^{(\psi)} \big)/T$ in combination with the uniform integrability of this sequence.
\end{proof}
%
%
\begin{corollary}
\label{COR:speed-mittl-wave}
For initial conditions $\psi \in \mittl^R$, 
\eqn{
  R_0\big( u_T^{(\nu^{(\psi)})} \big)/T \rightarrow B \mbox{ almost surely as } T \rightarrow \infty
}
and $\big( 0 \vee R_0\big( u_T^{(\nu^{(\psi)})} \big) \big)/T \rightarrow B$ in $\SL^1$.
\end{corollary}
%
%
\begin{proof}
By \eqref{equ:Tribe-tv-speed}, Lemma~\ref{LEM:EA-B} and Corollary~\ref{COR:speed-mittl},
\eqn{
  R_0\big( u_T^{(\nu^{(\psi)})} \big)/T \rightarrow A^{( \nu^{(\psi)} )} \mbox{ almost surely as } T \rightarrow \infty
}
with $\EE\big[ A^{( \nu^{(\psi)} )} \big]=B$. By Proposition~\ref{PRO:cvg-upper-A-B}, $A^{( \nu^{(\psi)} )} \leq B$ almost surely. Hence, $A^{( \nu^{(\psi)} )} = B$ almost surely and the first claim follows. The second claim follows by Lemma~\ref{LEM:wave-cvg-L1}.
\end{proof}
%
%
%
\subsection{Proof of Theorem~\ref{THM:speeds}}
\label{SUBSEC:Proof-result-2}
%
%
The first claim and \eqref{equ:result2b} follow from Proposition~\ref{PRO:cvg-upper-A-B} and Corollary~\ref{COR:as-limit-upper-wave}. Lemma~\ref{LEM:wave-cvg-L1} yields the $\SL^1$-convergence of the positive part of the right hand side of \eqref{equ:result2b}. The third claim follows from Corollary~\ref{COR:speed-mittl}. The forth and last claim follow from Corollary~\ref{COR:speed-mittl-wave}. This concludes the proof.
%
%
\section{Appendix}
\label{SEC:appendix}
%
%
%
\subsection{Duality}
%

A \textit{self duality relation} in the form of \cite[Subsection~1.2]{HT2004} respectively \cite[(2.1)]{K2017} holds for solutions to \eqref{equ:SPDE}. For solutions with additional annihilation due to competition with a deterministic process $\beta$, see \eqref{equ:SPDE-and-competition} below, a \textit{duality relation} is obtained analogously. Such solutions appear for instance in the context of \textit{monotonicity-couplings}, see \eqref{equ:SPDE-v-monotonicity-coupling}. For existence and uniqueness of solutions to \eqref{equ:SPDE-and-competition} respectively \eqref{equ:SPDE-and-competition-2}, see Theorem~\ref{THM:tribe}.  
%
%
\begin{corollary}
\label{COR:duality}
Let $\theta>0, T>0, \beta \in \SC([0,T],\SC_{tem}^+)$ arbitrarily fixed. Let $v, z$ be independent solutions to 
\eqn{
\lbeq{equ:SPDE-and-competition}
  \frac{\partial v}{\partial t} = \Delta v + ( \theta - v - \beta ) v + \sqrt{v} \dot{W}_1, \qquad v(0)=v_0
}
respectively
\eqn{
\lbeq{equ:SPDE-and-competition-2}
  \frac{\partial z}{\partial t} = \Delta z + \big( \theta - z - \beta_{T-\fatcdot} \big) z + \sqrt{z} \dot{W}_2, \qquad z(0)=z_0
}
for $0 \leq t \leq T$ with $v_0, z_0 \in \SC_{tem}^+$ and $W_1, W_2$ independent white noises. Then we have for $0 \leq s \leq T$,
\eqn{
  \EE\!\left[ e^{-2 \langle v(T) , z(0) \rangle } \right]
  = \EE\!\left[ e^{-2 \langle v(s) , z(T-s) \rangle } \right]
  = \EE\!\left[ e^{-2 \langle v(0) , z(T) \rangle } \right].
}
\end{corollary}
%
%
\begin{proof}
Reason as in \cite[Subsection~1.2]{HT2004}. Let 
\eqn{
  H(f,g) = 2 e^{-2 \langle f,g \rangle} \big( \langle f^2 , g \rangle + \langle f , g^2 \rangle - \langle f , \Delta g \rangle - \theta \langle f , g \rangle \big). 
}
Integration by parts yields $H(f,g)=H(g,f)$. The additional factor of $2$ in the exponent results from the use of different scaling constants in the original SPDEs. Then 
\eqn{
  e^{-2 \langle v_t , g \rangle} - \int_0^t H(v_s,g) ds - 2 \int_0^t e^{-2 \langle v_s , g \rangle } \langle \beta_s v_s , g \rangle ds
}
is a local martingale as well as 
\eqn{
  e^{-2 \langle z_t , f \rangle} - \int_0^t H(z_s,f) ds - 2 \int_0^t e^{-2 \langle z_s , f \rangle } \langle \beta_{T-s} z_s , f \rangle ds.
}
As
\eqan{
  & \frac{d}{ds} \EE\!\left[ e^{-2 \langle v_s , z_{T-s} \rangle } \right] \\
  &= \EE\!\left[ \left( H(v_s,z_{T-s}) + 2 e^{-2 \langle v_s , z_{T-s} \rangle } \langle \beta_s v_s , z_{T-s} \rangle \right) - \left( H(z_{T-s},v_s) + 2 e^{-2 \langle z_{T-s} , v_s \rangle } \langle \beta_s z_{T-s} , v_s \rangle \right) \right] 
  =0, \nn
}
the duality relation follows.
\end{proof}
%
%
\subsection{Travelling waves for the right upper invariant measure}
\label{SUBSEC:trav-waves-for-upper-inv-measure}
%

We extend the construction of travelling wave solutions from solutions with compactly supported initial conditions (cf. \cite{K2017}) respectively Heavyside initial data (cf. \cite{T1996}) to the right upper invariant measure case. As in \cite{K2017}, the right marker is used to center the waves. Recall the set $\mittl^R$ from \eqref{equ:def-mittl}. The constructions extend to initial conditions $u_0 \in \mittl^R$.

Let $\theta_c < \underline{\theta} \leq \theta \leq \overline{\theta}$ and $\nu_T^{*,l}=\nu_T^{*,l}(\theta) \in \CP(\SC_{tem}^+)$ be given by
\eqn{
\lbeq{equ:def-nu-star-lr}
  \nu_T^{*,l}(A) \equiv T^{-1}\!\int_0^T \PP\big( u_s^{*,l}(\fatcdot+R_0(s)) \in A \big) ds,
}
with $\SL\big( u_s^{*,l} \big) = \upsilon_s, s>0$ as in \eqref{equ:laplace-left}--\eqref{equ:coupling-left}. Note that by Corollaries~\ref{COR:e-bd}--\ref{COR:e-bd-small-t}, $R_0\big( u_s^{*,l} \big)$ is almost surely finite and thus $\nu_T^{*,l}$ is well-defined. Then the analogues of the following results of \cite{K2017} hold, where constants only depend on $\theta$ through $\underline{\theta}$ and $\overline{\theta}$. Here it is important to note that the tightness-result of Lemma~\ref{LEM:uniform-tightness} from below is uniform in $\underline{\theta} \leq \theta \leq \overline{\theta}$ as well. Note that in the proof of Lemma~\ref{LEM:analogue-49} we use Corollaries~\ref{COR:e-bd}--\ref{COR:e-bd-small-t} in place of \cite[Lemma~4.8]{K2017}. Also note that the constants in \cite[Lemmas~3.2--3.4]{T1996} hold uniform in $\underline{\theta} \leq \theta \leq \overline{\theta}$, which is easily be deduced using that $p_t^\theta(x) = e^{\theta t} p_t(x)$. Finally, note that the restriction to $N \in \NN$ in \cite[Lemma~3.7]{T1996} and \cite[Chapter~5]{K2017} was only due to the fact that the sequence $\{ \nu_T: T \in \NN \}$ was under consideration rather than the set $\{ \nu_T: T \geq 1 \}$. As we integrate from $1$ to $T$ in the proof of \cite[Lemma~3.7]{T1996} respectively \cite[(5.2)]{K2017}, part of the statements below are only valid for $T>1$. 
%
%
\begin{lemma}(Analogue to \cite[Lemma~4.9]{K2017})
\label{LEM:analogue-49}
If $t>0$, then there exists $C\big( \underline{\theta}, \overline{\theta}, t \big)$ such that for all $a>0, 0< s \leq t$ and $T \geq 1$,
\eqn{
  \PP_{\nu_T^{*,l}}( |R_0(s)| \geq a)
  \leq \frac{C\big( \underline{\theta}, \overline{\theta}, t \big)}{a}.
}
In particular, for $0<t \leq 1$,
\eqn{
  \PP_{\nu_T^{*,l}}( |R_0(s)| \geq a)
  \leq \frac{C\big( \underline{\theta}, \overline{\theta} \big) t^{1/4}}{a}
}
holds.
\end{lemma}
%
%
\begin{lemma}(Extension of \cite[Lemma~5.1]{K2017})
\label{LEM:uniform-tightness}
Let $\theta_c<\underline{\theta} \leq \theta \leq \overline{\theta}$ be arbitrarily fixed. Then the set $\{ \nu_T^{*,l}(\theta): T \geq 1 \}$ is tight. In particular, for every $\epsilon>0$ there exists a compact set $K_\epsilon=K_\epsilon\big( \underline{\theta}, \overline{\theta} \big) \subset \SC_{tem}^+$ such that
\eqn{
\lbeq{equ:uniform-tightness}
  \inf_{\underline{\theta} \leq \theta \leq \overline{\theta}} \big( \nu_T^{*,l}(\theta) \big)(K_\epsilon) \geq 1-\epsilon \quad \mbox{ for all } T>1.
}
\end{lemma}
%
%
\begin{proof}
Let $\lambda>0$ arbitrary. Then there exist $C<\infty, \gamma, \delta >0, \mu<\lambda$ and $A>0$, such that
\eqn{
\lbeq{equ:uniform-tightness-properties}
    \inf_{\underline{\theta} \leq \theta \leq \overline{\theta}} \big( \nu_T^{*,l}(\theta) \big)\big( K(C,\delta,\gamma,\mu) \cap \{ f \in \SC_{tem}^+: \langle f , \phi_1 \rangle \leq A \} \big) \geq 1-\epsilon \quad \mbox{ for all } T>1,
}
where $\phi_1(x) \equiv \exp(-|x|)$ and
\eqn{
\lbeq{equ:def-K}
  K(C,\delta,\gamma,\mu) \equiv \{ f \in \SC_{tem}^+: |f(x)-f(x')| \leq C |x-x'|^\gamma e^{\mu |x|} \mbox{ for all } |x-x'| \leq \delta \}.
}
Indeed, a look at the proof of \cite[Lemma~5.1]{K2017} shows that the sets under consideration, namely $K(C,\delta,\gamma,\overline{\mu})$ and $\{ f: \langle f , \phi_1 \rangle \leq N \}$ are independent of $\theta$. The bounds are derived from previous statements, where constants only depend on $\theta$ through $\underline{\theta}$ and $\overline{\theta}$.

Recall that $K \subset \SC_{tem}^+$ is (relatively) compact if and only if it is (relatively) compact in $\SC_\lambda^+$ for all $\lambda>0$ and that $K(C,\delta,\gamma,\mu) \cap \{ f \in \SC_{tem}^+: \langle f , \phi_1 \rangle \leq A \} \equiv K(\epsilon,\lambda)$ satisfying \eqref{equ:uniform-tightness-properties} is compact in $\SC_\lambda^+$ (cf. \cite[above (1.2)]{T1996}). Now set
\eqn{
  K_\epsilon = \overline{ \bigcap_{n \in \NN} K(\epsilon 2^{-n},1/n)}
}
to conclude the proof of the claim.
\end{proof}
%
%
\begin{lemma}(Analogue to \cite[Lemma~5.2]{K2017})
\label{LEM:A-small-mass-at-front}
Let $t \geq 0$ and $a, \tilde{m}>0$, $0<b \leq 1$ be arbitrarily fixed. Then 
\eqan{
  &\PP_{\nu_T^{*,l}}\!\left( \big\langle u_t(\fatcdot + R_0(u_t)) , \1_{(-2a,\infty)}(\fatcdot) \big\rangle < \tilde{m} \right) \\
   &\leq \left( \left( 1 - \frac{C_1\big( \underline{\theta}, \overline{\theta} \big) b^{1/4}}{a} \right) \vee 0 \right)^{-1} \left\{
  \frac{T+t}{T} \frac{C_2\big( \underline{\theta}, \overline{\theta} \big) b^{1/4}}{a} + \left( 1-e^{-2\overline{\theta} \frac{\tilde{m}}{1-e^{-\underline{\theta} b}}} \right) \right\} \nn
}
for all $T \geq 1$.
\end{lemma}
%
%

Recall from \eqref{equ:def-M} for $d_0, m_0>0$ the definition of
\eqn{
  M(d_0,m_0) = \big\{ f \in \SC_{tem}^+: \mbox{ there exist } -1/2 \leq l_0 < r_0 \leq 0 \mbox{ with } |r_0-l_0|=d_0 \mbox{ such that } f \geq m_0 \1_{[l_0,r_0]} \big\}.
}
%
%
\begin{corollary}
\label{COR:uniform-properties}
Let $\epsilon>0$ arbitrary. Then there exist $d_0=d_0(\epsilon), m_0=m_0(\epsilon)>0$ such that 
\eqn{
  \inf_{\underline{\theta} \leq \theta \leq \overline{\theta}} \big( \nu_T^{*,l}(\theta) \big)\big( M(d_0,m_0) \big) \geq 1-\epsilon 
}
for all $T>1$.
\end{corollary}
%
%
\begin{proof}
In Lemma~\ref{LEM:A-small-mass-at-front} choose $t=0$ and $a=1/4$. Then choose $b$ small enough and then $\tilde{m}$ small enough to obtain 
\eqn{
  \inf_{\underline{\theta} \leq \theta \leq \overline{\theta}} \big( \nu_T^{*,l}(\theta) \big)\!\left( \{ f: \langle f, \1_{(-1/2,0]}(\fatcdot) \rangle \geq \tilde{m} \} \right) \geq 1-\epsilon/2 \quad \mbox{ for all } T \geq 1.
}
By \eqref{equ:uniform-tightness-properties}--\eqref{equ:def-K} for $\lambda=1$ there exist $C<\infty, \gamma, \delta >0, \mu<1$ and $A>0$, such that
\eqn{
    \inf_{\underline{\theta} \leq \theta \leq \overline{\theta}} \big( \nu_T^{*,l}(\theta) \big)\big( K(C,\delta,\gamma,\mu) \cap \{ f \in \SC_{tem}^+: \langle f , \phi_1 \rangle \leq A \} \big) \geq 1-\epsilon/2 \quad \mbox{ for all } T>1.
  }
For deterministic $f_0 \in \SC_{tem}^+$, note that if $f_0 \in \{ \SC_{tem}^+: \langle f, \1_{[-1/2,0]} \rangle \geq \tilde{m} \} \cap \{ f \in \SC_{tem}^+: |f(x)-f(x')| \leq C |x-x'|^\gamma \mbox{ for all } x, x' \in [-1,0], |x-x'| \leq \delta \}$, then there exists $x_0 \in [-1/2,0]$ such that $f(x_0) \geq 2\tilde{m}$. Now use the H\"older-$\gamma$-continuity of $f_0$ around $x_0$ to obtain the existence of $d_0>0$ such that there exist $l_0 \leq x_0 \leq r_0, l_0, r_0 \in [-1/2,0], |r_0-l_0| = d_0$ and $0<m_0<2\tilde{m}$ such that $f(x) \geq m_0$ for all $l_0 \leq x \leq r_0$. 
\end{proof}
%
%
\begin{proposition}(Analogue to \cite[Proposition~1.7]{K2017}) 
\label{PRO:analogue-17}
Let $\nu_{T_n}^{*,l}$ be a subsequence that converges to $\nu^{*,l}$. Then $\nu^{*,l}(\{ f: R_0(f) = 0 \})=1$ and $\PP_{\nu^{*,l}}(u(t) \not \equiv 0)=1$ for all $t \geq 0$.
\end{proposition}
%
%
\begin{theorem}(Analogue to \cite[Theorem~1.6]{K2017})
\label{THM:analogue-16}
Every subsequential limit of the tight set $\{ \nu_T^{*,l}: T \geq 1 \}$ yields a travelling wave solution to equation \eqref{equ:SPDE}. 
\end{theorem}
%
%
\begin{remark}
\label{RMK:analogue-16}
Recall the set $\mittl^R$ from \eqref{equ:def-mittl}. The constructions and statements from above extend to initial conditions $u_0 \in \mittl^R$. Here we use Lemma~\ref{LEM:speed-bound-mittl} instead of \eqref{equ:neg-speed-bound-star}.
\end{remark}
%
%
\subsection{Coupling techniques}
\label{SUBSEC:coupling}
%
%

In what follows we shortly introduce the main coupling techniques and ideas that are used in this article. We start with the \textit{monotonicity-coupling} from \cite[Remark~2.1(i)]{K2017}.  
%
%
\begin{remark}[monotonicity-coupling]
\label{RMK:monotonicity-coupling}
Let $0 < \theta$ and $u_i \in \SC_{tem}^+, i=1,2$ with $u_1(x) \leq u_2(x)$ for all $x \in \RR$. Then there exists a coupling of solutions $u^{(i)}, i=1,2$ to \eqref{equ:SPDE} with initial conditions $u_i, i=1,2$ such that $u^{(1)}(t,x) \leq u^{(2)}(t,x)$ for all $t \geq 0, x \in \RR$ almost surely. For intuition purposes, compare the construction of \cite[Lemma~2.1.7]{MT1994}. The main idea is to write
\eqan{
  \frac{\partial u^{(1)}}{\partial t} &= \Delta u^{(1)} + \big( \theta - u^{(1)} \big) u^{(1)} + \sqrt{u^{(1)}} \dot{W_1}, \qquad u^{(1)}(0)=u_1, \\
  \frac{\partial v}{\partial t} &= \Delta v + \big( \theta - v - 2u^{(1)} \big) v + \sqrt{v} \dot{W_2}, \qquad v(0)=u_2-u_1, \lbeq{equ:SPDE-v-monotonicity-coupling}
}
where $W_1, W_2$ are independent white noises and $u^{(2)} \equiv u^{(1)} + v$ with $v(t,x) \geq 0$ for all $t \geq 0, x \in \RR$ almost surely. $v$ is constructed (conditional on $u^{(1)}$) as a process with annihilation due to competition with $u^{(1)}$. Now recall \cite[(1.8)]{K2017} to note that 
\eqan{
  & \Big\langle \int_0^{\fatcdot} \int \big| u^{(1)}(s,x) \big|^{1/2} \phi(x) dW_1(x,s) + \big| v(s,x) \big|^{1/2}  \phi(x) dW_2(x,s) \Big\rangle_t 
  = \int_0^t \int \big( u^{(1)}(s,x) + v(s,x) \big) \phi^2(x) dx ds \\
  &= \Big\langle \int_0^{\fatcdot} \int \big| u^{(1)}(s,x) + v(s,x) \big|^{1/2}  \phi(x) dW(x,s) \Big\rangle_t 
  = \Big\langle \int_0^{\fatcdot} \int \big| u^{(2)}(s,x) \big|^{1/2}  \phi(x) dW(x,s) \Big\rangle_t \nn
}
for $W$ a white noise appropriately chosen.
\end{remark}
%

In this article we call a \textit{$\theta$-coupling} a coupling in the spirit of \cite[Lemma~2.1.6]{MT1994}. To be more precise, use the techniques of \cite[(2.2)--(2.4)]{K2017} to show the following.
%
%
\begin{remark}[$\theta$-coupling]
Let $0 < \theta_1 < \theta_2$. Let $u_0 \in \SC_{tem}^+$. Then there exists a coupling of solutions $u^{(i)}, i=1,2$ to \eqref{equ:SPDE} with common initial condition $u_0$ but different parameters $\theta_1$ respectively $\theta_2$ such that $u^{(1)}(t,x) \leq u^{(2)}(t,x)$ for all $t \geq 0, x \in \RR$ almost surely. The main idea is to write
\eqan{
  \frac{\partial u^{(1)}}{\partial t} &= \Delta u^{(1)} + \big( \theta_1 - u^{(1)} \big) u^{(1)} + \sqrt{u^{(1)}} \dot{W_1}, \qquad u^{(1)}(0)=u_0, \\
  \frac{\partial v}{\partial t} &= \Delta v + (\theta_2-\theta_1) u^{(1)} + \big( \theta_2 - v - 2u^{(1)} \big) v + \sqrt{v} \dot{W_2}, \qquad v(0)=0, \lbeq{equ:SPDE-v-theta-coupling}
}
where $W_1, W_2$ are independent white noises and $u^{(2)} \equiv u^{(1)} + v$ with $v(t,x) \geq 0$ for all $t \geq 0, x \in \RR$ almost surely. $v$ is constructed (conditional on $u^{(1)}$) as a process with annihilation due to competition with $u^{(1)}$ and an immigration-term $(\theta_2-\theta_1) u^{(1)}$.
\end{remark}
%

In what follows we call a \textit{coupling with two independent processes} a coupling in the spirit of \cite[Lemma~2.1.7]{MT1994}. To be more precise, use the techniques of \cite[(2.2)--(2.4)]{K2017} to show the following.
%
%
\begin{remark}[coupling with two independent processes]
\label{RMK:coupling-independent}
Let $0<\theta$. Let $u_1, u_2 \in \SC_{tem}^+$ and $u_0 \equiv u_1+u_2$. Then there exists a coupling of solutions $u^{(i)}, i=0,1,2$ to \eqref{equ:SPDE} with initial conditions $u_i, i=0,1,2$ such that $u^{(1)}$ and $u^{(2)}$ are independent and $u^{(0)}(t,x) \leq u^{(1)}(t,x) + u^{(2)}(t,x)$ for all $t \geq 0, x \in \RR$ almost surely. The main idea is to write
\eqan{
  \frac{\partial u^{(1)}}{\partial t} &= \Delta u^{(1)} + \big( \theta - u^{(1)} \big) u^{(1)} + \sqrt{u^{(1)}} \dot{W_1}, \qquad u^{(1)}(0)=u_1, \\
  \frac{\partial v}{\partial t} &= \Delta v + \big( \theta - v - 2u^{(1)} \big) v + \sqrt{v} \dot{W_2}, \qquad v(0)=u_2, \nn\\
  \frac{\partial u^{(2)}}{\partial t} &= \Delta u^{(2)} + \big( \theta - u^{(2)} \big) u^{(2)} + \sqrt{u^{(2)}} \dot{W_2}, \qquad u^{(2)}(0)=u_2, \nn
}
where $W_1, W_2$ are independent white noises and $u^{(0)} \equiv u^{(1)} + v$ with $v(t,x) \leq u^{(2)}(t,x)$ for all $t \geq 0, x \in \RR$ almost surely. $v$ is constructed (conditional on $u^{(1)}$) as a process with annihilation due to competition with $u^{(1)}$ contrary to $u^{(2)}$, where no annihilation takes place. The independence of $u^{(1)}$ and $u^{(2)}$ follows from the independence of the white noises $W_1, W_2$.
\end{remark}
%

An \textit{immigration-coupling} is constructed similarly to a \textit{$\theta$-coupling}, where the immigration-term only depends on an outside source.
%
%
\begin{remark}[immigration-coupling]
Let $\alpha_1, \alpha_2-\alpha_1 \in \SC([0,\infty),\SC_{tem}^+)$. Let $u_0 \in \SC_{tem}^+$. Then there exists a coupling of solutions $u^{(i)}, i=1,2$ solving
\eqn{
  \frac{\partial u^{(i)}}{\partial t} = \Delta u^{(i)} + \alpha_i + \big( \theta - u^{(i)} \big) u^{(i)} + \sqrt{u^{(i)}} \dot{W_i}, \qquad u^{(i)}(0)=u_0, \qquad i=1,2
}
with $W_1, W_2$ two independent white noises, such that $u^{(1)}(t,x) \leq u^{(2)}(t,x)$ for all $t \geq 0, x \in \RR$ almost surely. The main idea is to write $u^{(2)} \equiv u^{(1)} + v$ with $v \geq 0$ satisfying
\eqn{
  \frac{\partial v}{\partial t} = \Delta v + (\alpha_2-\alpha_1) + \big( \theta - v - 2u^{(1)} \big) v + \sqrt{v} \dot{W_2}, \qquad v(0)=0.
}
$v$ is constructed (conditional on $u^{(1)}$) as a process with annihilation due to competition with $u^{(1)}$ and an immigration-term $\alpha_2-\alpha_1$.
\end{remark}
%
%
Note that conditional on $u^{(1)} \in \SC([0,\infty),\SC_{tem}^+)$ all the processes $(v(t))_{t \geq 0}$ fit into the framework of \cite[(2.4)]{T1996} and are as such non-negative.
%
%
The final coupling we present is of a different flavor. It is based on the approximation of solutions to \eqref{equ:SPDE} with initial conditions $u_0 \in \SC_{tem}^+$ by means of densities of rescaled long-range contact processes, see \cite[Theorem~1]{MT1995} for the convergence result. Note that the parameter $\theta_c$ in \cite{MT1995} denotes an arbitrary $\theta>0$ and does not have any relation to the critical parameter $\theta_c$ of the present article. Changes in constants are 

We use the construction of an approximating particle system $(\xi_t^n(f_0))_{t \geq 0}$ for $n \in \NN$ resulting in a solution to \eqref{equ:SPDE} with initial condition $f_0 \in \SC_{tem}^+$ from \cite{MT1995}. The dynamics are modeled by means of i.i.d. Poisson processes given at the beginning of Section~2. Their rates depend in a monotone way on the parameter $\theta$. Initial conditions $f_0$ get approximated by approximate densities $A_c(\xi_0^n(f_0)), n \in \NN$, compare the definition preceding Theorem~1. The approximate densities $(A_c(\xi_t^n(f_0))_{t \geq 0}$ converge to a solution to \eqref{equ:SPDE} with initial condition $f_0$. 
%
%

For the next lemma, recall the definition of $\upsilon_T = \upsilon_T(\theta)$ from \cite[Remark~2.8]{K2017}. Note in particular the use of the non-decreasing sequence $\zeta_N \in \SC_{tem}^+, N \in \NN$.
%
%
\begin{lemma}[$\theta$-$*$-coupling]
\label{LEM:star-coupling}
Let $0<\theta_1<\theta_2$ and $T>0$ be arbitrarily fixed. There exists a coupling of two processes $\big( u_{T+t}^{*,l}(\theta_i) \big)_{t \geq 0}$, $i=1,2$  such that $\SL\big( \big( u_{T+t}^{*,l}(\theta_i) \big)_{t \geq 0} \big) = \PP_{\upsilon_T(\theta_i)}$, $i=1,2$. Moreover,
\eqn{
\lbeq{equ:star-coupling}
  \big( u_{T+t}^{*,l}(\theta_1) \big)(x) \leq \big( u_{T+t}^{*,l}(\theta_2) \big)(x) \mbox{ for all } x \in \RR, t \geq 0 \mbox{ a.s.}
}
This result also holds for a finite number of $0<\theta_1<\cdots<\theta_m, m \in \NN$.
\end{lemma}
%
%
This coupling relies on two properties of the processes involved. Firstly, we use the monotonicity of the respective solutions resulting from $\theta_1<\theta_2$ for each initial condition $\zeta_N, N \in \NN$, secondly, for $\theta_i$ fixed, we use the construction of $\big( u_{T+t}^{*,l}(\theta_i) \big)_{t \geq 0}$ by means of a non-decreasing sequence $\big( u_{T+t}^{(\zeta_N)}(\theta_i) \big)_{t \geq 0}, N \in \NN$ as in \cite[Remark~2.8]{K2017}. Unfortunately, we could not make the constructions from above work to integrate these two steps into one. Thus we had to make use of the approximation by discrete particle systems, where at least the motivation for the veracity of the above result should be easily accessible to the reader.
%
%
\begin{proof}
The \textit{dynamics} of the $n^{th}$ approximation for $\theta_i, i=1,2$ use the same set of i.i.d. Poisson processes for death-events. For birth-events, consider i.i.d. Poisson processes 
\eqan{
  & \big( P_t(x,y): x, y \in n^{-2} \ZZ, x \mbox{ neigbor of } y \big) \ \mbox{ with rate } (2c_1n^{3/2})^{-1}(n+\theta_1), \\
  & \big( Q_t(x,y): x, y \in n^{-2} \ZZ, x \mbox{ neighbor of } y \big) \ \mbox{ with rate } (2c_1n^{3/2})^{-1}(\theta_2-\theta_1), \nn
}
where $c_1(n) \rightarrow 1$ as $n \rightarrow \infty$. For the $\theta_1$-system, at a jump of $P_t(x,y)$, if the site $x$ is occupied, there is a birth and the site $y$, if vacant, becomes occupied (cf. beginning of \cite[Section~2]{MT1995}). In our coupling, for the $\theta_2$-system, at a jump of $P_t(x,y)$ or $Q_t(x,y)$ the same holds. Note that $\big( P_t(x,y)+Q_t(x,y): x,y \in n^{-2} \ZZ, x \mbox{ neighbor of } y \big)$ is a family of i.i.d. Poisson processes with rate $(2c_1n^{3/2})^{-1}(n+\theta_2)$. As a result, given the same initial configurations, the $\theta_2$-system dominates the $\theta_1$-system. 

Additionally, we construct a set of \textit{initial conditions} $\big( \xi_0^n(\zeta_N): N \in \NN \big)$ of the $n^{th}$ approximating particle systems as follows below. They are the same for the $\theta_1$- and $\theta_2$-system. After linear interpolation in space, $A_c(\xi_0^n(\zeta_N))$ converges in $\SC_{tem}^+$ to $\zeta_N$ for $n \rightarrow \infty$ for all $N \in \NN$ and (use that the sequence $(\zeta_N)_{N \in \NN}$ is non-decreasing), $\xi_0^n(\zeta_{N_1}) \leq \xi_0^n(\zeta_{N_2})$ for $N_1 \leq N_2$. By \cite[Theorem~1]{MT1995}, the approximating densities $\big( A_c(\xi_t^n(\zeta_N))(\theta_i) \big)_{t \geq 0}, i=1,2$ converge in distribution for $n \rightarrow \infty$ to \textit{continuous} solutions $\big( u_t^{(\zeta_N)}(\theta_i) \big)_{t \geq 0}, i=1,2$ of \eqref{equ:SPDE} with initial conditions $\zeta_N \in \SC_{tem}^+$. By construction, 
\eqn{
\lbeq{equ:mon-coup-init-n}
  A_c(\xi_t^n(\zeta_{N_1}))(\theta_i) \leq A_c(\xi_t^n(\zeta_{N_2}))(\theta_j)
}
for all $t \geq 0, N_1 \leq N_2, \theta_i \leq \theta_j, i,j \in \{1,2\}$. 

For $n \in \NN$ fixed, use the following coupling to obtain $n^{th}$-approximations $\big( \xi_0^n(\zeta_N): N \in \NN \big)$ for a family of initial conditions $(\zeta_N)_{N \in \NN}$ as in \cite[Remark~2.8]{K2017}. Assume without loss of generality that for all $N \in \NN$, $\zeta_N \in \SC_{tem}^+$ is a bounded continuously differentiable function with bounded first derivatives. To construct the initial conditions of the $n^{th}$ approximating particle system recall that each site $z \in n^{-2} \ZZ$ has $2c_1n^{3/2}$ neighbors (including $z$) and for $f_0 \in \SC_{tem}^+$, $A_c(\xi_0^n(f_0))(z) = (2c_1 n^{1/2})^{-1} \sum_{y \; neighbor\; of\; z} \big( \xi_0^n(f_0) \big)(y)$. For $z \in n^{-2} \ZZ$, let
\eqn{
\lbeq{equ:approx-f0}
  \big( \xi_0^n(f_0) \big)(z) 
  = \begin{cases}
       1, & \exists k \in \ZZ: z \in \Big\{ \frac{k \cdot 2 c_1 n^{3/2}}{n^2}, \frac{k \cdot 2 c_1 n^{3/2}+1}{n^2}, \ldots, \frac{k \cdot 2 c_1 n^{3/2} + \lfloor 2 c_1 n^{1/2} f_0(k c_1 n^{-1/2}) \rfloor}{n^2} \Big\}, \cr
       0, & \mbox{ otherwise.}
     \end{cases}
}
Then $A_c(\xi_0^n(\zeta_N))$, after linear interpolation in space, converges in $\SC_{tem}^+$ to $\zeta_N$ for $n \rightarrow \infty$ for all $N \in \NN$ and by construction, \eqref{equ:mon-coup-init-n} is fulfilled. 

Assume without loss of generality that $\tilde{d}$ is such that $\big( \SD([0,\infty),\SC_{tem}^+),\tilde{d} \big)$ is a Polish space (recall $d$ from below \cite[(1.6)]{K2017} and cf. Ethier and Kurtz \cite[Theorem~III.5.6]{bEK2005}). We have
\eqan{
  &\big( A_c(\xi_{\fatcdot}^n(\zeta_N))(\theta_1), A_c(\xi_{\fatcdot}^n(\zeta_N))(\theta_2), A_c(\xi_{\fatcdot}^n(\zeta_N))(\theta_2) - A_c(\xi_{\fatcdot}^n(\zeta_N))(\theta_1), \\
  &\quad A_c(\xi_{\fatcdot}^n(\zeta_{N+1}))(\theta_2) - A_c(\xi_{\fatcdot}^n(\zeta_N))(\theta_2), A_c(\xi_{\fatcdot}^n(\zeta_{N+1}))(\theta_1) - A_c(\xi_{\fatcdot}^n(\zeta_N))(\theta_1) \big)_{N \in \NN} \nn\\
  & \equiv \big( A_n(N,\theta_1), A_n(N,\theta_2), A_n(N,\Delta\theta), A_n(\Delta N,\theta_1), A_n(\Delta N,\theta_2) \big)_{N \in \NN} \nn\\
  & \equiv A_n \in X \equiv \Big( \big( \SD([0,\infty),\SC_{tem}^+) \big)^5 \Big)^\NN. \nn
}
Then $X$ is a Polish space as well if we equip it with the metric $\rho(\bar{f},\bar{g}) \equiv \sum_{i \in \NN} 2^{-i} \big( \sum_{j=1,2,3,4,5} \tilde{d}(f_{i1},g_{i1}) \big) \wedge 1$ where $\bar{f} = (f_{ij})_{i \in \NN, j=1,2,3,4,5}, \bar{g} = (g_{ij})_{i \in \NN, j =1,2,3,4,5} \in X$, $f_{ij}, g_{ij} \in \SD([0,\infty),\SC_{tem}^+)$. Reason as in Jacod and Shiryaev \cite[Corollary~VI.3.33]{bJS2003}, to see that the convergence in distribution of the sequences $\big( A_c(\xi_{\fatcdot}^n(\zeta_N))(\theta_j) \big)_{n \in \NN}$ for $j \in \{1,2\}, N \in \NN$ fixed and $n \rightarrow \infty$ to a \textit{continuous} (in $t$) limit implies the convergence of 
\eqn{
  \Big( A_n(N,\theta_1), A_n(N,\theta_2), A_n(N,\Delta\theta_1), A_n(\Delta N,\theta_1), A_n(\Delta N,\theta_2) \Big)_{n \in \NN}
}
to 
\eqn{
  \Big( u_{\fatcdot}^{(\zeta_N)}(\theta_1), u_{\fatcdot}^{(\zeta_N)}(\theta_2), u_{\fatcdot}^{(\zeta_N)}(\theta_2)-u_{\fatcdot}^{(\zeta_N)}(\theta_1), u_{\fatcdot}^{(\zeta_{N+1})}(\theta_1)-u_{\fatcdot}^{(\zeta_N)}(\theta_1), u_{\fatcdot}^{(\zeta_{N+1})}(\theta_2)-u_{\fatcdot}^{(\zeta_N)}(\theta_2) \Big).
  } 
By the definition of $\rho(\bar{f},\bar{g})$, we can choose a subsequence such that $(A_{n_k})_{k \in \NN}$ converges in $X$. Note in particular that the marginal distributions of every subsequential limit are given by their respective one-dimensional limits. 

Fix this convergent subsequence. Now apply Skorokhod's theorem (cf. \cite[Theorem~III.1.8]{bEK2005}) to obtain that after possibly changing to another probability space, this convergence becomes almost sure convergence, that is 
\eqan{
\lbeq{equ:Xleq}
  & \big( A_{n_k}(N,\theta_1), A_{n_k}(N,\theta_2), A_{n_k}(N,\Delta\theta_1), A_{n_k}(\Delta N,\theta_1), A_{n_k}(\Delta N,\theta_2) \big)_{N \in \NN} \\
  & \ \rightarrow \big( (u_{\fatcdot}^{(\zeta_N)}(\theta_1), u_{\fatcdot}^{(\zeta_N)}(\theta_2), u_{\fatcdot}^{(\zeta_N)}(\theta_2)-u_{\fatcdot}^{(\zeta_N)}(\theta_1), \nn\\
  & \qquad\qquad u_{\fatcdot}^{(\zeta_{N+1})}(\theta_1)-u_{\fatcdot}^{(\zeta_N)}(\theta_1), u_{\fatcdot}^{(\zeta_{N+1})}(\theta_2)-u_{\fatcdot}^{(\zeta_N)}(\theta_2) \big)_{N \in \NN} \in X \  \mbox{ a.s. for } k \rightarrow \infty, \nn
}
where $u^{(\zeta_N)}(\theta_i)$ solves \eqref{equ:SPDE} with initial condition $\zeta_N$ and parameter $\theta_i$ and
\eqn{
\lbeq{equ:mon-coup-init}
  u^{(\zeta_{N_1})}(\theta_{j_1}) \leq u^{(\zeta_{N_2})}(\theta_{j_2}) \mbox{ for all } N_1 \leq N_2, j_1 \leq j_2 \mbox{ a.s.}
}
by \eqref{equ:Xleq} in combination with the definition of $X$.

Fix $T>0$. Let $u^{*,l}_{T+t}(\theta_i) = \uparrow \lim_{N \rightarrow \infty} u^{(\zeta_N)}_{T+t}(\theta_i), i=1,2, t \geq 0$. Reason as in \cite[Corollary~2.6]{K2017} to conclude that $\SL((u^{*,l}_{T+t}(\theta_i))_{t \geq 0} = \PP_{\upsilon_T}$. From \eqref{equ:mon-coup-init}, \eqref{equ:star-coupling} follows.
\end{proof}
%

%
{\bf Acknowledgements.}
This work was funded by the Deutsche Forschungsgemeinschaft (DFG, German Research Foundation) - Projectnumber 393092071.

\bibliography{bib}
\bibliographystyle{plain}


\end{document}